\documentclass[11pt,reqno]{amsart}
\usepackage{amssymb,amsmath,amsthm,latexsym,color,amsbsy}
\usepackage{pifont}
\usepackage{verbatim}
\usepackage{amsfonts}
\usepackage{hyperref}

\font\TenEns=msbm10 \font\SevenEns=msbm7 \font\FiveEns=msbm5
\newfam\Ensfam

\textfont\Ensfam=\TenEns \scriptfont\Ensfam=\SevenEns
\scriptscriptfont\Ensfam=\FiveEns

\newcommand{\ba}{{\bar a}}
\newcommand{\bpd}{{\bar \psi^2}}
\newcommand{\bpq}{{\bar \psi^4}}
\newcommand{\bpsi}{{\bar \psi}}
\newcommand{\bpt}{{\bar \psi^3}}
\newcommand{\bt}{{\bar T}}
\newcommand{\bv}{{\bar v}}

\newcommand{\f}{\frac}
\newcommand{\N}{\mathbb{N}}
\newcommand{\R}{\mathbb{R}}

\newcommand{{\vpz}}{{\varphi^0}}
\newcommand{\vq}{v}

\newtheorem{cor}{Corollary}[section]
\newtheorem{defi}[cor]{Definition}

\newtheorem{lem}[cor]{Lemma}
\newtheorem{prop}[cor]{Proposition}
\newtheorem{propo}{Proposition}
\newtheorem{theorem}[propo]{Theorem}
\newtheorem{coro}[propo]{Corollary}
\newtheorem{rem}[cor]{Remark}
\title[Blow-up profile for the heat equation with a nonlinear
gradient term]{Existence of a Stable Blow-up profile for the nonlinear heat equation with a critical power nonlinear gradient term}
\author[S. Tayachi and H. Zaag]{Slim Tayachi and Hatem
Zaag$^1$}
\address{\hspace{-0.5cm} Universit\'e de Tunis El Manar, Facult\'e des Sciences de Tunis, D\'epartement de
Math\'ematiques, Laboratoire  \'Equations aux D\'eriv\'ees
Partielles LR03ES04,  2092 Tunis, Tunisie.  e-mail: {\tt
slim.tayachi@fst.rnu.tn} \vspace{0.5cm}\newline Universit\'e Paris
13, Sorbonne Paris cit\'e, Institut Galil\'ee, CNRS UMR 7539 LAGA,
99 Avenue Jean-Baptiste Cl\'ement 93430 Villetaneuse, France.
e-mail: {\tt Hatem.Zaag@univ-paris13.fr}} 
\date{\today}
\thanks{$^1$ This author is supported by the ERC Advanced Grant no. 291214, BLOWDISOL and by the ANR project ANA\'E ref. ANR-13-BS01-0010-03}
\begin{document}

\begin{abstract}
We consider the nonlinear heat equation with  a nonlinear gradient term:
$\partial_t u =\Delta u+\mu|\nabla u|^q+|u|^{p-1}u,\; \mu>0,\;
q=2p/(p+1),\; p>3,\; t\in (0,T),\; x\in \R^N.$ We construct a
solution which blows up in finite time $T>0.$ We
also give a sharp description of its blow-up profile and show that it is stable with respect to perturbations in initial data. The proof
relies on the reduction of the problem to a finite dimensional one,
and uses the index theory to conclude.  The blow-up profile does not
scale as $(T-t)^{1/2}|\log(T-t)|^{1/2},$ like in the standard nonlinear heat equation, i.e. $\mu=0,$ but as
$(T-t)^{1/2}|\log(T-t)|^{\beta}$ with
$\beta=(p+1)/[2(p-1)]>1/2.$  We also show that $u$ and $\nabla u$ blow up simultaneously and at a single point, and give the final
profile. In particular, the final profile is more singular than the case of the standard nonlinear heat equation.
\end{abstract}

\maketitle

\section{Introduction and statement of the results}
\setcounter{equation}{0}
 We consider the problem
\begin{align}\label{eq:uequation}
\partial_t u & = \Delta u +\mu|\nabla u|^q+|u|^{p-1}u, \\
& u(\cdot,0) = u_0\in W^{1,\infty}(\mathbb{R}^N),\notag
\end{align}
 where $u=u(x,t)\in \R,\; t\in [0,T),\; x\in \R^N$, and the parameters $\mu,\;
 p$ and $q$ are such that
\begin{align}\label{condition:h}
     \mu>0,\; p>3,~  q=\frac{2p}{p+1}.
     \end{align}

Equation \eqref{eq:uequation} is wellposed in  $ W^{1,\infty}(\mathbb{R}^N)$.
See \cite[Proposition 4.1, p. 18]{AW92}, \cite[Corollary 3, p. 67]{SW} and \cite{SW1} for $\mu<0.$
The proof for the case $\mu>0$ follows similarly, using a fixed point argument.   Precisely, there exists a unique maximal solution on $[0,T)$ of \eqref{eq:uequation} with $T\leq \infty$.
Moreover, according to \cite{AW92}, since $p>1$ and $1<q<2,$ nonglobal solutions, i.e. solutions with $T<\infty,$ blow up in the $L^\infty$-norm.

The value $q=2p/(p+1)$ is a critical exponent for the equation \eqref{eq:uequation} for different reasons. One reason is that, when $q = 2p/(p+1),$
equation \eqref{eq:uequation} is  invariant under the transformation: $u_\lambda(t, x) = \lambda^{2/(p-1)}u(\lambda^2t, \lambda x),$ as
for the equation without the gradient term, that is $\mu=0$.  Another reason is related to the behavior of global or blowing up solutions.
Indeed, it is proved in \cite{SnTW,SnT} that when $\mu\not=0,$ the large time behavior of global solutions of \eqref{eq:uequation} depends on the position of $q$ with respect to $2p/(p+1).$
On the other hand, it is pointed out in some previous works that the
position of $q$ with respect to  $2p/(p+1)$ has an influence on the
behavior of blowing up solutions for $\mu<0$. See
\cite{CWbngsj89,CFQ,Sejde2001,{Snepp2005},ST,STWesbugiumj96,SW} and references therein.

Many works has been devoted to the blow-up profiles for  the
equation without the gradient term, i.e.  \eqref{eq:uequation} with
$\mu=0.$ See \cite{BK,MZsbupdmj97, BergerKhon,HViph93} and
references therein.  But, few results are known for the equation
with $\mu\not=0.$ For self-similar blow-up profiles, see
\cite{STWesbugiumj96} when $q=2p/(p+1),\; \mu<0$ and  \cite{GV1,
GV2} when $q=2,\; \mu>0.$
A blow-up profile is derived in  \cite{EZ} for the equation \eqref{eq:uequation}  when $q<2p/(p+1)$.
This blow-up profile is the same as for the standard nonlinear heat
equation obtained in \cite{MZsbupdmj97}. In this paper, we are
interested in the construction of a blowing up solution, with a
prescribed blow-up profile for the equation \eqref{eq:uequation},
when $q=2p/(p+1)$. We have obtained the following result.

\begin{theorem}[Blow-up profile for Equation \eqref{eq:uequation}]
\label{th1}
Let $\mu>0$ and $p,\; q$ be two real numbers such that
\begin{equation}
q=\frac{2p}{p+1}\;\mbox{and}\; p>3.
\end{equation}
Then, for any $\varepsilon>0,$ Equation \eqref{eq:uequation} has a
solution $u(x,t)$ such that $u$ and $\nabla u$ blow up in finite
time $T>0$ simultaneously. Moreover:
\begin{itemize}
\item[(i)] For all $t\in [0,T),$
\begin{eqnarray}
\nonumber
&&\left\|(T-t)^{{1\over p-1}}u(y\sqrt{T-t},t)-\vpz\left(\frac{y}{|\log(T-t)|^\beta}\right)
\right\|_{W^{1,\infty}(\R^N)}\\ \label{bupprof1}&&\hspace{2cm}\leq {C\over 1+|\log(T-t)|^{\min\left({2\over p-1},{p-3\over 2(p-1)}\right)-\varepsilon}},
\end{eqnarray}
where
\begin{equation}
 \label{145}
 \vpz(z)=\left(p-1+b|z|^2\right)^{-{1\over p-1}},\; z\in \R^N,
\end{equation}
\begin{equation}
\label{bbupprof1}
\beta={p+1\over 2(p-1)},\; \quad \; b={1\over 2}(p-1)^{{p-2\over p-1}}\left((4\pi)^{N\over 2}(p+1)^2N\over p
\int_{\R^N}|y|^qe^{-{|y|^2/4}}dy\right)^{{p+1\over p-1}}\mu^{-{p+1\over p-1}}>0.\;
\end{equation}
and $C$ is a positive constant.
\item[(ii)] The functions $u$ and $\nabla u$ blow up at the origin and only there.
\item[(iii)] For all $x\not=0,$ $u(x,t)\to u^*(x)$ as $t\to T$ in $C^1\left(\frac 1R <|x|<R\right)$ for any $R>0$,
with
\begin{equation}
\label{bupprof2}
u^*(x)\sim \left({b|x|^2\over \left[2\left|\log|x|\right|\right]^{{p+1\over p-1}}}\right)^{-{1\over p-1}},\; \mbox{ as } \; x\to\; 0,
\end{equation}
and for $|x|$ small,
\begin{equation}
\label{bupprof2na}
|\nabla u^*(x)|\leq  C {|x|^{-{p+1\over
p-1}}\over |\log|x||^{{1-3p\over (p-1)^2}-\epsilon}},\;
\mbox{ if }\; 3<p\leq 7,
\end{equation}

\begin{equation}
\label{bupprof2nb}
|\nabla u^*(x)|\leq  C {|x|^{-{p+1\over
p-1}}\over |\log|x||^{{-p^2+2p-5\over 2(p-1)^2}-\epsilon}},\;
\mbox{ if }\; p>7,
\end{equation}
where $C$ is a positive constant.
\end{itemize}
\end{theorem}

\begin{rem}{\rm From the previous Theorem, we have
for all $t\in [0,T),$
\begin{eqnarray*}
&&\left\|(T-t)^{{1/(p-1)}}u(x,t)-\left(p-1+\frac{b|x|^2}{(T-t)|\log(T-t)|^{{(p+1)/(p-1)}}}\right)^{-{1/(p-1)}}
\right\|_{L^{\infty}(\R^N)}\\ &&\hspace{2cm}\leq {C\over 1+|\log(T-t)|^{\min\left({2\over p-1},{p-3\over 2(p-1)}\right)-\varepsilon}},
\end{eqnarray*}
where $C$ is a positive constant.}
\end{rem}

\begin{rem}{\rm
 To have a flavor of the appearance of the particular shape for our profile $\vpz$ in \eqref{145} together with the scaling factor $\beta$ and the parameter $b$ in \eqref{bbupprof1},
 see the formal approach in Section 2 below. However, we would like to emphasize the fact that
 in the actual proof, those particular values are crucially needed in various algebraic identities. See the Remark \ref{4losange} following Lemma \ref{p52} below and Proposition \ref{prop:projet:q} below.}
\end{rem}

\begin{rem}{\rm
 The initial data giving rise to the constructed solution is given in Proposition \ref{initial-data} below.}
\end{rem}

\begin{rem}{\rm
Note that the  solution constructed in the above theorem does not exist in the case of the standard nonlinear heat equation, i.e. when $\mu=0$ in \eqref{eq:uequation}. Indeed, our solution has a profile depending on the reduced
variable
$$z={x\over \sqrt{T-t}\left|\log(T-t)\right|^\beta}$$ whereas, we know from the results in \cite{HViph93,Vtams93} that the blow-up profiles in the case $\mu=0$ depend on the reduced variables
$$z={x\over \sqrt{T-t}\left|\log(T-t)\right|^{{1\over 2}}}\; \mbox{or}\; z={x\over (T-t)^{{1\over 2m}}},\; \mbox{where}\; m\geq 2\; \mbox{is an integer}.$$}
\end{rem}

\begin{rem}{\rm We conjecture that identity \eqref{bupprof2} holds also after differentiation. Unfortunately, we have been able to derive only the weaker results given in \eqref{bupprof2na}  and \eqref{bupprof2nb}.}
\end{rem}

\begin{rem}{\rm
In the case $\mu=0,$ the final profile of the standard nonlinear heat equation is given by
 \begin{equation}
  u_0^*(x)\sim C\left({|x|^2\over \left|\log|x|\right|}\right)^{-{1\over p-1}},\; \mbox{ as } \; x\to\; 0,
\end{equation}
where $C$ is a positive constant (see \cite{Zihp98}). In our case, ( $\mu>0$), the final profile is given by \eqref{bupprof2}. Let us denote it by $u_\mu^*.$ Since $1<q<2,$  $u_\mu^*$ is more singular than $u_0^*$
for $x$ close to $0,$  in the sense that
$$u_0^*(x) \ll  u_\mu^*(x),\mbox{ as } x\to 0.$$
This shows the effect of the forcing gradient term in equation \eqref{eq:uequation} with $\mu>0$ in the equation.\\
On the other hand,  heuristically, $u_\mu^*$ and  $u_0^*$ have the same singularity near $x=0$ for the limiting case $q=2$ (that is when $p\to \infty$).\\
Let us note that in
\cite{STWesbugiumj96} and with $\mu<0$ in the equation \eqref{eq:uequation},  that is with a damping gradient term in the equation,  a less singular finale profile is obtained: $v_{\mu}^*(x) \sim |x|^{-{2\over p-1}}$ as $x\to 0$.}
\end{rem}
\begin{rem}{\rm We strongly believe that our strategy breaks down when $1<p\leq 3.$ See the remark following Lemma \ref{p52} below.}
\end{rem}

As a consequence of our techniques, we show the stability of the constructed solution, with respect to perturbations in initial data. More precisely, we have obtained the following result.

\begin{theorem}[Stability of the blow-up profile \eqref{bupprof1}]
\label{th2}
Let $\mu>0$ and $p,\; q$ be two real numbers such that
\begin{equation}
p>3\;\mbox{ and } q={2p\over p+1}.
\end{equation}
Let $\hat{u}$ be the solution of \eqref{eq:uequation} given by Theorem \ref{th1} with initial data $\hat{u}_0$ and which blows up at time
$\hat{T}.$ Then, there exists a neighborhood $V_0$ of $\hat{u}_0$ in $W^{1,\infty}(\R^N)$ such that for any $u_0\in V_0,$
Equation  \eqref{eq:uequation} has a unique solution $u$ with initial data $u_0,$ $u$ blows up in finite time $T(u_0)$ and at a single point $a(u_0).$
Moreover, items (i)-(iii) of Theorem \ref{th1}, are satisfied
by $u(x-a,t)$ and
$$T(u_0)\to \hat{T},\; a(u_0)\to 0,\; \mbox{ as }\; u_0\to \hat{u_0}\; \mbox{ in } W^{1,\infty}(\R^N).$$
\end{theorem}
\begin{rem}
{\rm In fact, we have a stronger version of the stability theorem, valid for single-point blow-up solutions of equation \eqref{eq:uequation}, enjoying the profile \eqref{145} only for a sequence of time. See Theorem \ref{th2}' in page \pageref{th2'} below.}
\end{rem}

Let us give an idea of the methods used to prove the results. We construct the blow-up solution with the profile in Theorem \ref{th1}, by following the methods of  \cite{BK} and \cite{MZsbupdmj97}, though we are far from a simple adaptation, since the gradient term needs genuine new ideas as we explain shortly below.
This kind of methods has been applied for various nonlinear evolution equations.
 For hyperbolic equations, it has been successfully used for the construction of multi-solitons for the semilinear wave equation in one space dimension (see \cite{CZ}). For parabolic equations, it has been used in \cite{MZ08} and \cite{Zihp98}
 for the complex Ginzburg-Landau equation with no gradient structure. See also the cases of the wave maps in \cite{RD}, the Schr\"odinguer maps in \cite{MRR}, the critical harmonic heat follow in \cite{RS1}, the two-dimensional Keller-Segel equation in \cite{RS2} and the nonlinear heat equation involving a subcritical nonlinear gradient term in \cite{EZ}.
 Recently, this method has been applied for a non variational parabolic system in \cite{NZ} and for a logarithmically perturbed nonlinear heat equation in \cite{NguyenZ}.

 Unlike in the subcritical case in \cite{EZ}, the gradient term in the critical case induces substantial changes in the blow-up profile as we pointed-out in the comments following Theorem \ref{th1}. Accordingly, its control requires special arguments.
So, working in the framework
 of \cite{MZsbupdmj97}, some crucial modifications are needed. In particular, we have to overcome the following challenges:
 \begin{itemize}
  \item[-] The prescribed profile is not known and not  obvious to find. See Section \ref{secform} for a formal approach to justify such a profile, and the introduction of the parameter $\beta$ given by \eqref{equbeta} below.
  \item[-] The profile is different from the profile in \cite{MZsbupdmj97}, hence also from all the previous studies in the parabolic case (\cite{MZsbupdmj97,EZ,NguyenZ,NZ}).
Therefore, brand new estimates are needed. See Section \ref{secexis} below.
  \item[-] In order to handle the new parameter $\beta$ in the profile, we introduce a new shirking set to trap the solution. See Definition \ref{prop:V-def} below. Finding such a set is not trivial,
  in particular the limitation $p>3$ in related to the choice of such a  set.
  \item[-] A good understanding of the dynamics of the linearized operator of equation \eqref{eq:wequation} below around the new profile is needed, taking into account the new shrinking set.
  \item[-] Some crucial global and pointwise estimates of the gradient of the solution as well as fine parabolic regularity results are needed (see Section \ref{subsecpr} below).
 \end{itemize}
 Then, following \cite{MZsbupdmj97}, the proof is divided in two steps. First, we reduce the problem to a finite dimensional one. Second, we solve the finite dimensional problem and conclude by contradiction, using index theory.

 To prove the single point blow-up result for the constructed solution, we establish a new ``no blow-up under some threshold'' criterion  for a parabolic inequality with a nonlinear gradient term.
See Proposition \ref{similar to Giga-Kohn} below. The final blow-up profile is determined then using the method of \cite{Zihp98}, Proposition  \ref{similar to Giga-Kohn} and \cite{MerleCPAM92}.

The stability result, Theorem \ref{th2}, is proved similarly as in
\cite{MZsbupdmj97} by interpreting the finite dimensional problem in
terms of the blow-up time and the blow-up point.

\bigskip

Let us remark that if $u$ is a solution of equation \eqref{eq:uequation}, then
$${\overline{u}}(x,t)=u\left(\mu^{-{2\over 2- q}}t,\mu^{-{1\over 2-q}}x\right),\; t\in \left[0,\mu^{{2\over 2- q}}T\right),\; x\in \R^N,$$ is a solution of the
equation
$$\partial_t {\overline{u}}  = \Delta {\overline{u}} +|\nabla {\overline{u}}|^q+\mu^{-{2\over 2-q}}|{\overline{u}}|^{p-1}{\overline{u}}.$$
Also, $${\tilde{u}}(x,t)=u\left(\mu^{-1}t,x\right),\; t\in \left[0,\mu T\right),\; x\in \R^N,$$ is solution of the equation
$$\partial_t {\tilde{u}}  =\mu^{-1} \Delta {\tilde{u}} +|\nabla {\tilde{u}}|^q+\mu^{-1}|{\tilde{u}}|^{p-1}{\tilde{u}}.$$
And for $\delta\not=1,$
$$\underline{u}(x,t)=\lambda^{2/(p-1)}u(\lambda^\delta x,\lambda^2 t),\; \lambda=\mu^{{1\over q(\delta-1)}}, \; t\in \left[0,\mu^{-{2\over q(\delta- 1)}}T\right),\; x\in \R^N,$$ is solution of the equation
$$\partial_t {\underline{u}}  =\mu^{-2/q}\Delta {\underline{u}} +|\nabla {\underline{u}}|^q+|{\underline{u}}|^{p-1}{\underline{u}}.$$
Then, since Theorem \ref{th1} is valid for all $\mu>0,$ we obtain the blow-up profile for a perturbed Viscous Hamilton-Jacobi (VHJ) equation, as well as for a perturbed Hamilton-Jacobi (HJ) equation, and nonlinear Hamilton-Jacobi (NHJ) equation.
More precisely, this is our statement:
\begin{coro}[Blow-up in the Hamilton-Jacobi style] Theorems \ref{th1} and \ref{th2} yield stable blow-up solutions:
\begin{itemize}
 \item[(i)] For the perturbed VHJ equation:
$$\partial_t u  = \Delta u +|\nabla u|^q+\nu|u|^{p-1}u,\;
\mbox{ with } \nu>0, \; 3/2<q<2,\; p={q\over 2-q}.$$
 \item[(ii)] For the perturbed HJ equation:
$$\partial_t u  = |\nabla u|^q+\nu'\Delta u+\nu'|u|^{p-1}u,\; \mbox{ with } \nu'>0, \; 3/2<q<2,\; p={q\over 2-q}.$$
\item[(iii)] For the perturbed NHJ equation:
$$\partial_t u  = |\nabla u|^q+|u|^{p-1}u+\nu''\Delta u ,\; \mbox{ with } \nu''>0, \; 3/2<q<2,\; p={q\over 2-q}.$$
\end{itemize}

In the three cases, the solutions and their gradients blow up simultaneously and only at one point. The blow-up profile is given by (\ref{bupprof1}) with appropriate scaling.
\end{coro}

 The organization of the rest of this paper is as follows. In Section 2, we explain formally how we obtain the the profile and the exponent $\beta.$
 In Section 3, we give a formulation of the problem in order to justify the formal argument. Section 4 is divided to two subsections: In subsection 4.1 we give the proof of the existence of the profile assuming the technical results.
 In particular, we construct a shrinking set and give an example of initial data giving rise to
 the prescribed blow-up profile. Subsection 4.2 is devoted to the proof of the technical results which are needed in the proof of the existence.  Section 5 is devoted to the proof of the single point blow-up and the determination of the final profile.
 In particular,
a new ``no blow-up under some threshold''
is established for parabolic equations (or inequalities) with nonlinear gradient terms. See Proposition \ref{similar to Giga-Kohn} below. Finally, in Section 6, we prove the stability result, that is Theorem \ref{th2} and give a more general stability statement (see Theorem \ref{th2}' page \pageref{th2'}). In all the paper the notation $A\ll B$ for positive real numbers $A$ and $B$ means that $A$ is very smaller with respect to $B.$

\section{A Formal Approach}\label{secform}
\setcounter{equation}{0}
The aim of this section is to explain formally how we derive the behavior given in Theorem \ref{th1}. In particular, how we obtain the profile $\vpz$ in \eqref{145}, the parameter $b$ and the exponent $\beta=2(p+1)/(p-1)$ in
\eqref{bbupprof1}.
 Consider an  arbitrary $T>0$ and the self-similar transformation of
(\ref{eq:uequation})
\begin{equation}\label{framwork:selfsimilar}
w(y,s)= (T-t)^{\frac{1}{p-1}}u(x,t),\; y=\frac{x}{\sqrt{T-t}},~
  s=-\log{(T-t)}.
  \end{equation}
It follows that if $u(x,t)$ satisfies (\ref{eq:uequation}) for all
$(x,t)\in \R^N\times [0,T),$ then $w (y,s)$ satisfies  the following
equation:
\begin{equation}\label{eq:wequation}
 \partial_s w=\Delta w-\frac{1}{2}y\cdot\nabla
  w-\dfrac{1}{p-1}w+\mu|\nabla
  w|^{q}+|w|^{p-1}w,
\end{equation}
for all $(y,s)\in \R^N\times[-\log T,\infty).$ Thus, constructing a
solution $u(x,t)$ for the equation (\ref{eq:uequation}) that blows
up at $T<\infty$ like $(T-t)^{-{\frac{1}{p-1}}}$ reduces to
constructing  a global solution $w(y,s)$ for equation
(\ref{eq:wequation}) such that
\begin{equation}
\label{225}
 0<\varepsilon\leq
\limsup_{s\to\infty}\|w(s)\|_{L^\infty(\R^N)}\leq {1\over \varepsilon}.
\end{equation}

A first idea to construct a blow-up solution for \eqref{eq:uequation}, would be to find a stationary solution for (\ref{eq:wequation}), yielding a self-similar solution for \eqref{eq:uequation}.
It happens that when $\mu<0$ and $p$ is close to $1,$ the first author together with Souplet and Weissler were able in \cite{STWesbugiumj96}   to construct such a solution. Now, if $\mu>0,$ we know, still from \cite{STWesbugiumj96}  that it is not possible to construct
such a solution in some restrictive class of solutions (see \cite[Remark 2.1, p. 666]{STWesbugiumj96}), of course, apart from the trivial constant solution $w\equiv \kappa$ of (\ref{eq:wequation}), where
\begin{equation} \label{equkappa}\kappa=\Big({1\over p-1}\Big)^{{1\over p-1}}.
\end{equation}

\subsection{Inner expansion}

Following the approach of Bricmont and Kupiainen in \cite{BK}, we may look for a solution $w$ such that $w\to \kappa$ as $s\to \infty.$ Writing $$w=\kappa+\overline{w},$$ we see that
$\overline{w}\to 0$ as $s\to \infty$ and   satisfies the equation:
\begin{equation}
 \label{wbarre}
\partial_s\overline{w} =\mathcal{L}\overline{w}+\overline{B}(\overline{w})+\mu|\nabla \overline{w}|^q,
\end{equation}
 where
\begin{equation}\label{eq:Operator:L}\mathcal{L}=\Delta-\frac{1}{2}y\cdot\nabla+1,~\end{equation}
and
\begin{equation}\label{eqbgare01}
 \overline{B}(\overline{w})= |\overline{w}+\kappa|^{p-1}(\overline{w}+\kappa)-\kappa^p-p\kappa^{p-1}\overline{w}.
\end{equation}
Note that
$$|\overline{B}(\overline{w})-{p\over 2\kappa}\overline{w}^2|\leq C |\overline{w}^3|,$$
where $C$ is a positive constant.

Let us recall some properties of $\mathcal{L}$. The operator $\mathcal{L}$  is  self-adjoint in $D(\mathcal L)
\subset L_\rho^2(\mathbb R^N)$ where
$$L_\rho^2(\mathbb{R^N})=\left\{f\in L_{loc}^2(\mathbb R^N)\; \Big|\; \int_{\mathbb R^N} \left(f(y)\right)^2\rho(y)dy<\infty \right\}$$
and
$$\rho(y)
=\frac{e^{\frac{-|y|^2}{4}}}{(4\pi)^{N/2}},\; y\in \R^N.$$
 The spectrum of
$\mathcal{L}$ is explicitly given by
$$spec(\mathcal L)=\left\{1-\frac{m}{2}\; \Big| \; m\in \mathbb N\right\}.$$
It consists only in  eigenvalues. For $N=1,$ all the eigenvalues are simple, and the eigenfunctions
are dilations of Hermite polynomials: the eigenvalue $
1-\frac{m}{2}$ corresponds  to  the following eigenfunction:
\begin{equation} \label{h_m} h_m(y)=\sum_{n=0}^{[\frac{m}{2}]} \frac{m!}{n!(m-2n)!}(-1)^n
y^{m-2n}.\end{equation} In particular $h_0(y)=1,\; h_1(y)=y$ and
$h_2(y)=y^2-2.$ Notice that $h_m$ satisfies:
\[
\int_{\R} h_n h_m \rho dx=2^nn!\delta_{nm}
\mbox{ and }
\mathcal{L}h_m=\Big(1-{m\over 2}\Big)h_m .
\]
 We also introduce
\begin{equation}\label{defkm}
k_m={h_m\over \|h_m\|_{L_\rho^2(\mathbb{R})}^2} .
\end{equation}
For $N\geq 2$, the eigenspace corresponding to $\lambda=1-\frac m2$ is given by
\[
\{h_{m_1}(y_1)\dots h_{m_N}(y_N)\;|\;m_1+\dots+m_N=m\}.
\]
In particular, for $\lambda=1$, the eigenspace is $\{1\}$, for $\lambda=\frac 12$, it is $\{y_i\;|\;i=1,\dots,N\}$, for $\lambda=0$, it is given by $\{h_2(y_i),\;y_iy_j\;|\;i=1\dots,N,\;i\neq j\}$.\\
%
 In compliance with the spectral properties of $\mathcal{L},$ we may look for a solution expanded as follows:
\[
\overline{w}(y,s)=\sum_{(m_1,\dots,m_N)\in\N^N}\overline{w}_{(m_1,\dots,m_N)} (s) h_{m_1}(y_1)\dots h_{m_N}(y_N).
\]
Since the eigenfunctions for $m_1+\dots+m_N\ge 3$
correspond to negative eigenvalues of $\mathcal{L}$, assuming $\overline{w}$ radial,
we may consider that
\begin{equation}\label{ansatz}
\overline{w}(y,s)=\overline{w}_0(s)+\overline{w}_2(s)\sum_{i=1}^Nh_2(y_i)=\overline{w}_0(s)+\overline{w}_2(s)(|y|^2-2N),
\end{equation}
with $\overline{w}_0,\; \overline{w}_2\;\to 0$ as $s\to \infty.$

Projecting Equation \eqref{wbarre}, and writing $\mu |\nabla \overline{w}|^q=\mu 2^q|y|^q|\overline{w}_2|^q,$ we derive the following ODE system for $\overline{w}_0$ and $\overline{w}_2:$
$$\overline{w}_0'=\overline{w}_0+{p\over 2\kappa}\left(\overline{w}_0^2+8N\overline{w}_2^2\right)+\tilde{c_0}|\overline{w}_2|^q+O\left(|\overline{w}_0|^3+|\overline{w}_2|^3\right),$$
$$\overline{w}_2'=0+{p\over \kappa}\left(\overline{w}_0\overline{w}_2+4\overline{w}_2^2\right)+\tilde{c_2}|\overline{w}_2|^q+O\left(|\overline{w}_0|^3+|\overline{w}_2|^3\right),$$
where
\begin{equation*}
 {\tilde{c}}_0=\mu2^q\int_{\R^N}|y|^q\rho\; \mbox{ and } \; {\tilde{c}}_2=\frac{\mu2^q}{8N}\int_{\R^N}|y|^q(|y|^2-2N)\rho.
\end{equation*}
Note that for this calculation, we need to know the values of
\begin{align*}
\int_{\R^N}(|y|^2-2N)^2\rho(y) dy &=N\int_{\R}(|\xi|^2-2)^2 \rho(\xi) d\xi=
 8N,\\
\int_{\R^N}(|y|^2-2N)^3\rho(y) dy& = N\int_{\R}(|\xi|^2-2)^3 \rho(\xi) d\xi=
 64N.
\end{align*}
Note also that the sign of ${\tilde{c}}_0$ and  ${\tilde{c}}_2$ is the same as for $\mu.$ Indeed, obviously $\int_{\R^N}|y|^q\rho(y)dy>0,$ and for $\int_{\R^N}|y|^q(|y|^2-2N)\rho(y)dy,$ using the radial coordinate $r=|y|$ and an
integration by parts, we write
\begin{align}
\nonumber
\frac{8N\tilde c_2}{2^q\mu}=&\int_{\R^N}|y|^q(|y|^2-2N)\rho(y)dy
 =
\int_{\R^N}|y|^{q+2}\rho(y)dy-2N\int_{\R^N}|y|^q\rho(y)dy\\ 
 =&
2(q+N)\int_{\R^N}|y|^{q}\rho(y)dy-2N\int_{\R^N}|y|^q\rho(y)dy \label{8+losange}
 = 2q\int_{\R^N}|y|^{q}\rho(y)dy>0.
\end{align}
From the equation on $\overline{w}_2',$ we write
$$\overline{w}_2'=\tilde{c_2}|\overline{w}_2|^q\left(1+O\left(|\overline{w}_2|^{2-q}\right)\right)+{p\over \kappa}\overline{w}_0\overline{w}_2+O\left(|\overline{w}_0|^3\right),$$
and assuming that
\begin{equation}
 \label{Hypotesesprofileh1}
|\overline{w}_0\overline{w}_2|\ll |\overline{w}_2|^q,\; |\overline{w}_0|^3\ll |\overline{w}_2|^q,
\end{equation}
we get that
$$\overline{w}_2'\sim sign(\mu)|{\tilde{c}}_2||\overline{w}_2|^q,$$
with $sign(\mu)=1$ if $\mu>0$ and $-1$ if $\mu<0.$\\
In particular, if $\mu>0,$ then $\overline{w}_2$ is increasing tending to $0$ as $s\to \infty$ hence $\overline{w}_2<0$,  while if $\mu<0,$  $\overline{w}_2$ is decreasing tending to $0$ as $s\to \infty$, hence $\overline{w}_2>0.$
Then, since $1<q<2,$
we get
$$\overline{w}_2 \sim -sign(\mu){B\over s^{{1\over q-1}}},$$
with
\begin{equation}
 B  = \left[(q-1)|\tilde{c}_2|\right]^{-{1\over q-1}} \label{5++++triangle}  =  \left[\frac{2^{q-2}}Nq(q-1)|\mu|\int_\R |y|^q\rho\right]^{-{1\over q-1}}
\end{equation}
from
\eqref{8+losange}.

From the equation on $\overline{w}_0',$ we write
$$\overline{w}_0'=\overline{w}_0\left(1+O\left(\overline{w}_0\right)\right)+\tilde{c_0}|\overline{w}_2|^q\left(1+O\left(|\overline{w}_2|^{2-q}\right)\right),$$
and assuming that
\begin{equation}
 \label{Hypotesesprofileh2}
|\overline{w}_0'|\ll \overline{w}_0,\;  |\overline{w}_0'|\ll |\overline{w}_2|^q,
\end{equation}
we derive that
$$\overline{w}_0 \sim -{\tilde{c}}_0|\overline{w}_2|^q \sim  {-{\tilde{c}}_0B^q\over s^{{q\over q-1}}}\ll |\overline{w}_2|.$$
Such $\overline{w}_0$ and $\overline{w}_2$ are compatible with the hypotheses \eqref{Hypotesesprofileh1} and \eqref{Hypotesesprofileh2}.

Therefore, since $w=\kappa+\overline{w}$,
it follows from \eqref{ansatz} that
\begin{eqnarray}
\nonumber
w(y,s) & = &\kappa+ \overline{w}_2(s)(|y|^2-2N)+o\left(\overline{w}_2\right)\\ \nonumber
&=&\kappa-\frac{sign(\mu)}{s^{\frac 1{q-1}}}B(|y|^2-2N)+o\left({1\over s^{{1\over q-1}}}\right)\\ \label{9++losange}
&=&\kappa
-sign(\mu)B{|y|^2\over s^{\frac 1{q-1}}}+2N\frac{sign(\mu)}{s^{\frac 1{q-1}}}B+o\left(\frac 1{s^{\frac 1{q-1}}}\right),
\end{eqnarray}
in $L^2_\rho(\R^N),$ and also uniformly on compact sets by standard parabolic regularity.

\subsection{Outer expansion} From \eqref{9++losange}, we see that the variable
$$z={y\over s^\beta},\; \mbox{ with } \beta={1\over 2(q-1)}={p+1\over 2(p-1)},$$
as given in \eqref{bbupprof1}, is perhaps the relevant variable for
blow-up. Unfortunately, \eqref{9++losange} provides no shape, since
it is valid only on compact sets (note that $z\to 0$ as $s\to
\infty$ in this case). In order to see  some shape, we may need to
go further in space, to the ``outer region'', namely when $z\not=
0.$ In view of \eqref{9++losange}, we may try to find an expression
of $w$ of the form
\begin{equation}
 \label{1++++triangle}
 w(y,s)= \vpz(z)+{a\over s^{2\beta}}+O\left({1\over s^\nu}\right),
\end{equation}
for some $\nu>2\beta.$ Plugging this ansatz in equation \eqref{eq:wequation}, keeping only the main order, we end-up with the following equation on $\vpz:$
\begin{equation}
 \label{ordre0profil}
-{1\over 2}z\cdot \nabla\vpz(z)-{1\over
p-1}\vpz(z_0)+[{\vpz}(z)]^p=0,\; z={y\over
s^\beta}.\end{equation} Recalling that our aim is to find $w$ a
solution of \eqref{eq:wequation} such that $w\to \kappa$ as $s\to
\infty$ (in $ L^2_\rho$, hence uniformly on every compact set), we
derive from \eqref{1++++triangle} (with $y=z=0$) the natural
condition
$$\vpz(0)=\kappa.$$
Recalling also that we already adopted radial symmetry for the inner equation, we do the same here. Therefore, integrating equation \eqref{ordre0profil}, we see that
\begin{equation}
 \label{2++++triangle}
 \vpz(z)=\Big(p-1+b|z|^2\Big)^{-{1\over
p-1}},
\end{equation}
for some $b\in \R.$ Recalling also that we want a solution $w\in
L^\infty(\R^N),$ (see \eqref{225}),  we see that $b\geq 0$ and for a
nontrivial solution, we should have
\begin{equation}
 \label{4++++triangle}
 b>0.
\end{equation}
Thus, we have just obtained from \eqref{1++++triangle} that
\begin{equation}
 \label{3++++triangle}
 w(y,s) =
 \Big(p-1+b|z|^2\Big)^{-{1\over
p-1}}+\frac a{s^{2\beta}}+O\left(\frac 1{s^\nu}\right),\; \mbox{ with }\; z={y\over s^\beta}\mbox{ and }\nu>2\beta.
\end{equation}
We should understand this expansion to be valid at least on compact
sets in $z,$ that is for $|y|< R s^\beta,$ for any $R>0.$

\subsection{Matching asymptotics} Since \eqref{3++++triangle} holds for $|y|< R s^\beta,$ for any $R>0$, it holds also uniformly on compact sets, leading to the following expansion for $y$ bounded:
\begin{equation*}
 w(y,s)
 = \kappa-{\kappa b \over (p-1)^2}{|y|^2\over s^{2\beta}}+{a\over s^{2\beta}}+O\left({1\over s^\nu}\right).
\end{equation*}
Comparing with \eqref{9++losange}, we find the following values for $b$ and $a:$
$$b=sign(\mu){B(p-1)^2\over \kappa } \; \mbox{ and }\; a=2N sign(\mu)B.$$
In particular, from \eqref{4++++triangle} we see that
$$ \mu>0.$$

In conclusion, using  \eqref{5++++triangle}, we see that we have just derived the following profile for $w(y,s):$
$$w(y,s) \sim \varphi (y,s)$$
with
\begin{equation}\label{feq}\varphi(y,s)=\vpz\Big({y\over s^\beta}\Big)+{a\over
s^{2\beta}}:=\Big(p-1+b{|y|^2\over s^{2\beta}}\Big)^{-{1\over
p-1}}+{a\over s^{2\beta}}.
\end{equation}
\begin{equation} \label{equbeta}
\beta={p+1\over 2(p-1) },
\end{equation}
\begin{equation}\label{aeq} a={2Nb\kappa\over (p-1)^2},\end{equation}
\begin{equation} \label{equb}b={1\over 2}(p-1)^{{p-2\over p-1}}\left((4\pi)^{N\over 2}(p+1)^2N\over p
\int_{\R^N}|y|^qe^{-{|y|^2/4}}dy\right)^{{p+1\over p-1}}\mu^{-(p+1)/(p-1)},
\end{equation}
\section{Formulation of the problem}
\setcounter{equation}{0}
In this section we formulate the problem in order to justify the formal approach given in the previous section. Very soon, actually starting from \eqref{def:q_b} given below, we will only focus on the case
\[
N=1
\]
for simplicity. The proof in higher dimensions is no more difficult.\\
 Let $w,\; y$ and $s$ be as in (\ref{framwork:selfsimilar}). Let us
introduce $\vq(y,s)$ such that
\begin{equation}\label{eq:q+phi+def} w(y,s)=\varphi(y,s)+\vq(y,s),\end{equation}
where $\varphi$ is given by \eqref{feq}.
If $w$ satisfies the equation (\ref{eq:wequation}), then $\vq$
satisfies the following equation:
\begin{equation}\label{qequation}
\partial_s \vq=(\mathcal{L}+V)\vq+B(\vq)+G(\vq)+{R}(y,s),
\end{equation}
where $\mathcal{L}$ is defined by \eqref{eq:Operator:L} and
\begin{equation}\label{eq:V}
 V(y,s)=p~\varphi(y,s)^{p-1}
-\frac{p}{p-1},
\end{equation}
\begin{equation}\label{eq:B}
B(\vq)=|\varphi+\vq|^{p-1}(\varphi+\vq)-\varphi^p-p\varphi^{p-1}\vq,\end{equation}
\begin{equation}\label{RNeq}
\begin{split}
&R(y,s)=\Delta\varphi-\frac{1}{2}y\cdot\nabla\varphi-\frac{\varphi}{p-1}+
\varphi^p-\frac{\partial \varphi}{\partial s}+\mu |\nabla \varphi|^q
\end{split}\end{equation}
and
 \begin{equation}\label{G} G(v)=\mu|\nabla\varphi+\nabla \vq|^q-\mu|\nabla \varphi|^q.
 \end{equation}
 Our aim is to construct  initial data $v(s_0)$ such that the equation
(\ref{qequation}) has a solution $\vq(y,s)$ defined for all
$(y,s)\in \R^N\times[-\log T,\infty),$ and satisfies:
\begin{equation}\label{eq:qt0} \lim_{s\rightarrow \infty
}\|\vq(s)\|_{W^{1,\infty}(\R^N)}=0.\end{equation} From Equation
(\ref{feq}), one sees that the variable  $z={y\over s^\beta}$ plays
a fundamental role. Thus we will consider the dynamics for $|z|>K$
and $|z|<2K$ separately for some $K>0$ to be fixed large.
Since
 \begin{equation}\label{estim:B-R}
|B(\vq)|\leq C|\vq|^2,\;  \|R(.,s)\|_{L^\infty}\leq \f{C}{s}, \|G(\vq)\|_{L^\infty(\R)}\leq {C\over \sqrt{s}} \|\vq\|_{L^\infty(\R)},\end{equation}  for $s$ large enough, (see \eqref{e20151}, \eqref{e20152} and \eqref{e20153} below),
it is then reasonable to think that the
dynamics of equation (\ref{qequation}) are influenced by the linear
part, namely $\mathcal{L}+V$.

\bigskip
The properties of the operator $\mathcal{L}$ were given in Section 2. In particular, $\mathcal{L}$ is predominant on all the modes, except on the null modes where the terms $V\vq $ and $G(\vq)$ will play a crucial role
(see item (ii) in Proposition \ref{prop:projet:q} below).

As for the potential  $V,$ it has two
fundamental properties  which will strongly  influence our strategy:
\begin{itemize}
 \item[(i)] we have  $ V(.,s)\rightarrow 0$ in
$L_\rho^2(\mathbb{R})$ when $s\rightarrow \infty.$ In practice,  the
effect of $V$ in the blow-up area $(|y|\leq Cs^\beta)$ is regarded
as a perturbation of the effect of  $\mathcal{L}$ (except on the null mode).
 \item[(ii)] outside of the blow-up area, we have the following property:
for all $\epsilon >0$, there exists $C_{\epsilon}>0$ and
$s_{\epsilon}$ such that $$ \sup_{s\geq
s_{\epsilon},~\frac{|y|}{s^{\beta}}\geq
C_{\epsilon}}\left|V(y,s)-(-\frac{p}{p-1})\right|\leq \epsilon,$$
with $-\frac{p}{p-1}<-1.$ As $1$ is the largest eigenvalue of the
operator $\mathcal{L}$, outside the blow-up area  we can consider
that  the operator $ \mathcal {L}+V$ is  an operator with negative
eigenvalues, hence, easily controlled.
\end{itemize}

Considering the fact that the behavior of $V$ is not the same inside
and outside  the blow-up area, we decompose  $ \vq$ as follows. Let
us consider a non-increasing cut-off function  $\chi_0\in
C^\infty\big([0,\infty),[0,1]\big)$ such that
$\mbox{supp}(\chi_0)\subset[0,2]$ and $\chi_0 \equiv 1 $ in $[0,1]$,
and introduce
\begin{equation}\label{def:chi}
\chi(y,s)=\chi_0\left(\frac{|y|}{K~s^{\beta}}\right)\end{equation}
with $K=\max(6,K_5)$ and $K_5=K_5(N,p,\mu)$ is some large enough constant introduced below in Lemma \ref{lem:A_5}.\\
Then, we write
\begin{equation}\label{def:q:proj}
\vq(y,s)=\vq_b(y,s)+\vq_e(y,s),
\end{equation}
with
\begin{equation}\label{def:q:projbis}
\vq_b(y,s)=\vq(y,s)\chi(y,s) \mbox{ and } \vq_e(y,s)=
\vq(y,s)\big(1-\chi(y,s)\big).\end{equation} We remark that $$
\mbox{supp }\vq_b(s)\subset B(0,2K{s^{\beta}}),\; \mbox{supp}~
\vq_e(s)\subset \mathbb{R}^N\setminus B(0,Ks^{\beta}).$$
As for $\vq_b,$ we will decompose it according to the sign of the eigenvalues of $\mathcal{L}$, by writing
\begin{equation}
\label{def:q_b}
\vq_b(y,s)= \sum_{m=0}^2
\vq_m(s)h_{m}(y)+\vq_{-}(y,s),
\end{equation}
where for $0\le m\le 2$, $v_m=P_m(v_b)$ and $\vq_{-}(y,s)=
P_{-}(\vq_{b})$, with $P_m$ the $L^2_\rho$ projector on $h_m$, the eigenfunction corresponding to $\lambda=1-\frac m2\ge 0$, and $P_{-}$  the projector on $\{h_i,\;|\;i\ge 3\}$, the negative
subspace of the  operator $\mathcal{L}$
(as announced in the beginning of the section, hereafter, we assume that $N=1$ for simplicity).\\
 Thus, we can decompose $\vq$ in  five components  as follows:
\begin{equation}\label{qprojection}
\vq(y,s)=\sum_{m=0}^2
\vq_m(s)h_{m}(y)+\vq_{-}(y,s)+\vq_e(y,s).
\end{equation}
Here and
throughout the paper, we call $\vq_{-} $ the negative mode of $\vq$,
$\vq_2 $ the null mode of $\vq$, and the subspace spanned by
$\left\{h_m\; |\; m\geq 3\right\}$ will be referred to as the negative
subspace.
\section{Existence}\label{secexis}
\setcounter{equation}{0}
In this section, we prove the existence of a solution $\vq$ of  (\ref{qequation}) such that
\begin{equation}
 \label{but1}
\lim_{s\to \infty}\|\vq(s)\|_{W^{1,\infty}(\R)}=0.
\end{equation}
Hereafter, we denote by $C$ a generic positive constant, depending only on $p$, $\mu$ and $K$ introduced in \eqref{def:chi}, itself depending on $p$ and $\mu$. In particular, $C$ does not depend on $A$ and $s_0$, the constants that will appear shortly and throughout the paper, and need to be adjusted for the proof.\\
We proceed in two subsections. In the first subsection, we give the
proof assuming the technical details. In the second subsection we
give the proofs of the technical details.
\subsection{Proof of the existence assuming technical results} Since $p>3,$ we see that, by definition of $\beta$ given by \eqref{equbeta}, $\beta\in ({1\over 2},1).$
Our construction is build on a careful choice of the initial data for $\vq$ at a time $s_0.$ We will choose it in the following form:
\begin{defi}[Choice of the initial data]\label{initial data}
Let us define, for $A\geq 1,$ $s_0=-\log T>1$ and $d_0,\;d_1\in \R,$
the function
\begin{equation}\label{initial-data}
 \psi_{s_0,d_0,d_1}(y)={A\over s_0^{2\beta+1}}\Big(d_0h_0(y)+d_1h_1(y)\Big)\chi(2y,s_0),
\end{equation}
where $h_i,\; i=0,\,1$ are defined by (\ref{h_m}) and $\chi$ is
defined by (\ref{def:chi}).
\end{defi}

The solution of equation \eqref{qequation} will be denoted by
$\vq_{s_0,d_0,d_1}$ or $\vq$ when there is no ambiguity. We will
show that if $A$ is fixed large enough, then, $s_0$ is fixed large
enough depending on $A,$  we can fix the parameters $(d_0,d_1)\in
[-2,2]^2,$ so that the solution $\vq_{s_0,d_0,d_1}(s)\to 0$ as $s\to
\infty$ in $W^{1,\infty}(\R),$ that is, \eqref{but1} holds. Owing to
the decomposition given in \eqref{initial-data}, it is enough to
control the solution in a shrinking set  defined as follows:
\begin{defi}[A set shrinking to zero]\label{prop:V-def} Let $\gamma$ be any real number  such that
\begin{equation}\label{gamma}
3\beta<\gamma<\min(5\beta-1,2\beta+1).
\end{equation}
 For all  $ A\geq 1$ and $
s\geq 1$, we define $\vartheta_A(s)$ as the set of all functions $r
\in L^\infty(\R)$ such that

\[||r_e||_{L^\infty(\R)}~\leq~ {A^2\over
s^{\gamma-3\beta}},\; \;\;\;\Big\|{r_-(y)\over 1+|y|^3}
\Big\|_{L^\infty(\R)}~\leq~ {A\over s^{\gamma}},\; ~\]

\[|r_0|,\; |r_1|\leq
~{A\over s^{2\beta+1}},\;\;\;~|r_2|~\leq~{\sqrt{A}\over
s^{4\beta-1}} ,\]
\\
where $r_-$, $r_e$ and  $r_m$ are defined in (\ref{qprojection}).
\end{defi}
\begin{rem}
{\rm Since $p>3,$ it follows that $\frac 12<\beta<1,$ in particular the range for $\gamma$ in \eqref{gamma} is not empty. Of course, the set $\vartheta_A(s)$ depends also on the choice of $\gamma$ satisfying \eqref{gamma}.
However, while $A$ will be chosen large enough so that various estimates hold, $\gamma$
will be fixed once for all throughout the proof.}
\end{rem}

Since $A\geq1,$  then the sets $\vartheta_A(s)$ are increasing (for
fixed $s$) with respect to $A$ in the sense of inclusion.  We also
show  the following property of elements of $\vartheta_A(s):$

\medskip
{\it
For all $A\geq 1$, there exists $s_{01}(A) \geq 1$ such that, for
all $s\geq s_{01}$  and $r\in \vartheta_A(s),$ we have
\begin{equation}
 \label{prop:V-defbis}
 \|r\|_{L^\infty(\R)}\leq C\frac{A^2}{s^{\gamma-3\beta}},
\end{equation}
where  $C$ is a positive constant (see Proposition \ref{prop:V-def-prop} below for the proof).
}

\medskip

By \eqref{prop:V-defbis}, if a solution $v$ stays in $\vartheta_A(s)$ for $s\geq s_0,$ then it converges to $0$ in $L^{\infty}(\R)$ (the convergence of the gradient will follow from parabolic regularity). Reasonably, our aim is then reduced to prove the following proposition:

\begin{prop}[{\it Existence of solutions trapped in  $\vartheta_A(s)$}]\label{dans vs} There exists $A_2\ge 1$ such that for $A\ge A_2$ there exists $s_{02}(A)$ such that for all $s_0\ge s_{02}(A)$, there exists  $(d_0,d_1)$ such that if
$v$ is the solution of (\ref{qprojection}) with initial data at $s_0,$  
given by  (\ref{initial-data}), then
 $\vq(s)\in \vartheta_A(s)$, for all $s\geq s_0.$
\end{prop}

This proposition gives the stronger convergence to $0$ in
$L^{\infty}(\R)$ thanks to \eqref{prop:V-defbis}, and the
convergence in $W^{1,{\infty}}(\R)$ will follow from parabolic
regularity as we explain below.

Let us first be sure that we can choose the initial data such that it starts in  $\vartheta_A(s_0).$ In other words, we will define a set where we will at the end select the good parameter $(d_0, d_1)$ that will give the conclusion of Proposition \ref{dans vs}. More precisely,
we have the following result:

\begin{prop}[Properties of initial data]\label{prop:diddc} For each $A\geq 1$, there
exists $s_{03}(A)>1$ such that for all $s_0\geq s_{03}$:
\begin{itemize}
\item[(i)] There exists a rectangle
\begin{equation}\label{D}
\mathcal{D}_{s_0}\subset [-2,2]^2
 \end{equation}
 such that the mapping
\begin{eqnarray*}
\Phi :\R^2&\rightarrow & \R^2,\\(d_0,d_1) &\mapsto&
(\psi_0,\psi_1).
\end{eqnarray*}
(where $\psi$ stands for $\psi_{s_0,d_0,d_1}$) is linear, one to
one from $\mathcal{D}_{s_0}$ onto $[-{A\over s_0^{2\beta+1}},{A\over
s_0^{2\beta+1}}]^2$ and maps $
\partial \mathcal{D}_{s_0} $ into $\partial\Big(
[-{A\over s_0^{2\beta+1}},{A\over s_0^{2\beta+1}}]^2\Big)$. Moreover, it has degree one on the
boundary.
\item[(ii)]  For all $(d_0,d_1) \in \mathcal{D}_{s_0},$
$\psi:=\psi_{s_0, d_0, d_1} \in \vartheta_A(s_0)$ with strict
inequalities except for $(\psi_0,\psi_1)$, in the sense that
\begin{equation}
\label{e54}
  \;\psi_e\equiv 0~,\;\; ~ |\psi_-(y)|<~
{1\over s_0^{\gamma}} (1+|y|^3),\;\forall\; y\in\; \R,
\end{equation}

\begin{equation}
\label{e54bis}|\psi_0|\leq ~{A\over s_0^{2\beta+1}},~\; |\psi_1|\leq
~{A\over s_0^{2\beta+1}},~|\psi_2|<~{1\over
s_0^{4\beta-1}} .
\end{equation}

\item[(iii)] Moreover, for all $(d_0,d_1)\in \mathcal{D}_{s_0},$ we
have
\begin{equation}
\label{qrdienq_s0}
 \|\nabla\psi\|_{L^\infty(\R)}\leq {CA\over s_0^{2\beta+1}} \leq \frac{1}{s^{\gamma-3\beta}_0},
\end{equation}

\begin{equation}
\label{qrdienq_s0point2} |\nabla \psi_-(y)|\leq
\frac{1}{s_0^{\gamma}}(1+|y|^3),
 \; \forall \; y\in \R.
 \end{equation}
\end{itemize}
\end{prop}

The proof of the previous proposition is postponed to Subsection
4.2. Let us now give the proof of Proposition \ref{dans vs}.

\begin{proof}[Proof of Proposition \ref{dans vs}] Let us consider $A\geq 1,$ $s_0\geq s_{03},$ $(d_0,d_1)\in \mathcal{D}_{s_0},$ where $s_{03}$ is given by Proposition \ref{prop:diddc}. From the existence theory (which follows from the Cauchy problem
for equation \eqref{eq:uequation} in $W^{1,\infty}(\R))$ mentioned in the introduction), starting in $\vartheta_A(s_0)$ which is in $\vartheta_{A+1}(s_0),$ the solution stays in $\vartheta_A(s)$ until some maximal time $s_*=s_*(d_0,d_1).$ If
 $s_*(d_0,d_1)=\infty$ for some $(d_0,d_1)\in \mathcal{D}_{s_0}$, then  the proof is complete. Otherwise, we argue by contradiction and suppose that $s_*(d_0,d_1)<\infty$ for any $(d_0,d_1)\in \mathcal{D}_{s_0}.$
 By continuity and the definition of $s_*$,  the solution at the  point $s_*$, is on the boundary of
 $\vartheta_{A}(s_*).$ Then, by definition of $\vartheta_{A}(s_*),$ one at least of the inequalities in that definition  is an equality. Owing to the following proposition, this can happen only
 for the first two components. Precisely, we have the following result:
 \begin{prop}[{\it Control of  $\vq(s)$ by $(\vq_0(s),\vq_1(s))$ in $ \vartheta_A(s) $}] \label{prop:rt2dimen} There
exists $A_4\geq 1$ such that  for each $A\geq~A_4,$
there exists
$s_{04}(A)\in \R$
such that for all $s_0\geq s_{04}(A)$,
 the following holds:

If $\vq$ is a solution of (\ref{qequation}) with initial data at $s=s_0$ given by
(\ref{initial-data}) with $(d_0,d_1)~\in~\mathcal{D}_{s_0} $, and $\vq(s)~\in~\vartheta_A(s)$  for all  $ s \in~[s_0,s_1],~$
with $\vq(s_1)\in \partial \vartheta_A(s_1)$ for some $s_1\geq s_0$, then:
\begin{itemize}
\item[(i)] (Reduction to a finite dimensional problem) We have: $$ \left(\vq_0(s_1),\vq_1(s_1)\right)\in \partial\left( \left[-\f{A}{s_1^{2\beta+1}},\f{A}{s_1^{2\beta+1}}\right]^2\right).$$
\item[(ii)] (Transverse crossing) There exist $m\in \{0,1\}$ and $\omega \in \{-1,1\}$ such that
\begin{equation*}
 \omega \vq_m(s_1)=\f{A}{s_1^{2\beta+1}}\mbox{ and }  \omega\vq_m'(s_1)>0.
\end{equation*}
\end{itemize}
\end{prop}
Assume the result of the previous proposition, for which the proof is given in Subsection 4.2 below, and continue the proof of Proposition \ref{dans vs} . Let $A\ge A_4$ and $s_0\geq s_{04}(A).$
It follows from Proposition \ref{prop:rt2dimen},
part (i), that $\left(\vq_0(s_*),\vq_1(s_*)\right)\in \partial\left( \left[-\f{A}{s_*^{2\beta+1}},\f{A}{s_*^{2\beta+1}}\right]^2\right),$
and the following function \begin{align*}\Phi&: \mathcal{D}_{s_0} \rightarrow \partial\left([-1,1]^2\right)\\
&(d_0,d_1)\mapsto
\dfrac{s_*^{2\beta+1}}{A}(\vq_0,\vq_1)_{(d_0,d_1)}(s_*),\; \mbox{ with } s_*=s_*(d_0,d_1),\end{align*}
is well defined. Then, it follows from Proposition \ref{prop:rt2dimen}, part (ii) that  $\Phi$ is continuous. On the other hand, using Proposition \ref{prop:diddc}, parts (i) and (ii)  together with  the fact that
$\vq(s_0)= \psi_{s_0,d_0,d_1}$, we see that when $(d_0,d_1)$ is on the boundary of the rectangle $\mathcal{D}_{s_0},$ we have strict
inequalities for the other components. Applying the transverse crossing property  given in Proposition \ref{prop:rt2dimen}, part (ii), we see that
$\vq(s)$  leaves  $\vartheta_A(s)$ at $s=s_0$,  hence $s_*(d_0,d_1)=s_0$. Using Proposition \ref{prop:diddc}, part (i), we see that the  restriction of $\Phi$
to the boundary is of degree 1. A contradiction then follows from the index theory. Thus, there exists a value $(d_0,d_1)\in \mathcal{D}_{s_0}$
such that for all $s\geq s_0,~ \vq_{s_0,d_0, d_1}(s)\in \vartheta_A(s)$.
This concludes the proof of  Proposition \ref{dans vs}.

\bigskip
{\it Completion of the proof of \eqref{but1}.} By Proposition
\ref{dans vs} and \eqref{prop:V-defbis}, it remain only to show that
\newline $\|\nabla \vq(s)\|_{L^\infty(\R)}\to 0$ as $s\to \infty.$
We will prove the following parabolic regularity for equation
(\ref{qequation}):

\medskip

{\it
For all $A\geq 1$, there exists
$s_{05}(A)\geq s_{04}(A)$ such that for all $s_0\geq s_{05}(A)$ the
following holds: If $\vq(s)$ is a solution of equation
(\ref{qequation}) on $[s_0,s_1]$ where $s_1\geq s_0$ with initial
data at $s=s_0,$ $\vq(s_0)=\psi_{s_0,d_0,d_1},$ $(d_0,d_1)\in
\mathcal D_{s_0},$
 $\vq(s)\in \vartheta_A(s) \mbox{ for all } s\in [s_0,s_1]$,
 then, for all $s\in[s_0,s_1],$ we have
\begin{equation}
\label{prop:regu-parab-q-equationbis} \|\nabla \vq(s)\|_{L^\infty(\R)}\leq \frac{C_1A^2}{s^{\gamma-3\beta}},
\end{equation}
where $C_1$ is a positive constant.}

\medskip

 This will be proved in Subsection \ref{4.2} below. Then, from \eqref{prop:V-defbis}, Proposition \ref{dans vs} and \eqref{prop:regu-parab-q-equationbis}, we have
$$\|\vq(s)\|_{W^{1,\infty}(\R)}\leq {C(A)\over s^{\gamma-3\beta}},$$
hence, by \eqref{gamma}, \eqref{but1} follows by taking $s_{02}\geq \max\left(s_{01},s_{03},s_{04},s_{05}\right)$.

\end{proof}

\subsection{Proof of the technical results}\label{4.2} In this section, we prove all the technical results used without proof in the previous one, thus, finishing the argument for the proof of the existence of a solution of \eqref{qequation}
satisfying \eqref{but1}. More precisely, we proceed in 4 steps, each given in a separate section.
\begin{itemize}
 \item[-] We first establish the needed properties on initial data and stated in Proposition \ref{prop:diddc}. In particular, we show that initial data is trapped in $\vartheta_A(s_0),$ provided that $s_0$ is large enough, and the parameters
  $(d_0,d_1)$ are in a suitable set.
 \item[-]Then, we show that the rest and the nonlinear terms of equation \eqref{qequation} are trapped in $\vartheta_C(s)$ for some positive $C,$ assuming $\vq\in  \vartheta_A(s)$ if necessary. For the potential term, we show that it is in $\vartheta_{CA}(s),$
 assuming $\vq\in \vartheta_A(s).$
 \item[-] In the third step, we give parabolic regularity estimates, proving in particular estimate \eqref{prop:regu-parab-q-equationbis}.
 \item[-] Finally, we prove Proposition \ref{prop:rt2dimen}, concerning the reduction of the problem to a two-dimensional one.
\end{itemize}

\subsubsection{Preparation of the initial data}
In this subsection, we give some properties
of the set $\vartheta_A(s)$  introduced in Definition \ref{prop:V-def} and prove Proposition \ref{prop:diddc} concerning initial data. We first claim the following:
\begin{prop}[{\it Properties of elements of $\vartheta_A(s)$}]\label{prop:V-def-prop} For all $A\geq 1,$ there exists $s_{10} \geq 1$
such that, for all $s\geq s_{10}$  and $r\in \vartheta_A(s),$ we
have

\begin{itemize}
\item[(i)] $\|r\|_{L^\infty(|y|\leq 2Ks^\beta)}\leq  C\frac{A}{s^{\gamma-3\beta}}$
and  $\|r\|_{L^\infty(\R)}\leq C\frac{A^2}{s^{\gamma-3\beta}}.$
\vspace{0,2cm}
\item[(ii)]$|r_b(y)|\leq C\frac{A}{s^{4\beta-1}}(1+|y|^3),\; |r_e(y)|\leq
C\frac{A^2}{s^{\gamma}}(1+|y|^3),$
and $|r(y)|\leq
C\frac{A^2}{s^{4\beta-1}}(1+|y|^3),\;\forall\; y\in \R.$
\vspace{0,2cm}
\item[(iii)] $|r(y)|\leq C\Big[\frac{A}{s^{2\beta+1}}(1+|y|)+\frac{\sqrt{A}}{s^{4\beta-1}}(1+|y|^2)+
\frac{A^2}{s^{\gamma}}(1+|y|^3)\Big],\;\forall\;y \in \R,$
\end{itemize}
where $C$ is a positive constant.

\end{prop}

\begin{proof} Take $A\geq 1,\; s\geq 1$, $r\in \vartheta_A(s)$ and
$y\in \R.$ Recall that $r(y)=r_b(y)+ r_e(y)$ where
$$r_b(y)=\sum_{m=0}^2 r_mh_m(y)+r_-(y),\; r_e=r(1-\chi)$$
with $\chi$ defined by (\ref{def:chi}). In particular, $\mbox{supp }
r_b\subset \{|y|\leq 2Ks^\beta\}$ and $\mbox{supp } r_e\subset
\{|y|\geq Ks^\beta\}.$

(i) If $|y|\leq 2Ks^\beta ,$ using  the definition (\ref{h_m}) of
$h_m$  and that of $\vartheta_A(s)$ we get:
\begin{equation}
\label{e41} |r_b(y)|\leq (1+|y|)\frac{A}{s^{2\beta+1}}+
C(1+|y|^2)\frac{\sqrt{A}}{{s^{4\beta-1}}}+
(1+|y|^3)\frac{A}{s^{\gamma}}.
\end{equation}
It follows, for $s$ sufficiently large, that
\begin{eqnarray*}
|r_b(y)|&\leq& (1+2Ks^\beta)\frac{A}{s^{2\beta+1}}+
C\big(1+(2Ks^\beta)^2\big)\frac{\sqrt{A}}{{s^{4\beta-1}}}+
\big(1+(2Ks^\beta)^3\big)\frac{A}{s^{\gamma}}\\
&\leq& C\frac{A}{s^{\beta+1}}+
C\frac{\sqrt{A}}{{s^{2\beta-1}}}+ C\frac{A}{s^{\gamma-3\beta}}\\
&\leq & C_1\frac{A}{s^{\gamma-3\beta}},
\end{eqnarray*}
since $\gamma<5\beta-1$ and $A\geq 1.$ Moreover,
\begin{eqnarray*}
\|r\|_{L^\infty(\R)}&\leq&
\|r_b\|_{L^\infty(\R)}+\|r_e\|_{L^\infty(\R)} \leq
C\frac{A}{s^{\gamma-3\beta}}+\frac{A^2}{s^{\gamma-3\beta}} \\
&\leq & C\frac{A^2}{s^{\gamma-3\beta}},
\end{eqnarray*}
which gives (i).

(ii) Since  $A\geq 1,\; s\geq 1$ and $4\beta-1<\min(\gamma,2\beta+1)$ (by definition \eqref{gamma} of $\gamma$), if  $|y|\leq 2Ks^\beta ,$ we write from (\ref{e41}):
$$|r_b(y)|\leq C \frac{A}{s^{4\beta-1}}(1+|y|^3).$$
Since $r_b(y,s)\equiv 0$ when $|y|\leq 2Ks^\beta ,$ the last inequality is obviously true also.

If $|y|\geq Ks^\beta,$ we have, the definition of $\vartheta_A(s)$,
\begin{equation}
\label{27012015-1}
|r_e(y)|\leq \|r_e\|_{L^\infty(\R)}\leq \frac{A^2}{s^{\gamma-3\beta}}\leq
\frac{A^2}{s^{\gamma-3\beta}}{|y|^3\over (Ks^\beta)^3}\leq
C\frac{A^2}{s^{{\gamma}}}(1+|y|^3).
\end{equation}
The inequality  is verified for all $y\in \R$,
since $r_e(y,s)\equiv 0$ when $|y|\le K s^\beta$.

On the other hand, by the conditions on $A$ and $\beta,$ we have
$$|r(y)|\leq |r_b(y)|+|r_e(y)|\leq C\frac{A^2}{s^{4\beta-1}}(1+|y|^3),\;\forall\; y\in
\R,$$ which gives (ii).

(iii) Using (\ref{e41}) and \eqref{27012015-1}, we get
\begin{eqnarray*}
|r(y)|&\leq& |r_b(y)|+|r_e(y)|\\ &\leq &
(1+|y|)\frac{A}{s^{2\beta+1}}+
C(1+|y|^2)\frac{\sqrt{A}}{{s^{4\beta-1}}}+
(1+|y|^3)\frac{A}{s^{\gamma}}+C\frac{A^2}{s^{\gamma}}(1+|y|^3)\\
&\leq & (1+|y|)\frac{A}{s^{2\beta+1}}+
C(1+|y|^2)\frac{\sqrt{A}}{{s^{4\beta-1}}}+C\frac{A^2}{s^{\gamma}}(1+|y|^3).
\end{eqnarray*}

This finishes the proof of Proposition \ref{prop:V-def-prop}.

\end{proof}

We need a second technical estimate before proving  Proposition \ref{prop:diddc}:
%
%
\begin{lem}
\label{l41}
There exists $s_{10}'$ such that for all $s_0\ge s_{10}'$,
if $g$ is given by $\chi(2y,s_0)$ or
$y\chi(2y,s_0) $, then
$$g_e(y) \equiv 0,\;\;\Big\|{g_-(y)\over 1+|y|^3}\Big\|_{L^\infty(\R)}\leq {C\over s_0^{2\beta}},$$
and all $g_m$ are less then $C\mbox{e}^{- s_0^{2\beta}}$,
except:
\begin{itemize}
\item[(i)] $|g_0-1|\leq C\mbox{e}^{-s_0^{2\beta}}$ when $g(y)=\chi(2y,s_0),$
\item[(ii)]$|g_1-1|\leq C\mbox{e}^{-s_0^{2\beta}}$ when $g(y)=y\chi(2y,s_0)$.
\end{itemize}
\end{lem}
\begin{rem}{\rm Here in this lemma, we need the fact that $K\ge 6$.}
\end{rem}
\begin{proof}
Since $(1-\chi(y,s_0))\chi(2y, s_0)\equiv 0$ and $\chi(y,s_0))\chi(2y, s_0\equiv 1$ , it follows that  $g_e(y)\equiv 0$ and $g_b(y)=\chi(y,s_0)g(y) = g(y)$. In particular, $g_m=P_m(g)$ and $g_-=P_-(g)$, where $P_m$ and $P_-$ are the $L^2_\rho$ projectors on $h_m$ and $\{h_i\;|\;i\ge 3\}$, respectively.
Writing
$g(y)=\bar g(y)+r(y)$ where $\bar g(y)=1$ or $y$  and $r(y)=\bar g(y)\big(\chi(2y,s_0)-1\big),$ the result will follow by linearity.\\
Starting first with $\bar g$, we see that $P_-(\bar g)\equiv 0$ and all $P_m(\bar g)$ are zero except
$$P_0(\bar g)=1, \; \mbox{ when } \bar g(y)=1,\mbox{ and }
P_1(\bar g)=1,\; \mbox{ when } \bar g(y)=y.$$
It remains then to handle $r$.
Since $1-\chi(y,s_0)=0$ for $|y|\le K s_0^\beta$, we see that
\begin{equation}\label{e56}
0\le 1-\chi(y,s_0)\leq
\Big(\f{|y|}{Ks_0^\beta}\Big)^i,\;\;0\le i\le 3,\mbox{ hence }
0\leq 1-\chi(2y,s_0)\leq \left({2|y|\over Ks_0^\beta}\right)^2,
\end{equation}
and since $K\ge 6$ (see \eqref{def:chi}),
\begin{equation}\label{rhoext}
\rho(y)\big(1-\chi(y,s_0)\big)\leq \sqrt{\rho(y)}\sqrt{\rho\Big(Ks_0^\beta\Big)}\leq C\mbox{e}^{-{K^2s_0^{2\beta}\over
8}}\sqrt{\rho(y)} \leq C\mbox{e}^{-s_0^{2\beta}}\sqrt{\rho(y)},
\end{equation}
and similarly,
\[
\rho(y)\big(1-\chi(y,s_0)\big)\leq C\mbox{e}^{-s_0^{2\beta}}\sqrt{\rho(y)}.
\]
Therefore,
$$|r(y)|\leq C(1+|y|)\Big({2|y|\over Ks_0^\beta}\Big)^2\leq
{C\over s_0^{2\beta}}(1+|y|^3)\mbox{ and }
|r_m|\leq C\mbox{e}^{-s_0^{2\beta}},\; m=0,1,2.$$ Hence, using the fact that $|h_m(y)|\leq
C(1+|y|),\; m=0,\,1,$ we get also
$$|r_-(y)|\leq {C\over s_0^{2\beta}}(1+|y|^3).$$ This concludes the proof of Lemma \ref{l41}.
\end{proof}

With Lemma  \ref{l41} at hand, we are ready to give the proof of Proposition \ref{prop:diddc}.

\begin{proof}[Proof of Proposition \ref{prop:diddc}] For simplicity we write $\psi$ instead of $\psi_{s_0,d_0,d_1}.$\\
(i) Using Lemma \ref{l41}, we see that
\begin{equation}\label{p0p1}
\psi_0=d_0\left({A\over s_0^{2\beta+1}}+O\left(\mbox{e}^{- s_0^{2\beta}}\right)\right)
\mbox{ and }
\psi_1=d_1\left({A\over s_0^{2\beta+1}}+O\left(\mbox{e}^{- s_0^{2\beta}}\right)\right),
\end{equation}
and the conclusion of item (i) follows directly.\\
(ii) The fact that $|\psi_m|\le \frac A{s_0^{2\beta}}$ for $m=0,1$ follows from item (i). Then, using Lemma \ref{l41} and linearity, we see that
\begin{equation}\label{p2}
\psi_e(y) \equiv 0,\;\;
\left\|\frac{\psi_-(y)}{1+|y|^3}\right\|_{L^\infty(\R)}\le \frac {CA}{s_0^{4\beta+1}}(|d_0|+|d_1|),\;\;
|\psi_2|\le C(|d_0|+|d_1|)e^{-s_0^{2\beta}}.
\end{equation}
Since
\begin{equation}\label{dm}
|d_m|\leq 2\mbox{ for }m=0,\,1
\end{equation}
 from item (i), recalling that $\gamma<4\beta+1$ from \eqref{gamma}, we get the conclusion of item (ii).\\
%
%
%
 (iii)  From (\ref{initial-data}), we have that
\begin{equation}
\label{e58}
\partial_y\psi(y)= d_1\f{A}{s_0^{2\beta+1}}\chi(2y,s_0)+\f{A}{s_0^{2\beta+1}}(d_0+d_1y)\chi_0'\Big(\f{2y}{Ks_0^\beta}\Big)\f{2}{Ks_0^\beta},
\end{equation}
where $\chi_0$ is defined by (\ref{def:chi}). Since
$\|z\chi_0'(z)\|_{L^\infty(\R)}$ and $\f{2}{Ks_0^\beta}$ are bounded,  then for
$s_0$ sufficiently large we have, from \eqref{dm} and the definition \eqref{gamma} of $\gamma,$
$$\|\partial_y\psi\|_{L^\infty(\R)}\leq C\frac{A}{s^{2\beta+1}_0}\leq \frac{1}{s^{\gamma-3\beta}_0}.$$
As for the estimate on $\partial_y \psi_-$, since  $\psi_e\equiv 0,$ we write from \eqref{qprojection}
\[
\psi_-(y) = \psi(y) - (\psi_0+\psi_1 y + \psi_2(y^2-1)),
\]
hence
\[
\partial_y \psi_-(y) =\partial_y \psi(y) - (\psi_1 + 2\psi_2y).
\]
Using \eqref{e58}, \eqref{p0p1} and \eqref{p2}, we see that
\begin{equation}\label{e59}
\partial_y\psi_-(y)
=d_1\f{A}{s_0^{2\beta+1}}
\Big(\chi(2y,s_0)-1\Big)+O\Big(\mbox{e}^{-s_0^{2\beta}}\Big)|y|
+\f{A}{s_0^{2\beta+1}}(d_0+d_1y)\chi_0'\Big(\f{2y}{Ks_0^\beta}\Big)\f{2}{Ks_0^\beta}.
\end{equation}
We remark now that
\begin{equation} \label{e57}|\chi_0'(z)|\leq C|z|^i,\;\; i=0,\;1,\;2,\;3
\end{equation}
(in fact, $\chi_0'(z)=0$ for $|z|\leq 1$ and it is bounded for $z\in \R$).
Using \eqref{dm}, (\ref{e56}), (\ref{e57}) and (\ref{e59}), we obtain
\begin{equation*}
|\partial_y\psi_-(y)|\leq
C\Big[\f{A}{s_0^{5\beta+1}}+\f{A}{s_0^{6\beta+1}}+\mbox{e}^{-s_0^{2\beta}}\Big](1+|y|^3)
\leq  \f{CA}{s_0^{5\beta+1}}(1+|y|^3)
\leq  \f{1}{s_0^{\gamma}}(1+|y|^3),
\end{equation*}
since $0<\gamma<5\beta+1.$ Thus, the last inequality in item (iii) follows. This concludes the proof of Proposition \ref{prop:diddc}.

% ...



\end{proof}

\subsubsection{Preliminary estimates on various terms of equation \eqref{qequation}}
In this step, we show that the rest term is trapped in $\vartheta_C(s)$ for some $C>0,$ provided that
$s$ is large. Then, assuming in addition that $v(s) \in \vartheta_A(s)$, we show that the nonlinear term in also trapped  in $\vartheta_C(s)$, and the potential term $V\vq,$ in $\vartheta_{CA}(s).$

This is our first statement.
\begin{lem}[Estimates on the rest term and the potential]
\label{p52}
There exits $s_{11}$ sufficiently
large such that for $s\geq s_{11},$ we have the following
\begin{itemize}
\item[(i)] $R\in \vartheta_C(s)$ and $|R_2(s)|\leq {C\over s^{4\beta}}$,
\item[(ii)] $\|V(s)\|_{L^\infty(\R)}\leq C,\; |V(y,s)|\leq C  {(1+|y|^2)\over s^{2\beta}},\; \forall\; y\in \R$.
\end{itemize}
where $C$ is a positive constant,   $V$ and $R$ are given by (\ref{eq:V}) and (\ref{RNeq}).
\end{lem}

\begin{rem}
 \label{4losange} {\rm As we stated in a remark following Theorem \ref{th1}, the particular value of $b$ we fixed in \eqref{bbupprof1} is natural from the formal approach in Section 2 above. In fact, it is crucial in the algebraic
 identity leading from \eqref{6losange} to \eqref{triangle2}. Indeed, with a different $b,$ we would have a larger $R_2 \sim {C\over s^{2\beta+1}}\gg {1\over s^{4\beta}},$ making the convergence of $\vq$ to zero in \eqref{eq:qt0}
 more difficult
 (and probably impossible) to obtain.}
 \end{rem}

\begin{rem}
 \label{imposiblepinf3} {\rm With some more work, we may show that:
 $$\left\|{R_-(y,s)\over 1+|y|^3}\right\|_{L^\infty(\R)}\sim {C\over s^{2\beta+1}},\; \mbox{ as }\; s\to \infty.$$
 This implies, in particular, that any attempt to adapt the powers of $1/s$ in the definition of $\vartheta_A(s)$ should respect
 \begin{equation}
  \label{necessaire1}
  \gamma \leq 2\beta+1
 \end{equation}
 where $\gamma$ is such that
 $$\left\|{\vq_-(y,s)\over 1+|y|^3}\right\|_{L^\infty(\R)}\leq {C\over s^{\gamma}},$$
on the one hand. On the other hand, bearing in mind that we need to evaluate $|\vq_-(y,s)|$ on the support of $\vq_b,$ that is, when $|y|\leq 2Ks^\beta,$ we see that
$$\mbox{ for } |y|\leq 2Ks^\beta,\; \mbox{ we have}\; |\vq_-(y,s)|\leq {8AK^3\over s^{\gamma-3\beta}},$$
and the right hand side of the last inequality goes to zero if and only if
\begin{equation}
 \label{necessaire2}
 \gamma>3\beta.
\end{equation}
From \eqref{necessaire1} and \eqref{necessaire2}, we see that $3\beta<\gamma\leq 2\beta+1,$ which yields the natural condition
$$\beta<1,\; \mbox{ i.e.}\; p>3.$$
Presumably, our strategy based on the shrinking set
$\vartheta_A(s),$ in the same style as \cite{BK} and
\cite{MZsbupdmj97} breaks down when $\beta\geq 1.$ However, we are
not saying that Theorem \ref{th1} is not true for $1<p\leq 3.$
Perhaps a substantial adaptation of the method of \cite{BK} and
\cite{MZsbupdmj97}, or some other strategy, may give the result. The
question remain open when
$$1<p\leq 3.$$
 }\end{rem}

Before proving Lemma \ref{p52}, let us state and prove the following lemma, where we make an expansion of $R(y,s)$:
\begin{lem}[Bound and expansion of the rest term]
\label{devr1}
For $s$ large enough, we have
\begin{equation}
\label{e20152}
\|R(s)\|_{L^\infty(\R)}\leq {C\over s},
\end{equation}
and
\begin{eqnarray}
\nonumber
R(y,s)&=&{a\over s^{2\beta}}-{{2b\kappa\over (p-1)^2}{1\over s^{2\beta}}+{2\beta a\over s^{2\beta+1}}+{2b^2p\kappa\over (p-1)^4}{1\over s^{2\beta}}z^2+{(2b)^2p\kappa\over (p-1)^4}}{1\over s^{2\beta}}z^2\\ \label{developmentR} &&-
{2\beta b\kappa\over (p-1)^2}{1\over s}z^2+\mu \left({2b\kappa\over (p-1)^2}\right)^q{1\over s^{q\beta}}|z|^q+O({z^4\over s})+O({1\over s^{4\beta}})+O({|z|^{q+2}\over s^{q\beta}}),
\end{eqnarray}
where $z={y\over s^\beta}\in \R.$
\end{lem}
\begin{proof}
By the definitions of $\varphi$ and $\vpz,$ we have
\begin{align}
\label{e20151}
\partial_y\varphi(y,s)&=-{2b\over
(p-1)s^{\beta}}\left({y\over
s^{\beta}}\right)\left({\vpz}({y\over s^{\beta}})\right)^p, \mbox{ hence } \|\partial_y\varphi(s)\|_\infty\leq {C\over s^{\beta}},\\
\label{e20154}\partial_s\varphi(y,s)&={2\beta b\over
(p-1)s}\left({y\over
s^{\beta}}\right)^2\left(\vpz({y\over s^{\beta}})\right)^p-{2\beta a\over
s^{2\beta+1}}, \mbox{ hence } \|\partial_s\varphi(s)\|_\infty\leq {C\over s},\\
\label{devlopderisecond}
\partial^2_y\varphi(y,s)&={2b\over (p-1)s^{2\beta}}\left[-\left(\vpz(z)\right)^p+{2bp\over
p-1}z^2\vpz^{2p-1}(z)\right],
\mbox{ hence }  \|\partial^2_y\varphi(s)\|_\infty\leq {C\over
s^{2\beta}}.
\end{align}
 On the other hand, since we have from \eqref{ordre0profil}, $$-{1\over 2}{y\over
s^\beta}\vpz'({y\over s^\beta})+\left(\vpz({y\over
s^\beta})\right)^p-{1\over p-1}\vpz({y\over s^\beta})=0,$$ we write
\begin{eqnarray}\nonumber-{1\over 2}y\partial_y\varphi-{1\over p-1}\varphi+\varphi^p& = & -{1\over 2}{y\over
s^\beta}\left[\vpz\right]'({y\over
s^\beta})-{1\over p-1}\vpz({y\over
s^\beta})-{a\over (p-1)s^{2\beta}}\\ \nonumber  & &+\varphi^p-\left[\vpz({y\over s^\beta})\right]^p+\left[\vpz({y\over s^\beta})\right]^p\\ \label{196}& = & \varphi^p-\left[\vpz({y\over s^\beta})\right]^p-{a\over (p-1)s^{2\beta}}.
\end{eqnarray}
By Lipschitz property, we have that
$$|\varphi^p-\left(\vpz\right)^p|\leq {C\over s^{2\beta}}.$$
Hence
$$\|-{1\over 2}y\partial_y\varphi-{1\over p-1}\varphi+\varphi^p\|_{L^\infty(\R)}\leq {C\over s^{2\beta}}.$$
Since $1<{p\over p-1}=q\beta<2\beta,$ we see that by definition \eqref{RNeq} of $R$ that \eqref{e20152} holds.

\medskip

Now, as for the expansion \eqref{developmentR}, it simply follows from Taylor expansions, derived from (\ref{e20151}), (\ref{e20154}), (\ref{devlopderisecond})  and (\ref{196}), for all $s\geq 1$ and $y\in \R$:
\begin{align}
\partial^2_y\varphi(y,s)&={2b\over (p-1)s^{2\beta}}\left[-{\kappa \over p-1}\left(1-{bp\over (p-1)^2}z^2\right)+{2bp\kappa\over
(p-1)^3}z^2+O(z^4)\right],\\
-{1\over 2}y\partial_y\varphi-{1\over p-1}\varphi+\varphi^p&={pa\over (p-1)s^{2\beta}}\left(1-{b\over p-1}z^2+O(z^4)\right)-{a\over (p-1)s^{2\beta}}+O\left({1\over s^{4\beta}}\right),\\
\partial_s\varphi(y,s)&={2\beta b\over (p-1)s}z^2\left({\kappa \over p-1}+O(z^2)\right)-{2\beta a \over s^{2\beta+1}},\\
|\partial_y \varphi|^q&=\left({2b\over (p-1)s^\beta}\right)^ q\left({\kappa \over p-1}\right)^q |z|^q\left(1+O(z^2)\right).
\end{align}
This concludes the proof of Lemma \ref{devr1}.
\end{proof}


With Lemma \ref{devr1}, we are ready to prove Lemma \ref{p52}.
\begin{proof}[Proof of Lemma \ref{p52}] $ $

(i) Following the decomposition \eqref{qprojection}, we write $R$ as
$$R=R\chi+R(1-\chi)=\left(\sum_{m=0}^{2}R_mh_m+R_-\right)+R_e.$$
Since $R$ is symmetric with respect to $y$, we have $R_1=0.$\\
Furthermore, inequality \eqref{e20152} in the previous lemma implies in particular  that
$$\|R_e(s)\|_{L^\infty(\R)}\leq {C\over s}\leq {C\over s^{\gamma-3\beta}},$$ by definition of
 $\gamma$ given by \eqref{gamma}.\\
As for $R_0,$ using \eqref{developmentR} and the fact that $2(q+1)\beta=4\beta+1,$ we write
\begin{equation*}
R_0=\int_{\R} R\chi \rho=\int_{\R} R \rho+\int_{\R} R(\chi-1) \rho
 =  \left[a-{2b\kappa\over (p-1)^2}\right]{1\over s^{2\beta}}+O({1\over s^{2\beta+1}})
 =  O({1\over s^{2\beta+1}}),
\end{equation*}
from the choice of $a$ made in \eqref{aeq}, therefore,
\begin{equation}
\label{triangle1}
|R_0|\leq {C\over s^{2\beta+1}}.
\end{equation}
Now, considering $R_2$ and using the fact that
$$\int h_2\rho=0,\; \int k_2y^2\rho=1,\; 2(q+1)\beta=4\beta+1,$$
we write from \eqref{developmentR} together with \eqref{rhoext},
\begin{eqnarray}
 \nonumber
R_2&=&\int_{\R}R\chi k_2\rho=\int_{\R}R k_2\rho+\int_\R R(\chi-1)k_2\rho\\ \nonumber
 &=&  \int_{\R}R k_2\rho+O\left(e^{-s^{2\beta}}\right)\\ \label{6losange}
&=& \left[-{2b\beta\kappa\over (p-1)^2}+\mu \left({2b\kappa\over (p-1)^2}\right)^q\int|y|^qk_2\rho\right]{1\over s^{2\beta+1}}+O\left({1\over s^{4\beta}}\right).
\end{eqnarray}
Since
$${2\beta b\kappa\over (p-1)^2}\|h_2\|_{L^2_\rho}^2=\mu \Big({2b\kappa\over
(p-1)^2}\Big)^q\int_{\R}|y|^qh_2\rho dy,$$  from the choice of $b$ in \eqref{equb}, we  have that
\begin{equation}
\label{triangle2}
|R_2|\leq
{C\over s^{4\beta}}\leq {\sqrt{C}\over s^{4\beta-1}}.
\end{equation}
It remains now to bound $R_-(y,s)$.
From \eqref{developmentR} in Lemma \ref{devr1}
 and the choice of $a$ in \eqref{aeq} and $b$ in \eqref{equb}, we see that
$$R(y,s)=-{2\beta b\kappa\over (p-1)^2}{h_2(y)\over s^{2\beta+1}}+\mu \left({2b\kappa\over (p-1)^2}\right)^q{|y|^q\over s^{2\beta+1}}+{6b^2p\kappa\over (p-1)^4}{y^2\over s^{4\beta}}+O({y^4\over s^{4\beta+1}})+O({1\over s^{4\beta}}).$$
Using that $|y|\leq 2Ks^\beta$ on the support of $\chi,$ we see that
$$|R\chi |(y,s)\leq C{1\over s^{2\beta+1}}(1+|y|^3).$$
Using \eqref{triangle1} and \eqref{triangle2}, we see that
\begin{eqnarray*}
|R_-|& =& |R\chi-\left(R_0+R_2h_2\right)|\\ &\leq & |\chi R|+|R_0| +C|R_2|(1+|y|^2)\\ &\leq &  {C\over s^{2\beta+1}}(1+|y|^3)\\ &\leq &  {C\over s^{\gamma}}(1+|y|^3),
\end{eqnarray*}
by the hypotheses \eqref{gamma} on $\gamma.$ This concludes the proof of item (i) of the lemma.

\medskip

(ii) Since $b>0$, the first statement follows By \eqref{eq:V} and \eqref{feq}. For the second statement, we write, by definition (\ref{eq:V}) of $V$ and
a Taylor expansion, (remark that $p>3$):
\begin{eqnarray*}
 V(y,s)& = & p\left(\vpz\left({y\over s^\beta}\right)+{a\over s^{2\beta}}\right)^{p-1}-{p\over p-1}\\
 & = & p\left(\vpz\left({y\over s^\beta}\right)\right)^{p-1}-{p\over p-1}+O\left({1\over s^{2\beta}}\right)\\
 & = & p[\vpz(0)]^{p-1}-{p\over p-1}+O\left({1+|y|^2\over s^{2\beta}}\right).
\end{eqnarray*}

Since $[\vpz(0)]^{p-1}=1/(p-1),$ the second statement follows. This concludes the proof of Lemma \ref{p52}.

\end{proof}
\begin{lem}
\label{l52} Let $V,\; B$ and $G$ be given by \eqref{eq:V}, (\ref{eq:B}),  (\ref{RNeq}) and \eqref{G}.  Then, for all $A\geq 1,$ there exits
$s_{12},$ sufficiently large, such that for all $s\geq s_{12},$ if $\vq \in \vartheta_A(s)$, then we have
the following:
\begin{itemize}
\item[(i)] $\|V\vq(s)\|_{L^\infty(\R)}\leq
{CA^2\over s^{\gamma-3\beta}},$ $|V\vq(y,s)|\leq {CA^2\over
s^{\gamma-\beta}}(1+|y|^2),\, \forall\; y\in \R,$
$$|(V\vq)_m|\leq {1\over s^{2\beta+1}},\; \mbox{ for }\; m=0,\;1,\; |(V\vq)_2|\leq {1\over s^{4\beta}},$$
$$|(V\vq)_-(y,s)|\leq {CA\over s^\gamma}(1+|y|^3)\; \mbox{ and }\; \|(V\vq)_e\|_{L^\infty(\R)}\leq {CA^2\over s^{\gamma-3\beta}}.$$
In particular, $V\vq\in \vartheta_{CA}(s).$
\item[(ii)]  $B(\vq)\in \vartheta_C(s)$ and $|(Bv)_2(s)|\leq {C\over s^{4\beta}}.$
\item[(iii)]Furthermore, if $\|\partial_y\vq(s)\|_\infty \leq  {CA^2\over
s^{\gamma-3\beta}}$ and $$|\partial_y\vq(y,s)|\leq
\frac{CA}{s^{2\beta+1}}(1+|y|)+\frac{C\sqrt{A}}{s^{4\beta-1}}(1+|y|^2)+
\frac{CA^2}{s^{\gamma}}(1+|y|^3),\;\forall\;y \in \R,$$ then
 ${G}\in \vartheta_{C}(s),$ for some positive constant $C.$
\end{itemize}
\end{lem}
\begin{proof}
(i) Since $\vq \in \vartheta_A(s), $ we have from Proposition
\ref{prop:V-def-prop}:
$$\|\vq (s)\|_{L^\infty(\R)}\leq C{A^2\over s^{\gamma-3\beta}}.$$
Then, using item  (ii) of Lemma \ref{p52}, we get,
\begin{equation}
 \label{280120153}
\|V\vq (s)\|_{L^\infty(\R)}\leq
\|V(s)\|_{L^\infty(\R)}\|\vq (s)\|_{L^\infty(\R)}\leq C{A^2\over
s^{\gamma-3\beta}},\end{equation}
and
$$|V\vq|(y,s)\leq  \|\vq
(s)\|_{L^\infty(\R)}|V|(y,s)\leq C{A^2\over
s^{\gamma-\beta}}(1+|y|^2).$$
Furthermore, using item (iii) of Proposition \ref{prop:V-def-prop}  and Definition  \ref{prop:V-def} of $\vartheta_{A}(s),$ we see that
\begin{eqnarray}
\nonumber
 \hspace{-2cm} |(V\vq)_b(y,s)|=|V\vq_b(y,s)|\hspace{-0,2cm} &\leq & \hspace{-0,2cm} {C(1+|y|^2)\over s^{2\beta}}\left(\sum_{m=0}^2|\vq_m(s)|(1+|y|)^m+\left\|{\vq_-(y,s)\over 1+|y|^3}\right\|_{L^\infty(\R)}\hspace{-1cm}(1+|y|^3)\right)\\ &\leq & \label{etoile1}
 {CA\over s^{4\beta+1}}(1+|y|^3)+{C\sqrt{A}\over s^{6\beta-1}}(1+|y|^4)+{CA\over s^{\gamma+2\beta}}(1+|y|^5).
\end{eqnarray}
By definition of $(V\vq)_m$ we see that
\begin{eqnarray*}
|(V\vq)_m(s)|& = & \left|\int_R k_m(y)(V\vq)_b\rho(y)dy\right|\\
 & \leq & {CA\over s^{4\beta+1}}+{C\sqrt{A}\over s^{6\beta-1}}+{CA\over s^{\gamma+2\beta}}\\
& \leq & {1\over s^{4\beta}},
\end{eqnarray*}
since ${1\over 2}<\beta<1$ and $\gamma$ satisfies \eqref{gamma}.\\
As for $(V\vq)_-,$ noting that $|y|\leq 2Ks^\beta$ on the support of $(V\vq)_b,$ we write from \eqref{etoile1}:
$${|(V\vq)_b|(y,s)\over 1+|y|^3}\leq {CA\over s^{4\beta+1}}+{C\sqrt{A}\over s^{5\beta-1}}+{CA\over s^{\gamma}}\leq {C_1A\over s^\gamma},$$
since $\gamma < \min(5\beta-1,2\beta+1).$
Therefore,
${|(V\vq)_-|(y,s)\over 1+|y|^3}\leq {CA\over s^\gamma},$ which is the desired estimate for $(V\vq)_-.$ As for the estimate on $\|(V\vq)_e\|_{L^\infty(\R)},$
it follows from \eqref{280120153}.

(ii) From a Taylor expansion,  we have
\begin{equation}
 \label{e20153}
|B(\vq)|\leq C|\vq|^2.\end{equation}
Since $\vq \in \vartheta_A(s)$, from item (i) in Proposition \ref{prop:V-def-prop},
 we have
$$\|B(\vq)_e\|_\infty\leq \|B(\vq)\|_{L^\infty(R)}\leq C\|\vq\|_{L^\infty(\R)}^2\leq
C{A^4\over s^{2(\gamma-3\beta)}}\leq {1\over s^{\gamma-3\beta}},$$
for sufficiently large $s.$ Moreover, we have that,
\begin{eqnarray}
\nonumber
|(B(\vq))_b(y,s)|=|\chi(y,s)B(\vq)(y,s)|&\leq & C|\vq(y,s)|^2 \\ \nonumber
&\leq & C\Big( \sum_{m=0}^{2}|\vq_m|^2h_m^2+|\vq_-(y,s)|^2+|\vq_e(y,s)|^2\Big)\\ \nonumber
&\leq & C\Big({A^2\over s^{4\beta+2}}(1+|y|^2)+{A\over
s^{8\beta-2}}(1+|y|^4)+\\ \label{4275}
&& {A^2\over s^{2\gamma}}(1+|y|^6)\Big)\mathbf{1}_{|y|\leq 2Ks^\beta}+\mathbf{1}_{|y|>Ks^\beta}{A^4\over
s^{2\gamma-6\beta}} \Big)
\\ \nonumber
&\leq &C\Big({A^2\over s^{4\beta+2}}+{A\over
s^{7\beta-2}}+{A^4\over s^{2\gamma-3\beta}}\Big)(1+|y|^3),
\end{eqnarray}
where $\mathbf{1}_X$ is the   characteristic function of a set $X.$\\
Hence, using \eqref{4275}, we write by definition of $(B(v))_m,$  $$|B(\vq)_m|\leq C\Big({A^2\over s^{4\beta+2}}+{A\over
s^{8\beta-2}}+{A^2\over s^{2\gamma}}\Big).$$
Therefore, by the conditions on $\gamma$ given by (\ref{gamma}) and
since $\beta>1/2,$ we see that $$|B(\vq)_m|\leq {1\over s^{2\beta+1}},\;
m=0,\;1\;\mbox{ and }\; |B(\vq)_2|\leq  {C\over s^{4\beta}}\leq {1\over s^{4\beta-1}}.$$
Furthermore, by the expression of $\gamma$ given by (\ref{gamma}), and
since $\beta>1/2,$ we have
\begin{eqnarray*}
|B(\vq)_-(y,s)|& = &|B(\vq)_b(y,s)-\sum_{m=0}^{2}(B\vq)_mh_m|\\
&\leq & C\Big({A^2\over s^{4\beta+2}}+{A\over
s^{7\beta-2}}+{A^4\over s^{2\gamma-3\beta}}\Big)(1+|y|^3)\\
&\leq & {1\over s^{\gamma}}(1+|y|^3),
\end{eqnarray*}
for $s$ sufficiently large. This finishes the proof of item (ii).

\medskip

(iii) Using the inequality
$$\left||x+x'|^q-|x|^q\right|\leq C\left(|x|^{q-1}|x'|+|x'|^q\right),\; 
\forall x\in
\R,\; x'\in \R,$$ we deduce that
\begin{equation}
\label{estG} |G(y,s)|\leq C \left(|\partial_y
\varphi|^{q-1}|\partial_y\vq|+|\partial_y\vq|^q\right).
\end{equation}
Since
$$\|\partial_y\vq(s)\|_\infty \leq {CA^2\over
s^{\gamma-3\beta}}\; \mbox{ and }\;
 \|\partial_y\varphi(s)\|_\infty \leq {C\over s^{\beta}},$$
 it follows that
 \begin{eqnarray*}
 \|G(s)\|_\infty &\leq & {1\over
s^{\beta(q-1)}} {CA^2\over s^{\gamma-3\beta}}+{CA^{2q}\over
s^{q(\gamma-3\beta)}}\\
&\leq  & {1\over s^{\gamma-3\beta}},
\end{eqnarray*}
because $q>1.$ Hence
\begin{equation}
\label{estGe} \|G_e(s)\|_\infty \leq   {1\over
s^{\gamma-3\beta}}.
\end{equation}
On the other hand, from the fact that
$$|\partial_y\varphi(y,s)|\leq C{|y|\over s^{2\beta}},\;\forall\;y \in \R,$$
and $$|\partial_y\vq(y,s)|\leq
\frac{CA}{s^{2\beta+1}}(1+|y|)+\frac{C\sqrt{A}}{s^{4\beta-1}}(1+|y|^2)+
\frac{CA^2}{s^{\gamma}}(1+|y|^3),\;\forall\;y \in \R,$$ using
\eqref{estG} and the identity
\begin{equation}
 \label{4295}
2\beta(q-1)=1,
\end{equation}
we deduce that
$$|G(y,s)|\leq C\left({|y|^{q-1}\over
s}|\partial_y\vq|+|\partial_y\vq|^q\right),$$
hence,
\begin{eqnarray}
\nonumber
|G(y,s)|&\leq & {CA\over
s^{2\beta+2}}|y|^{q-1}(1+|y|)+{C\sqrt{A}\over
s^{4\beta}}|y|^{q-1}(1+|y|^2)\\ \nonumber &&+{CA^2\over
s^{\gamma+1}}|y|^{q-1}(1+|y|^3)+{CA^q\over s^{q(2\beta+1)}}(1+|y|^q)\\ \label{21A}
&&+{CA^{q/2}\over s^{q(4\beta-1)}}(1+|y|^{2q})+{CA^{2q}\over
s^{q\gamma}}(1+|y|^{3q}).
\end{eqnarray}
Therefore, from the conditions on
$\gamma$ given by (\ref{gamma}), we have
$$
|G_m(s)|\leq  {CA\over s^{2\beta+2}}+{C\sqrt{A}\over
s^{4\beta}}+{CA^2\over s^{\gamma+1}}+{CA^q\over
s^{q(2\beta+1)}}+{CA^{q/2}\over s^{q(4\beta-1)}}+{CA^{2q}\over
s^{q\gamma}}.$$
Noting from \eqref{4295} that $2\beta+1=2\beta q<q(4\beta-1)$ and $2\beta+1=2\beta q<3\beta q<\gamma q,$ and using the fact that $\beta>1/2,\, \gamma>3\beta,\, 1<q<2,$ we get
$$|G_m(s)|\leq  {1\over s^{2\beta+1}}\; \mbox{ for }\; m=0, 1.$$
Since $\beta<3/2$ and $\gamma>4\beta-1$, we get
$$|G_2(s)|\leq {1\over s^{4\beta-1}}.$$
We now turn to the estimate on $G_-.$\\
Using \eqref{21A} and the fact that $|y|\leq 2Ks^\beta$on the support of $G_-,$ for $G_b$ we write from \eqref{4295} and the fact that $q>3/2$ (which follows from the fact that $p>3$ and $q=2p/(p+1))\!:$
\begin{eqnarray*}
 {|G_b(y,s)|\over 1+|y|^3} & \leq &  {CA\over s^{2\beta+2}}+{C\sqrt{A}\over s^{4\beta}}+{CA^2\over s^{\gamma+{1\over 2}}} \\ && + {CA^q\over s^{q(2\beta+1)}}+{CA^{q/2}\over s^{q(4\beta-1)-(2q-3)\beta}}+
 {CA^{2q}\over s^{q\gamma-3(q-1)\beta}}.
\end{eqnarray*}
Noting from the conditions \eqref{gamma} on $\gamma$ and \eqref{4295} that
\begin{equation}
 \label{4345}
2\beta+2>2\beta+1>\gamma;\; 4\beta>5\beta-1>\gamma;\; q(2\beta+1)>2\beta+1>\gamma;
\end{equation}
$$q(4\beta-1)-(2q-3)\beta=2q\beta-q+3\beta=5\beta+1-q>5\beta-1>\gamma;$$
$$q\gamma-3(q-1)\beta=q\gamma-{3\over 2}=(q-1)\gamma-{3\over 2}+\gamma={\gamma\over 2\beta}-{3\over 2}+\gamma={\gamma-3\beta\over 2\beta}+\gamma>\gamma,$$
we see that
$${|G_b(y,s)|\over 1+|y|^3} \leq {1\over s^\gamma},$$
hence,
$$|G_m(s)|\leq {C\over s^\gamma},\; m=0,\; 1,\; 2,$$ and
$$|G_-(y,s)|\leq {C\over
s^{\gamma}}(1+|y|^3),\; \forall \; y\in \R.$$
 This finishes the proof of Lemma \ref{l52}.
\end{proof}

\subsubsection{Parabolic regularity}\label{subsecpr}
In this subsection, we prove the parabolic regularity results and prove \eqref{prop:regu-parab-q-equationbis}. To do so:
\begin{itemize}
 \item[-] We first give some linear parabolic regularity estimates on the linear operator $\mathcal{L}$ defined in \ref{eq:Operator:L}. See Lemma \ref{l53} below.
 \item[-] Then, since we aim at obtaining some fine estimates on $\partial_y\vq_-$ if $\vq(s)\in \vartheta_A(s),$ we will write down the equation satisfied by $\vq_-$ and bound its source term. See Lemma \ref{l54} below.
 \item[-] Finally, using the linear estimate along with the preliminary estimates given in the previous subsection, we derive the intended regularity estimates for the full equation \eqref{qequation}.
\end{itemize}

Let us first give some linear regularity estimates:
\begin{lem}[Properties of the semigroups $e^{\theta \mathcal{L}}$] \label{l53} The Kernel $ e^{\theta \mathcal{L}}(y,x)$ of the semi-group $e^{\theta \mathcal{L}}$ is given by:
  \begin{equation}\label{eq:semigroup:theta-L-case1}
   e^{\theta \mathcal{L}}(y,x)=
\frac{e^\theta}{\sqrt{4\pi(1-e^{-\theta})}}\exp{[-\frac{(ye^{-\theta/2}-x)^2}{
4(1-e^{-\theta})}]}.
\end{equation}
for all  $\theta >0$, and $e^{\theta \mathcal{L}}$ is defined by
\begin{equation}\label{eq:semigroup:theta-L-case1-def}
   e^{\theta \mathcal{L}}r(y)=\int_\R
e^{\theta \mathcal{L}}(y,x)r(x)dx.
\end{equation}
We also have the following:
\begin{itemize}
\item[(i)] If $r_1\leq r_2$
then $e^{\theta \mathcal{L}}r_1\leq e^{\theta \mathcal{L}}r_2.$
\item[(ii)]$\|\nabla (e^{\theta \mathcal{L}}r)\|_{L^{\infty}(\R)}\leq
Ce^{\frac{\theta }{2}} \|\nabla r\|_{L^\infty(\R)},\; r\in
W^{1,\infty}(\R).$
\item[(iii)] $\|\nabla (e^{\theta \mathcal{L}}r)\|_{L^{\infty}(\R)}\leq
   \frac{Ce^{\frac{\theta
   }{2}}}{\sqrt{1-e^{-\theta}}}\|r\|_{L^\infty(\R)},\, r\in L^{\infty}(\R).$
\item[(iv)] If $|r(x)|\leq \eta (1+|x|^m),\; \forall \; x\in \R,$ then
  $|e^{\theta \mathcal{L}}r(y)|\leq C\eta e^{\theta} (1+|y|^m),\; \forall \; y\in
  \R.$
\item[(v)] If $|\nabla r(x)|\leq \eta (1+|x|^m),\; \forall \; x\in
\R,$
  then
   $|\nabla (e^{\theta \mathcal{L}}r)(y)|\leq C\eta e^{\frac{\theta }{2}} (1+|y|^m),\; \forall \; y\in
   \R,$
\item[(vi)] If $|r(x)|\leq  \eta (1+|x|^m),\; \forall \; x\in
\R,$
  then
   $|\nabla (e^{\theta \mathcal{L}}r)(y)|\leq C\eta {e^{\frac{\theta }{2}}\over \sqrt{1-e^{-\theta}}}  (1+|y|^m),\; \forall \; y\in
   \R$,
 \end{itemize}
where $C$ is a positive constant and $m\geq 0.$
\end{lem}

\begin{proof} The expressions of $e^{\theta \mathcal{L}}(y,x)$
and $e^{\theta \mathcal{L}}$ are given in \cite[Formula (44), p.
554]{BK}. See also \cite{S}.

(i) Follows by the positivity of the kernel. (ii) and (iii) Follow by simple calculations using
(\ref{eq:semigroup:theta-L-case1}) and
(\ref{eq:semigroup:theta-L-case1-def}) so we omit the proof. (iv) Follows from (\ref{eq:semigroup:theta-L-case1}) and (\ref{eq:semigroup:theta-L-case1-def}). See also \cite[Lemma 4, p.555]{BK}. (v)-(vi) follow also by simple calculations.
\end{proof}

Now, we write down in the following, the equation satisfied by $\vq_-$ and estimate its source term:
\begin{lem}[Equation satisfied by $v_-$]
\label{l54}
For all $A\geq 1,$ there exits
$s_{13}(A)$ sufficiently large such that for $s\geq s_{13},$ if $\vq \in \vartheta_A(s)$,
\begin{equation}
\label{4triangle}
 \|\partial_y\vq(s)\|_\infty \leq {CA^2\over
s^{\gamma-3\beta}}
\end{equation}
 and
 \begin{equation}
 \label{6triangle}
|\partial_y\vq(y,s)|\leq
\frac{CA}{s^{2\beta+1}}(1+|y|)+\frac{C\sqrt{A}}{s^{4\beta-1}}(1+|y|^2)+
\frac{CA^2}{s^{\gamma}}(1+|y|^3),\;\forall\;y \in \R,
\end{equation}
then we have
\begin{equation}
\label{fbare}
\partial_s\vq_-=\mathcal{L}\vq_-+\overline{F}(y,s),
\end{equation}
with
\begin{equation}
\label{estimfbare}|\overline{F}(y,s)|\leq C{A^2\over s^\gamma}\left(1+|y|^3\right),\; \forall\; y\in \R.\end{equation}
\end{lem}
\begin{proof}
From equation \eqref{qequation}, we have
$$\left(\partial_s \vq\right)_-=\left(\mathcal{L}\vq\right)_-+\left(V\vq\right)_- +\left(B(\vq)\right)_-+\left(G(\vq)\right)_-+\left({R}(y,s)\right)_-.$$
Using the hypotheses of the Lemma, the results obtained in Lemmas \ref{p52} and \ref{l52}, we deduce that, for $s$ sufficiently large,
\begin{equation}
\label{(1)}
\left|\left(\partial_s \vq\right)_--\left(\mathcal{L}\vq\right)_-\right|\leq {CA\over s^\gamma}\left(1+|y|^3\right),\; \forall\; y\in \R.\end{equation}
Furthermore,
we have from \eqref{def:q:projbis} and \eqref{def:q_b}
\begin{equation}
\label{4323}
\vq_b=\chi \vq=\sum_{m=0}^2\vq_mh_m+\vq_-,
\end{equation}
hence
$$(\partial_s\chi)\vq+\chi\partial_s\vq=\sum_{m=0}^2\vq_m'h_m+\partial_s\vq_-,$$
on the one hand. On the other hand, applying \eqref{def:q_b} to $\partial_s\vq $, we have
$$\chi \partial_s\vq=\sum_{m=0}^2(\partial_s\vq)_mh_m+(\partial_s\vq)_-.$$
Then, by taking the difference of the last two
identities,
we get
\begin{equation}
 \label{1triangle}
\left|\partial_s\vq_--\left(\partial_s \vq\right)_-\right|\leq |\vq||\partial_s\chi|+C\left(\sum_{m=0}^2|\vq_m'-(\partial_s\vq)_m|\right)(1+|y|^2).
\end{equation}
But we have by definition
$$\vq_m'={d\over ds}\left(\int_\R\chi k_m\vq\rho\right)=\int_\R\chi k_m\partial_s\vq\rho+\int_\R\partial_s\chi k_m\vq\rho$$
and
$$\left(\partial_s\vq\right)_m=\int_\R\chi k_m\partial_s\vq\rho.$$
Then, since $\|\vq(s)\|_{L^\infty(\R)}\leq {CA^2\over
s^{\gamma-3\beta}}$ by item (i) of Proposition
\ref{prop:V-def-prop}, and
\begin{equation}
 \label{4325}
|\partial_s\chi|\leq {C\over s}\mathbf{1}_{\{Ks^\beta<|y|<2Ks^\beta\}},
\end{equation}
where $\mathbf{1}_{\{Ks^\beta<|y|<2Ks^\beta\}}$ is the  characteristic function of the set $\{Ks^\beta<|y|<2Ks^\beta\}$, we deduce that
\begin{equation}
 \label{2triangle}
\left|\vq_m'-\left(\partial_s\vq\right)_m\right|=\left|\int_\R\partial_s\chi h_m\vq\rho\right|\leq {CA^2\over s^{\gamma-3\beta+1}}e^{- s^{2\beta}}\leq {1\over s^{4\beta}}\leq {1\over s^{2\beta+1}} \leq {1\over s^\gamma},\end{equation}
for $s$ large enough. Now, using \eqref{4325},
we see that
$$|\partial_s\chi|\leq C{|y|^i\over s^{1+i\beta}},\; i=0,\;1,\;2.$$
Therefore, since $\vq\in \vartheta_A(s),$ by Proposition \ref{prop:V-def-prop}, part (iii) and by using the conditions on $\gamma$ given by (\ref{gamma}), we get
\begin{eqnarray}
 \nonumber
|\vq||\partial_s\chi|& \leq & C\left({A\over s^{4\beta+2}}+{\sqrt{A}\over s^{5\beta}}+{A^2\over s^{\gamma+1}}\right)(1+|y|^3)\\ \label{3triangle}
&\leq & C{A^2\over s^{\gamma+1}}(1+|y|^3),
\end{eqnarray}
and we deduce, from \eqref{1triangle}, \eqref{2triangle} and \eqref{3triangle} that
\begin{equation}
 \label{(2)}
\left|\partial_s\vq_--\left(\partial_s \vq\right)_-\right|\leq C{A^2\over s^{\gamma+1}}(1+|y|^3),\; \forall \; y\in \R.
\end{equation}
Since $\mathcal{L}(h_m)=\left(1-{m\over 2}\right)h_m,$ we write from \eqref{4323} and \eqref{def:q_b} applied to $\mathcal{L}\vq,$
$$\mathcal{L}(\chi\vq)=\sum_{m=0}^2 \vq_m \left(1-{m\over 2}\right)h_m+\mathcal{L}\vq_-,$$
$$\chi\mathcal{L}\vq=\sum_{m=0}^2 \left(\mathcal{L}\vq\right)_mh_m+\left(\mathcal{L}\vq\right)_-.$$
Therefore, we deduce
\begin{equation}
 \label{7triangle}
\left|\mathcal{L}\vq_--\left(\mathcal{L}\vq\right)_-\right|\leq \left|\mathcal{L}(\chi\vq)-\chi\mathcal{L}\vq\right|+\left(\sum_{m=0}^2\left|\vq_m \left(1-{m\over 2}\right)-\left(\mathcal{L}\vq\right)_m\right|\right)(1+|y|^2).\end{equation}
Since $\mathcal{L}$ is self-adjoint, we have
$$\left(\mathcal{L}\vq\right)_m=\int_\R\chi k_m \mathcal{L}\vq \rho= \int_\R \mathcal{L}(\chi k_m) \vq \rho,$$
then
\begin{eqnarray*}
 \left|\vq_m \left(1-{m\over 2}\right)-\left(\mathcal{L}\vq\right)_m\right|& = &\left|\int_\R\chi (\mathcal{L}k_m)\vq\rho-\int_\R \mathcal{L}(\chi k_m)\vq \rho\right|\\
& = & \left|\int_\R\left(\chi\left(\mathcal{L}k_m\right)-\mathcal{L}\left(\chi k_m\right)\right)\vq\rho\right|.
\end{eqnarray*}
Now, using the expression of $\mathcal{L}$ given by (\ref{eq:Operator:L}), we have for any regular function $r$,
\begin{equation}
 \label{4335}
\left|\mathcal{L}(\chi r)-\chi\mathcal{L}r\right|\leq |r|\left({1\over 2}|y||\partial_y\chi|+|\partial_y^2\chi|\right)+2|\partial_y\chi||\partial_yr|.
\end{equation}
Since
\begin{equation}
 \label{5triangle}
s^\beta|\partial_y\chi|+s^{2\beta}|\partial_y^2\chi|\leq \mathbf{1}_{\{Ks^\beta<|y|<2Ks^\beta\}},
\end{equation} using item (i) of Proposition \ref{prop:V-def-prop}, and \eqref{4triangle}  with $r=\chi k_m,$ we get,
\begin{equation}
 \label{8triangle}
 \left|\vq_m \left(1-{m\over 2}\right)-\left(\mathcal{L}\vq\right)_m\right|\leq {CA^2\over s^{\gamma-3\beta}}e^{-s^{2\beta}}
\leq {1\over s^{4\beta}}\leq {1\over s^{2\beta+1}}\leq {1\over s^\gamma},
 \end{equation}
for $s$ sufficiently large.
On the other hand, since
$$|y||\partial_y\chi|+|\partial_y\chi|+|\partial_y^2\chi|\leq C{|y|^i\over s^{i\beta}},\; i=0,\; 1,\; 2,$$
from \eqref{5triangle}, using the hypotheses on $\vq$ and $\partial_y\vq,$ (namely item (iii) of Proposition \ref{prop:V-def-prop} and \eqref{6triangle}), then by applying the inequality \eqref{4335} with $r=\vq,$ we get
\begin{equation}
\label{9triangle}
\left|\mathcal{L}(\chi\vq)-\chi\mathcal{L}\vq\right|\leq \left( {CA\over s^{4\beta+1}}+{C\sqrt{A}\over s^{5\beta-1}}+{CA^2\over s^{\gamma}}\right)(1+|y|^3)\leq C{A^2\over s^{\gamma}}(1+|y|^3),
\end{equation}
since $\gamma\leq \min(5\beta-1,4\beta+1).$
Then, we deduce from \eqref{7triangle}, \eqref{8triangle} and \eqref{9triangle},
\begin{equation}
\label{(3)}\left|\mathcal{L}\vq_--\left(\mathcal{L}\vq\right)_-\right|\leq C{A^2\over s^{\gamma}}(1+|y|^3).
\end{equation}
Now using (\ref{(1)}), (\ref{(2)}) and (\ref{(3)}) we conclude the proof of the lemma.
\end{proof}

Now, we are in a position to give our parabolic regularity
statement:
\begin{prop}[Parabolic regularity for equation (\ref{qequation})]
\label{prop:regu-parab-q-equation} For all $A\geq 1$, there exists
$s_{14}(A)$ such that for all $s_0\geq s_{14}(A)$ the following
holds:

Consider $\vq(s)$  a solution of equation (\ref{qequation}) on
$[s_0,s_1]$ where $s_1\geq s_0$ with initial data at $s=s_0$
$$\vq(y,s_0)=\psi_{s_0,d_0,d_1}(y)$$
defined in \eqref{initial-data} with $(d_0,d_1)\in \mathcal D_{s_0},$ and
\begin{equation}\label{eq:q-in-VA-bis}
 \vq(s)\in \vartheta_A(s) \mbox{ for all } s\in [s_0,s_1].\end{equation}
  Then, for all $s\in[s_0,s_1],$ we have
\begin{itemize}
\item[(i)] $ \|\nabla \vq(s)\|_{L^\infty(\R)}\leq C\frac{A^2}{s^{\gamma-3\beta}}.$
\item[(ii)]$|\nabla \vq(y,s)|\leq C\frac{A}{s^{2\beta+1}}(1+|y|)+
C\frac{\sqrt{A}}{{s^{4\beta-1}}}(1+|y|^2)+C
\frac{A^2}{s^{\gamma}}(1+|y|^3),\; \forall\; y\in \R.$
\item[(iii)] $|\nabla \vq_-(y,s)|\leq C{A^2\over s^{\gamma}}(1+|y|^3),\; \forall \; y\in \R,$
\end{itemize}
where $C$ is a positive constant.
\end{prop}
\begin{proof}   We consider $A\geq 1$, $s_0\geq 1$ and $\vq(s)$ a solution of equation (\ref{qequation})
defined on $[s_0,s_1]$ where $s_1\geq s_0\geq 1$ and
$\vq(s_0)$ given by (\ref{initial-data}) with $(d_0,d_1)\in
\mathcal D_{s_0}$. We also assume that $\vq(s)\in \vartheta_A(s)$ for
all $s\in[s_0,s_1]$.  For each of the items (i), (ii) and (iii), we consider two cases in the proof: $s\leq
s_0+1$ and $s> s_0+1.$

\bigskip
{\bf Proof of Part (i)}\\ {\it Case 1:} ${s\leq s_0+1.}$ Let
$s_1'=\min (s_0+1,s_1)$ and take $s\in [s_0,s_1']$. Then, since $s_0\geq 1$, we have for
any $t\in[s_0,s]$,
\begin{equation}\label{canada}
 s_0\leq t\leq s\leq s_0+1\leq 2s_0, \mbox{ hence } \frac{1}{s}\leq \frac{1}{t}\leq
 \frac{1}{s_0}\leq {2\over s}
 .\end{equation}
 From equation (\ref{qequation}), we write for any   $s\in[s_0,s_1'],$
  \begin{equation}\label{eq:q:int:operator:L-case1-001}
  \vq(s)= e^{(s-s_0)\mathcal{L}} \vq(s_0)+ \int_{s_0}^{s} e^{(s-t)\mathcal{L}} F(t)dt,
  \end{equation}
where \begin{equation}\label{def:F}
  F(x,t)=V\vq(x,t)+{G}(x,t)+B(\vq)+R(x,t).\end{equation}
Hence
  \begin{equation}\label{eq:q:int:operator:L-case1-N}
  \nabla \vq(s)= \nabla e^{(s-s_0)\mathcal{L}} \vq(s_0)+ \int_{s_0}^{s} \nabla e^{(s-t)\mathcal{L}} F(t)dt,
  \end{equation}
and
   $$|\nabla \vq(y,s)|\leq  |\nabla e^{(s-s_0)\mathcal{L}} \vq(s_0)|+ \int_{s_0}^{s} |\nabla e^{(s-t)\mathcal{L}}
   F(t)|dt.$$
Then, we write from (\ref{eq:q:int:operator:L-case1-N}) and Lemma \ref{l53},
for all $s\in [s_0,s_1'],$
   \begin{align}
   \|\nabla \vq(s)\|_{L^\infty(\R)}&\leq \|\nabla e^{(s-s_0)\mathcal{L}}
   \vq(s_0)\|_{L^\infty}+\int_{s_0}^{s}\|\nabla e^{(s-t)\mathcal{L}}F(t)\|_{L^{\infty}}dt \nonumber\\
   &\leq C\|\nabla \vq(s_0)\|_{L^\infty}+C\int_{s_0}^{s}
   \frac{\|F(t)\|_{L^{\infty}}}{\sqrt{1-e^{-(s-t)}}}dt. \label{ineg:grad-q-case1}
   \end{align}
    Using
(\ref{qrdienq_s0}) and (\ref{canada}), we write
    \begin{equation}\label{ineg:regular-parabo-q0}
    \|\nabla \vq(s_0)\|_{L^\infty(\R)}\leq \frac{1}{s_0^{\gamma-3\beta}}
    \leq \frac{C}{s^{\gamma-3\beta}}.\end{equation}

    Furthermore, using \eqref{e20151}, \eqref{estG} and \eqref{4295}, we write
  \begin{eqnarray*}
  |G(x,t)|& \leq & C|\nabla\varphi|^{q-1}|\nabla \vq|+C|\nabla \vq|^q\\
   & \leq & {C\over t^{\beta(q-1)}}\|\nabla \vq\|_\infty+C\|\nabla \vq\|_\infty^q\\
   &\leq & {C\over \sqrt{t}}\|\nabla \vq\|_\infty+C\|\nabla \vq\|_\infty^q.
   \end{eqnarray*}

Using (\ref{eq:q-in-VA-bis}),
    (ii) of Proposition
\ref{prop:V-def-prop},  Lemmas \ref{p52} and \ref{l52}, together with (\ref{canada}), we write
   for all $t\in[s_0,s]$ and $x\in \mathbb{R}$,
  \begin{equation}
  \label{ineg:estim1:F}
  |F(x,t)|\leq \frac{CA^2}{s^{\gamma-3\beta}}+\frac{C}{\sqrt{t}}\|\nabla \vq(t)\|_{L^\infty(\R)}
  +C\|\nabla \vq(t)\|_{L^\infty(\R)}^{q}.
  \end{equation}
 Therefore, from  (\ref{ineg:regular-parabo-q0}) and (\ref{ineg:estim1:F})
  we write with  $g(s)=\|\nabla \vq(s)\|_{L^\infty(\R)},$
  \begin{equation}
  \label{e514}
  g(s)\leq \frac{C_1A^2}{s^{\gamma-3\beta}}+C\int_{s_0}^{s}\frac{t^{-{1\over 2}}g(t)+ g(t)^q}{\sqrt{1-e^{-(s-t)}}}dt,\end{equation}
 for some universal constant $C_1>0.$ Using a Gronwall's argument, we claim that
\begin{equation}
 \label{4carre}
 \forall  s\in [s_0,s_1'],~ g(s)\leq \frac{2C_1A^2}{s^{\gamma-3\beta}},
 \end{equation}
for $s_0$ large enough. Indeed, let
$$s_\star=\sup\left\{s\in [s_0,s_1'] \, |\, \forall\, s'\in [s_0,s],\, g(s')\leq
\frac{2C_1}{s'^{\gamma-3\beta}}\right\}.$$  From \eqref{e514}, $s_\star$ is well defined and $s_\star>s_0.$ Suppose, by contradiction, that
$s_\star<s_1'.$  In this case, by continuity, we have:
\begin{equation}
 \label{3carre}
 g(s_\star)= \frac{2C_1}{s_\star^{\gamma-3\beta}}
\end{equation}
By \eqref{e514}, \eqref{canada}  and the definition of $s_\star,$ we
have, for all $s\in [s_0,s_\star],$
\begin{eqnarray*}
g(s)&\leq & \frac{C_1A^2}{s^{\gamma-3\beta}}+C\int_{s_0}^{s}
\frac{\left(\frac{2C_1A^2}{t^{\gamma-3\beta}}\right)^q+2C_1A^2t^{3\beta-\gamma-{1\over 2}}}{\sqrt{1-e^{-(s-t)}}}dt\\
&\leq & \frac{C_1A^2}{s^{\gamma-3\beta}}+C\left[\left(2C_1A^2\right)^q{1\over
s^{q(\gamma-3\beta)}}+2C_1A^2s^{3\beta-\gamma-{1\over 2}}\right]\int_{s_0}^{s} {1\over \sqrt{1-e^{-(s-t)}}}dt
\\ & < & \frac{C_1}{s^{\gamma-3\beta}}+\left[{(2C_1A^2)^q\over s^{q(\gamma-3\beta)}}+{2C_1A^2\over s^{\gamma-3\beta+{1\over 2}}}\right]\\ &\leq & \frac{3C_1}{2s^{\gamma-3\beta}},
\end{eqnarray*}
for $s_0$ sufficiently large, since $q>1.$ This is a contradiction
by \eqref{3carre}. Hence $s_1'=s_\star,$ and \eqref{4carre} holds.
This concludes the proof of Part (i) when $s\leq s_0+1.$

\bigskip
{\it  Case 2:} ${s>s_0+1.}$ (Note that this case does
not occur when $s_1\leq
 s_0+1$). Take $s\in (s_0+1,s_1]$. Then, we have for any
$s'\in(s-1,s]$ and $t\in [s-1,s'],$  $s\geq s_0+1\geq 2,$ hence $s=s-1+1\leq 2(s-1)$. Therefore,
\begin{equation}\label{canadacase2}
 s-1\leq t\leq s'\leq s\leq 2(s-1) \mbox{ hence }
 \frac{1}{s}\leq \frac{1}{s'}\leq \frac{1}{t}\leq
 \frac{1}{s-1}\leq {2\over s}
 .\end{equation}
   From equation (\ref{qequation}), we write for any $s'\in[s-1,s],$
  \begin{equation}\label{eq:q:int:operator:L-case2-001}
  \vq(s')= e^{(s'-s+1)\mathcal{L}} \vq(s-1)+ \int_{s-1}^{s'} e^{(s'-t)\mathcal{L}} F(t)dt,
  \end{equation}
where $F(x,t) $ and $e^{\theta \mathcal L}$ are given in
(\ref{def:F}) and (\ref{eq:semigroup:theta-L-case1}). Using Lemma \ref{l53}, we see  that
for all $s'\in [s-1,s]$
   \begin{align}
   \|\nabla \vq(s')\|_{L^\infty(\R)}&\leq \|\nabla e^{(s'-s+1)\mathcal{L}}\vq(s-1)\|_{L^\infty(\R)}+
   \int_{s-1}^{s'}\|\nabla e^{(s'-t)\mathcal{L}}F(t)\|_{L^{\infty}(\R)}dt \nonumber\\
   &\leq \frac{C}{\sqrt{1-e^{-(s'-s+1)}}}\|\vq(s-1)\|_{L^\infty(\R)}+C\int_{s-1}^{s'}
   \frac{\|F(t)\|_{L^{\infty}(\R)}}{\sqrt{1-e^{(s'-t)}}}dt. \label{ineg:grad-q-case2}
   \end{align}
  Recall from \eqref{eq:q-in-VA-bis}, Proposition \ref{prop:V-def-prop} and \eqref{canadacase2} that
$$\| \vq(s-1)\|_{L^\infty(\R)}\leq
  \frac{CA^2}{(s-1)^{\gamma-3\beta}}\leq
  \frac{CA^2}{{s'}^{\gamma-3\beta}}$$
  for $s_0$ sufficiently large.
Therefore, using \eqref{ineg:estim1:F}, \eqref{canadacase2}, we write with  $g(s')=\|\nabla \vq(s')\|_{L^\infty(\R)}$,
  \begin{equation*}
  g(s')\leq \frac{C_1A^2}{s'^{\gamma-3\beta}\sqrt{1-e^{-(s'-s+1)}}}+C\int_{s-1}^{s'}\frac{ {t}^{-1/2}g(t) +g(t)^q}{\sqrt{1-e^{-(s'-t)}}}dt.
  \end{equation*}
   Using a Gronwall's argument, as for the previous case,
  we see that for $s$ large enough,
  \[ \forall\;  s'\in [s-1,s],~ g(s')\leq \frac{2C_1A^2}{s'^{\gamma-3\beta}\sqrt {1-e^{-(s'-s+1)}}}.\]
 Taking $s'=s$ concludes the proof of Proposition
 \ref{prop:regu-parab-q-equation}, part (i), when $s>s_0+1$.

\bigskip

 {\bf Proof of Part (ii)} Let us define the function
 $$H(y,s)=\frac{A}{s^{2\beta+1}}(1+|y|)+
\frac{\sqrt{A}}{{s^{4\beta-1}}}(1+|y|^2)+
\frac{A^2}{s^{\gamma}}(1+|y|^3),\; \forall\; s>0,\; \forall\;  y\in \R.$$
We consider $s\in [s_0,s_1]$. Since $\vq \in \vartheta_A(s),$ using Proposition \ref{prop:V-def-prop}, part (iii), we see that
$$|\vq(y,s)|\leq C H(y,s),\; \mbox{ for all }\;  y\in \R,$$
provided that $s_0 \mbox{ is sufficiently large.}$ Similarly, since $V\vq(s)\in \vartheta_{CA}(s) ,\; B(\vq)\in \vartheta_C(s)$ and $R\in \vartheta_C(s)$ by Lemma \ref{p52} and \ref{l52}, it follows that
$$|V\vq(y,s)|+|B(v)|+|R(y,s)|\leq C H(y,s).$$
Using \eqref{e20151} and \eqref{estG}, we see by definition \eqref{def:F} of $F$ that
$$|F(y,s)|\leq C H(y,s)+{C\over \sqrt{s}}|\nabla\vq(y,s)|+|\nabla\vq(y,s)|^q.$$
Using the fact that $\vq \in \vartheta_A(s),$ it follows by the previous Part (i) that
$$\|\nabla\vq(s)\|_{L^\infty(\R)}^{q-1}\leq C{A^{2(q-1)}\over s^{(q-1)(\gamma-3\beta)}},$$
which is sufficiently small for $s_0$ large. Therefore, noting from \eqref{4295} and \eqref{4345} that $$(q-1)(\gamma-3\beta)={\gamma-3\beta\over 2\beta}<{1\over 2},$$ we see that
\begin{equation}
\label{estimF}
|F(y,s)|\leq C H(y,s)+{CA^{2(q-1)}\over s^{(q-1)(\gamma-3\beta)}}|\nabla\vq(y,s)|.
\end{equation}
Furthermore,
we have by Proposition \ref{prop:diddc}, Part (iii) that
$$|\nabla\vq(y,s_0)|\leq C_0 H(y,s_0),$$
for some $C_0>0.$ If we introduce the norm (depending on $s$)
$$\mathcal{N}(r)=\left\|{r\over H(\cdot,s)}\right\|_{L^\infty(\R)},$$
then we see that
\begin{equation}
 \label{4695}
\mathcal{N}(\nabla\vq(s_0))\leq C_0.\end{equation}
We should show that
\begin{equation}
\label{3+triangle}
\sup_{s\in [s_0,s_1]}\mathcal{N}(\nabla\vq(s))\leq M,
\end{equation}
where $M$ is a fixed constant to be determined later, in particular $M>2C_0.$ We argue by contradiction. From \eqref{4695}, we consider  $s_*\in (s_0,s_1)$ the largest one such that the previous inequality is satisfied.
Using item (i) of Proposition \ref{prop:regu-parab-q-equation} (which is already proved),
we can restrict $y$ to some compact interval depending on $s.$ Then, using the uniform continuity, we deduce that
\begin{equation}
 \label{4705}
 \mathcal{N}(\nabla\vq(s_*))=M,\; \mbox{ and }\; \mathcal{N}(\nabla\vq(s))\leq M,\; \mbox{for all}\;  s\in [s_0,s_*].
\end{equation}
We handle two cases:

\medskip

{\it Case 1:} ${s_*\leq s_0+1.}$ Recalling the integral equation \eqref{eq:q:int:operator:L-case1-N}, we have that
\begin{equation}\label{eq:q:int:operator:L-case1}
 \forall \; s\in [s_0,s_\star],\;  \nabla \vq(s)= \nabla e^{(s-s_0)\mathcal{L}} \vq(s_0)+ \int_{s_0}^{s} \nabla\Big(e^{(s-t)\mathcal{L}} F(t)\Big)dt.
  \end{equation}
 Then, by Lemma \ref{l53}, Part (ii) we have that
\begin{equation}
 \label{estimvq}
\mathcal{N}\Big(\nabla e^{(s-s_0)\mathcal{L}}\vq(s_0)\Big)\leq C_1C_0,\end{equation}
where $C_0$ is used in \eqref{4695}, and  $C_1>0$ is a constant. On the other hand, using \eqref{estimF} and \eqref{4705}, we see that
\begin{eqnarray*}
|F(x,t)|&\leq & C H(x,t)+{CA^{2(q-1)}\over t^{(q-1)(\gamma-3\beta)}}|\nabla\vq(x,t)|\\
&\leq & C\left(1+{A^{2(q-1)}\over t^{(q-1)(\gamma-3\beta)}}M\right)H(x,t),\; \forall \; t\in [s_0,s_*].
\end{eqnarray*}
Using Lemma \ref{l53}, Part (vi) and \eqref{canada} we deduce that for $s_0$ sufficiently large,
\begin{eqnarray}
\left|\int_{s_0}^{s} \nabla e^{(s-t)\mathcal{L}}F(t)dt\right| &\leq &
 \int_{s_0}^{s} \left|\nabla e^{(s-t)\mathcal{L}}
   F(t)\right|dt \nonumber\\  \nonumber
   &\leq & C\left(1+{A^{2(q-1)}\over s_0^{(q-1)(\gamma-3\beta)}}M\right)\int_{s_0}^{s}\frac{H(y,t)}{\sqrt{1-e^{-(s-t)}}}dt\\  \nonumber
   &\leq & C\left(1+{A^{2(q-1)}\over s_0^{(q-1)(\gamma-3\beta)}}M\right)\left(\int_{s_0}^{s}\frac{1}{\sqrt{1-e^{-(s-t)}}}dt\right)H(y,s)
   \\
   &\leq & C_2\left(1+{A^{2(q-1)}\over s_0^{(q-1)(\gamma-3\beta)}}M\right)H(y,s),\label{1+triangle}
\end{eqnarray}
 for some $C_2>0.$ Hence, using also \eqref{estimvq}, we get from \eqref{eq:q:int:operator:L-case1},
 \begin{equation}
  \label{4725}
|\nabla\vq(y,s)|\leq \left(C_1C_0+C_2\left(1+{A^{2(q-1)}\over s_0^{(q-1)(\gamma-3\beta)}}M\right)\right) H(y,s), \; \forall \; s\in [s_0,s_*].\end{equation}
 Assuming that
 $$M\geq \max\left(2C_0,2C_1C_0+4C_2\right),$$
then taking $s_0$ large so that
\begin{equation}
 \label{2+triangle}
 {A^{2(q-1)}\over s_0^{(q-1)(\gamma-3\beta)}}M\leq 1,
\end{equation}
we see from \eqref{4725} that
$$\forall\; s\in [s_0,s_\star],\; \forall\; y\in \R,\; |\nabla\vq (y,s)|\leq (C_0C_1+2C_2)H(y,s)\leq {M\over 2}H(y,s),$$
which is a contradiction by \eqref{4705}.

 \medskip
{\it  Case 2:} ${s_*>s_0+1.}$  From \eqref{eq:q:int:operator:L-case2-001}, we
write for any $s'\in[s_*-1,s_*],$
  \begin{equation}\label{eq:q:int:operator:L-case2new01}
  \nabla \vq(s')= \nabla e^{(s'-s_*+1)\mathcal{L}} \vq(s_*-1)+ \int_{s_*-1}^{s'} \nabla e^{(s'-t)\mathcal{L}} F(t)dt.
  \end{equation}
Since $\vq(s_*-1)\in \vartheta_A(s_*-1),$  by Proposition \ref{prop:V-def-prop} Part (iii) we have $|\vq(s_*-1)|\leq C H(y,s_*),$ for $s_0$ sufficiently large. By Lemma \ref{l53}, Part (vi), we have for all $s'\in (s_*-1,s_*]$
 $$\left|\nabla e^{(s'-s_*+1)\mathcal{L}}\vq(s_*-1)\right|\leq  C{e^{{(s'-s_*+1)\over 2}}\over \sqrt{1-e^{-(s'-s_*+1)}}}H(y,s_*).
$$
Hence at $s'=s_*$, we have
$$
 \left|\nabla e^\mathcal{L}\vq(s_*-1)\right|\leq C_3 H(y,s_*),$$
 for some constant $C_3>0.$ Proceeding as for \eqref{1+triangle}, we write
\begin{eqnarray*}
\left|\int_{s_*-1}^{s_*} \nabla e^{(s_*-t)\mathcal{L}}F(t)dt\right| &\leq &
 C_2\left(1+{A^{2(q-1)}\over s_0^{(q-1)(\gamma-3\beta)}}M\right)H(y,s_*).
\end{eqnarray*}
Then, from \eqref{eq:q:int:operator:L-case2new01}, we derive that
\begin{eqnarray*}
 | \nabla \vq(y,s_*)| & \leq & \left|\nabla e^{\mathcal{L}} \vq(s_*-1)\right|+ \left|\int_{s_*-1}^{s_*} \nabla e^{(s_*-t)\mathcal{L}} F(t)dt\right|,\\
&\leq & \left[C_3+C_2\left(1+{A^{2(q-1)}\over s_0^{(q-1)(\gamma-3\beta)}}M\right)\right]H(y,s_*).
  \end{eqnarray*}

  Fixing
 $$M= 2\max\left(C_0,C_1C_0+2C_2,C_3+2C_2\right),$$
and taking $s_0$ large enough so that \eqref{2+triangle} holds, we see that
$$| \nabla \vq(y,s_*)|\leq (C_3+2C_2)H(y,s_*)\leq {M\over 2}H(y,s_*),$$
and we get a contradiction with \eqref{4705}.

Since the value we have just fixed for $M$ also leads to a contradiction in Case 1, we have just proved the validity of \eqref{3+triangle}. This finishes the proof of Part (ii).

\bigskip

{\bf Proof of Part (iii)} By Lemma \ref{l54}, we have that
\begin{equation}
 \label{4733}
\partial_s\vq_-=\mathcal{L}\vq_-+\overline{F}(y,s),\end{equation}
with
\begin{equation}
 \label{4734}
 |\overline{F}(y,s)|\leq {CA^2\over s^\gamma}\left(1+|y|^3\right),\; \forall\; y\in \R.\end{equation}
By Proposition \ref{prop:diddc}, Part (iii), we have that
\begin{equation}
 \label{4735}|\nabla \vq_-(y,s_0)|\leq {1\over s_0^{\gamma}}(1+|y|^3),\; \forall\; y\in \R.
 \end{equation}
We should show that
\begin{equation}\label{aim}
 |\nabla \vq_-(y,s)|\leq {\overline{M}\over s_0^{\gamma}}(1+|y|^3),\; \forall\; y\in \R,\; \forall\; s\in [s_0,s_1],
\end{equation}
 for some $\overline{M}\geq 2$ that will be fixed later, provided that $s_0$ is large enough.\\
 We then distinguish two cases.

 \medskip

 {\it Case 1:}  $s_*\leq s_0+1.$ From \eqref{4733}, we write
$$\vq_-(s)= e^{(s-s_0)\mathcal{L}} \vq_-(s_0)+ \int_{s_0}^{s} e^{(s-t)\mathcal{L}} \overline{F}(t)dt,$$
and
$$|\nabla\vq_-(s)|\leq  |\nabla e^{(s-s_0)\mathcal{L}} \vq_-(s_0)|+ \int_{s_0}^{s} |\nabla e^{(s-t)\mathcal{L}} \overline{F}(t)|dt.$$
Then, by \eqref{canada}, Lemma \ref{l53}, Parts (v) and (vi) together with \eqref{4735} and \eqref{4734}, we see that
 \begin{eqnarray*}
 \forall\; y\in \R,\; |\nabla\vq_-(y,s)|&\leq &  {C\over s_0^{\gamma}}(1+|y|^3)+C{A^2\over s^\gamma}\left(\int_{s_0}^{s}{1\over\sqrt{1-e^{-(s-t)}}}dt\right)(1+|y|^3)\\
 & \leq &  {C_4A^2\over s^{\gamma}}(1+|y|^3),
 \end{eqnarray*}
 for $s_0$ sufficiently large and some constant $C_4>0.$

 \medskip

  {\it Case 2:} $s_*>s_0+1.$ Using \eqref{4733}, we write
  $$\nabla\vq_-(s)= \nabla e^{(s-s_*+1)\mathcal{L}} \vq_-(s_*-1)+ \int_{s_*-1}^{s} \nabla e^{(s-t)\mathcal{L}} \overline{F}(t)dt,\, s\in (s_*-1,s_*].$$
  Since
  $$|\vq_-(y,s_*-1)|\leq {A\over s^\gamma}(1+|y|^3),$$
  from the fact that $\vq\in \vartheta_A(s_*-1),$ using Lemma \ref{l53} Part (vi), together with \eqref{4734} we get
 \begin{eqnarray*}
 |\nabla\vq_-(y,s_*)|&\leq & {CA\over (s_*-1)^{\gamma}\sqrt{1-e^{-1}}}(1+|y|^3)+{CA^2\over s_*^\gamma}\left(1+|y|^3\right)\left(\int_{s_*-1}^{s_*}{1\over \sqrt{1-e^{-(s-t)}}}dt\right)\\
 & \leq &  {C_5A^2\over s_*^{\gamma}}(1+|y|^3).
 \end{eqnarray*}
 Fixing
 $$\overline{M} = 2\max(2,C_4A^2, C_5A^2),$$
we see that \eqref{aim} follows.
This finishes the proof of Proposition \ref{prop:regu-parab-q-equation}.
 \end{proof}

\subsubsection{Reduction to a finite dimensional problem}
In the following, we reduce the problem to a finite-dimensional one.
Namely, we prove Proposition \ref{prop:rt2dimen}. To do so, we
project Equation (\ref{qequation}) on the different components of
the decomposition \eqref{qprojection}. Let us first give those
projections in the following proposition, then use it to derive
Proposition \ref{prop:rt2dimen}. This is the statement of the
following proposition.

\begin{prop}[Dynamics of the different components]\label{prop:projet:q}
There exists $A_6\geq 1$ such that for all $A\geq A_6$ there exists $s_{06}(A)$ large enough such that the following holds for all $s_0\geq s_{06}(A)$:

Assume that  for some
$s_1\geq\tau\geq s_0,$ we have
\[\vq(s)\in \vartheta_A(s),\; \mbox{ for all }\;s\in [\tau,s_1],\] and that $\nabla\vq (s)$ satisfies the estimates stated in Parts (i)-(ii)-(iii) of Proposition \ref{prop:regu-parab-q-equation}.  Then, the following holds for all $s\in[\tau,s_1]$:
\begin{itemize}
\item[(i)] (ODE satisfied by the positive modes) For $m=0$ and $m=1$, we have
$$ \left|\vq_m'(s) -(1-\frac{m}{2})\vq_m(s)\right|\leq \frac{C}{s^{2\beta+1}}.$$
\item[(ii)] (ODE satisfied by the null mode) For $m=2$, we have
$$ \left|\vq_2'(s)+{2\beta+1\over s}\vq_2(s)\right|\leq \frac{C}{s^{4\beta}}.$$
\item[(iii)] (Control of the negative and outer modes) We have
 \begin{align*}
\left\| \frac{\vq_{-}(y,s)}{1+|y|^3}\right\|_{L^{\infty}}&\leq C e^{-\frac{(s-\tau)}{2}}\left\|\frac{\vq_{-}(y,\tau)}{1+|y|^3}\right\|_{L^\infty}
+C\frac{e^{-(s-\tau)^2}\| \vq_{e}(\tau)\|_{L^\infty}}{s^{{3\beta}}}+\frac{C(1+s-\tau)}{s^\gamma},\\
\|\vq_{e}(s)\|_{L^\infty}&\leq Ce^{-\frac{(s-\tau)}{p}}\|\vq_{e}(\tau)\|_{L^\infty}+Ce^{s-\tau} s^{3\beta}\left\|\frac{\vq_{-}(y,\tau)}{1+|y|^3}\right\|_{L^\infty}+C\frac{\left(1+(s-\tau)e^{s-\tau}\right)}{s^{\gamma-3\beta}}.
\end{align*}
\end{itemize}
\end{prop}
\begin{rem}
{\rm In item (ii), the value of the coefficient $2\beta+1$ in front of ${\vq_2\over s}$ crucially comes from an algebraic identity at the end of the proof of this item, involving the parameter $b$ defined in
\eqref{bbupprof1}.}
\end{rem}

\begin{proof}
The proof will be carried-out in 3 steps:
\begin{itemize}
\item[-] In the first step, we write equations satisfied by $\vq_0,\; \vq_1.$ Then, we prove (i) of Proposition \ref{prop:projet:q}.
\item[-] In the second step, we write an equation satisfied by $\vq_2.$ Then, we prove (ii) of Proposition \ref{prop:projet:q}.
\item[-] In the third step, we write integral equations satisfied by $\vq_-$ and $\vq_e.$ Then, we prove  Part (iii) of Proposition \ref{prop:projet:q}.
\end{itemize}

\bigskip

{\bf Step 1: The  positive modes.} As in the proof of Lemma \ref{l54}, we project equation \eqref{qequation}, on the modes $m=0,\; 1$ defined in \eqref{qprojection}:
\begin{equation}
 \label{5+carre}
(\partial_s\vq)_m= (\mathcal{L}\vq)_m+(V\vq)_m+\left(B(\vq)\right)_m+\left(G(\vq)\right)_m+R_m,\; m=0,\;1,\; 2.
\end{equation}
Recall from \eqref{2triangle} and \eqref{8triangle} that
\begin{equation}
 \label{6+carre}
 \left|\vq_m'-\left(\partial_s\vq\right)_m\right|+\left|\left(1-{m\over 2}\right)\vq_m-\left(\mathcal{L}\vq\right)_m\right|\leq {2\over s^{4\beta}}\leq {2\over s^{2\beta+1}},\; m=0,\; 1,\; 2.
\end{equation}
On the other hand,
for $s_0$ sufficiently large, by Lemmas \ref{p52} and \ref{l52} together with the hypotheses, we have
$$\left|(V\vq)_m+\left(B(\vq)\right)_m+\left(G(\vq)\right)_m+R_m\right|\leq {C\over s^{2\beta+1}},\; m=0,\; 1.$$
This proves (i) of the proposition.

\bigskip

{\bf Step 2: The  null mode.} From \eqref{5+carre}, we have
$$(\partial_s\vq)_2=(\mathcal{L}\vq)_2+(V\vq)_2+\left(B(\vq)\right)_2+R_2+\left(G(\vq)\right)_2.$$
Note that the first terms are estimated in \eqref{6+carre}. Using Lemmas \ref{p52} and \ref{l52}, we write
$$\left|(V\vq)_2+\left(B(\vq)\right)_2+R_2\right|\leq {C\over s^{4\beta}}.$$
Thus, it remains to prove that
\begin{equation}
 \label{5losange}
\left|\left(G(\vq)\right)_2+{2\beta+1\over s}\vq_2(s)\right|\leq {C\over s^{4\beta}}
\end{equation}
in order to conclude. Let us prove that. Write
\begin{equation}
 \label{7+carre}
\vq=\vq_1h_1+\vq_2h_2+\tilde{\vq}.
\end{equation}
Then
$$\partial_y\vq=\vq_1+2y\vq_2+\partial_y\tilde{\vq}.$$ Write also
\begin{eqnarray}
 \nonumber
 G(\vq) & = & \mu\Big[|\partial_y\varphi+\partial_y\vq|^q-|\partial_y\varphi|^q\Big]\\ \nonumber &=& \mu \Big[|\partial_y\varphi+\partial_y\vq|^q-|\partial_y\varphi+\vq_1+2y\vq_2|^q\Big]\\ \nonumber & & \hspace{-0.3cm}+
 \mu \Big[|\partial_y\varphi+\vq_1+2y\vq_2|^q-|\partial_y\varphi|^q\Big]\\ \label{4855}
& := & G^1(\vq)+G^2(\vq).
\end{eqnarray}
We begin by estimating $G^1(\vq).$ If $\chi$ is defined in \eqref{def:chi}, we have
\begin{eqnarray*}
|\chi G^1(\vq)| & = & \mu \chi\Big[|\partial_y\varphi+\partial_y\vq|^q-|\partial_y\varphi+\vq_1+2y\vq_2|^q\Big]\\
&=& \mu\chi \Big||\partial_y\varphi+\vq_1+2y\vq_2+\partial_y\tilde{\vq}|^q-|\partial_y\varphi+\vq_1+2y\vq_2|^q\Big|\\ & \leq &
C |\partial_y\varphi+\vq_1+2y\vq_2+\theta \partial_y\tilde{\vq}|^{q-1}|\partial_y\tilde{\vq}|\chi,\;
\end{eqnarray*}
for some $\theta\in [0,1].$ Note from \eqref{7+carre} and
\eqref{qprojection}  that $\tilde{\vq}=\vq_0+\vq_-+\vq_e.$ Then,
$\partial_y\tilde{\vq}=\partial_y\vq_-+\partial_y\vq_e
=\partial_y\vq_-+(1-\chi)\partial_y\vq-\partial_y\chi \vq.$ Since we
assumed that $\vq(s)\in \vartheta_A(s)$ and  $\partial_y\vq$
satisfies the identities stated in Parts (i), (ii) and (iii) of
Proposition \ref{prop:regu-parab-q-equation}, it follows that
$$\forall\; |y|\leq 2Ks^\beta,\; |\partial_y\tilde{\vq}|\leq {CA^2\over s^\gamma}(1+|y|^3),\; \mbox{ for }\; s_0\; \mbox{ large enough.}$$
Then, we write from \eqref{e20151}, \eqref{4295} and the fact that $\vq(s)\in \vartheta_A(s),$ for all $|y|\leq 2Ks^\beta,$
\begin{eqnarray*}
 |\partial_y\varphi+\vq_1+2y\vq_2+\theta \partial_y\tilde{\vq}|^{q-1}&\leq & C\left[
 |\partial_y\varphi|^{q-1}+|\vq_1|^{q-1}+|y|^{q-1}|\vq_2|^{q-1}+|\partial_y\tilde{\vq}|^{q-1}\right]\\ &\leq & {C|y|^{q-1}\over s}+{CA^{q-1}\over s^{(2\beta+1)(q-1)}}+
 {CA^{{q-1\over 2}}\over s^{(4\beta-1)(q-1)}}|y|^{q-1}\\ && +{CA^{2(q-1)}\over s^{\gamma(q-1)}}\left(1+|y|^{3(q-1)}\right)\\ &\leq &
{C\over s}(1+|y|^{q-1})+ {CA^{2(q-1)}\over s^{\gamma(q-1)}}\left(1+|y|^{3(q-1)}\right),
\end{eqnarray*}
for $s_0$ large enough, where we need the fact, that $\gamma<2\beta+1,$ and $ (4\beta-1)(q-1)=(4\beta-1)/(2\beta)=2-1/(2\beta)>1,$ because $\beta>1/2.$
Hence, since $\gamma q>\gamma+1$ from \eqref{gamma} and \eqref{4295}, we get
\begin{eqnarray*} |\chi G^1(\vq)|\leq  {CA^2\over s^{\gamma+1}}(1+|y|^{q+2})+ {A^{2q}\over s^{\gamma q}}(1+|y|^{3q})
 \leq   {CA^{2q}\over s^{\gamma +1}}(1+|y|^{3q}).
\end{eqnarray*}
Then, since $\gamma >4\beta-1,$ we get for $s_0$ sufficiently large,
\begin{equation}
 \label{triangle+6}
\left|\int \chi G^1(\vq)k_2\rho\right|\leq  {CA^{2q}\over s^{\gamma +1}}<{1\over s^{4\beta}}.
\end{equation}

Let us now consider $G^2(\vq).$ Using \eqref{e20151} we write
$$\partial_y\varphi=-{2b\kappa\over (p-1)^2}{y\over s^{2\beta}}\left[1+O\left({|y|^2\over s^{2\beta}}\right)\right]$$ hence
\begin{equation}
 \label{4856}
|\partial_y\varphi|^q=\left({2b\kappa\over (p-1)^2}\right)^q{|y|^q\over s^{2\beta+1}}+ O\left({|y|^{q+2}\over s^{4\beta+1}}\right).
\end{equation}
On the other hand, since
\begin{equation}
 \label{4857}
 |\vq_2s^{2\beta}|\leq {\sqrt{A}\over s^{2\beta-1}}\to 0 \; \mbox{as }\; s\to \infty \; \mbox{ and }\; |\vq_1s^{2\beta}|\leq {A\over s}\to 0 \; \mbox{ as } s\to \infty,
 \end{equation}
 we write
$$\partial_y\varphi+\vq_1+2y\vq_2= -{2b\kappa\over (p-1)^2}{y\over s^{2\beta}}\left[1-\vq_1 {s^{2\beta}\over y}{1\over a}-2\vq_2s^{2\beta} {1\over a}+O\left({|y|^2\over s^{2\beta}}\right)\right],$$
for $|y|\leq 2Ks^\beta, \; y\not=0$ and where $a$ is given in \eqref{aeq}. Let us consider two subsets:
\begin{equation}
 \label{4858}
I=\left\{y\; |\; |y|\geq {2\over a}|\vq_1|s^{2\beta}\right\},\; J=\left\{y\; |\; |y|< {2\over a}|\vq_1|s^{2\beta}\right\}.
\end{equation}
Let $\varepsilon_0>0.$ If $|y|<\varepsilon_0s^\beta$ and $y\in I,$  since $2\beta q=2\beta+1,$ we have
\begin{eqnarray*}
|\partial_y\varphi+\vq_1+2y\vq_2|^q&=&\left({2b\kappa\over (p-1)^2}\right)^q{|y|^q\over s^{2\beta+1}}\left[1-q\vq_1 {s^{2\beta}\over y}{1\over a}-2q\vq_2s^{2\beta} {1\over a}\right] \\ & &
+O\left({|y|^{q+2}\over s^{4\beta+1}}\right)+O(\vq_2^2s^{2\beta-1}|y|^q)+O\left(\vq_1^2s^{2\beta-1} |y|^{q-2}\right),
\end{eqnarray*}
hence, from \eqref{4855} and \eqref{4856},
\begin{eqnarray*}
G^2(\vq)&=&-q\mu a^{q-1}\vq_1{|y|^q\over y}{1\over s}-2q\mu a^{q-1}\vq_2{|y|^q\over s}  +O\left({|y|^{q+2}\over s^{4\beta+1}}\right)\\ & & +O(\vq_2^2s^{2\beta-1}|y|^q)+O\left(\vq_1^2s^{2\beta-1} |y|^{q-2}\right).
\end{eqnarray*}
Then, since $1<q<2,$  using the fact that function ${|y|^q\over y}$ is odd, we get from \eqref{4857}, for $s_0$ large enough,
\begin{eqnarray*}\int_{I\cap \{|y|<\varepsilon_0s^\beta\}}\chi G^2(\vq)k_2\rho & = & -2q\mu a^{q-1}\frac{\vq_2(s)}s\int_{{2\over a}|v_1(s)|s^{2\beta}\leq |y|\leq \varepsilon_0s^\beta}|y|^qk_2\rho
\\ && +O({1\over s^{4\beta+1}})+O({A\over s^{6\beta-1}})+O({A^2\over s^{2\beta+3}}).\end{eqnarray*}
Then we write
$$\int_{{2\over a}|v_1|s^{2\beta}\leq |y|\leq \varepsilon_0s^\beta}|y|^qk_2\rho=\int_\R |y|^qk_2\rho-I_{ext}-I_{int}$$
with
$$I_{ext}=\int_{|y|>\varepsilon_0s^\beta}|y|^qk_2\rho\; \mbox{and }\; I_{int}=\int_{|y|<{2\over a}|v_1|s^{2\beta}}|y|^qk_2\rho.$$
Using \eqref{4857}, we write
$$I_{int}\leq \int_{|y|<{2A\over as}}|y|^qk_2\rho\leq {CA^{q+1}\over s^{q+1}}$$
and $$ I_{ext}\leq Ce^{-s^{2\beta}}.$$
Therefore, using again \eqref{4857} we see that
\begin{eqnarray}
\nonumber
 \int_{I\cap \{|y|<\varepsilon_0s^\beta\}}\chi G^2(\vq)k_2\rho & = & -2q\mu a^{q-1}\frac{\vq_2(s)}s\int_\R|y|^qk_2\rho+O\left({A^{q+1}\over s^{4\beta+q-1}}\right)\\ \nonumber & &
+O({1\over s^{4\beta+1}})+O({A\over s^{6\beta-1}})+O({A^2\over s^{2\beta+3}}) \\ \label{1losange} &=& -2q\mu a^{q-1}{\vq_2(s)\over s}\int_\R|y|^qk_2\rho+O({1\over s^{4\beta}}),
\end{eqnarray}
since $p>3,$ $\beta<{3\over 2}.$\\
Since $(4\beta-1)q>2\beta q=2\beta+1,$ using \eqref{e20151} and \eqref{4857}, we see by definition \eqref{4855} that
\begin{eqnarray*}
|G^2(\vq)| &\leq & C|\partial_y\varphi|^q+C|\vq_1|^q+C|\vq_2|^q|y|^q\\ &\leq & {CA^{q}\over s^{q(2\beta+1)}}+ {CA^{q/2}|y|^q\over s^{2\beta+1}}.
\end{eqnarray*}
In particular,
\begin{equation}
 \label{2losange}
\left|\int_{|y|>\varepsilon_0s^\beta}\chi G^2(\vq)k_2\rho\right|\leq {CA^qe^{-s^{2\beta}}\over s^{2\beta+1}}\leq {1\over s^{4\beta}},
\end{equation}
and, using \eqref{4857} again, we see that $J\subset \{|y|<{2A\over as}\},$ hence
\begin{eqnarray}
 \nonumber
\left|\int_J\chi G^2(\vq)k_2\rho\right|&\leq & {CA^q\over s^{q(2\beta+1)}}\int_{|y|<{2A\over as}}\rho+{CA^{q/2}\over s^{2\beta+1}}\int_{|y|<{2A\over as}}|y|^q\rho\\ \label{3losange} & \leq & {CA^{q+1}\over s^{q(2\beta+1)+1}}
+{CA^{(q/2)+q+1}\over s^{q(2\beta+1)+q+1}}\leq {1\over s^{4\beta}},
\end{eqnarray}
from the fact that
$q(2\beta+1)+q=1>q(2\beta+1)+1=2\beta+q+2>4\beta$,
by \eqref{4295}
and the fact that $q>1>\beta.$ Finally,
by definition \eqref{4858} of the sets $I$ and $J,$ and estimates \eqref{1losange}, \eqref{2losange} and \eqref{3losange}, we see that
\begin{equation}
 \label{triangle+7}
 \int\chi G^2(\vq)k_2\rho= -2q\mu a^{q-1}{\vq_2(s)\over s}\int_\R|y|^qk_2\rho+O({1\over s^{4\beta}}).
\end{equation}
In conclusion, from \eqref{triangle+6} and \eqref{triangle+7}, we have
$$(G(\vq))_2=-{\tilde{c}\over s}\vq_2+O\left({1\over s^{4\beta}}\right),$$
with
$$\tilde{c}=2q\mu\left({2b\kappa\over (p-1)^2}\right)^{q-1}\int_\R|y|^qk_2\rho.$$
From the particular choice of $b$ given in \eqref{bbupprof1}, together with the definition \eqref{defkm} of $k_2$ and identity \eqref{8+losange},
 it appears that
$$\tilde{c}=2\beta+1$$
and \eqref{5losange} follows. Since \eqref{5losange} was the last identity to check for item (ii) of Proposition \ref{prop:projet:q}, we are done.

\bigskip

{\bf Step 3: The  infinite-dimensional part $\vq_-$ and $\vq_e$.} Let us write the equation \eqref{qequation} on $\vq$ in the following integral form:
\begin{equation}
\label{qinteg}
\vq(s)=S(s,\tau )\vq(\tau)+\int_{\tau}^{s}
S(s,\sigma)B(\vq(\sigma))d\sigma+
 \int_\tau ^s S(s,\sigma)R(\sigma)d\sigma+\int_\tau ^s S(s,\sigma)
G(\sigma)d\sigma,\end{equation}
where $S$ is the fundamental solution  of the operator  $\mathcal{L }+V.$ We
write
   $$ \vq=\mathcal{A}+\mathcal{B}+\mathcal{C}+\mathcal{D} $$ where
\begin{align} \mathcal{A}(s)&=S(s,\tau )\vq(\tau),~\hspace{3cm} \mathcal{B}(s)=\int_{\tau}^{s}S(s,\sigma)B(\vq(\sigma))d\sigma,\\
\mathcal{C}(s)&=\int_\tau ^s S(s,\sigma)R(\sigma)d\sigma, \;\hspace{1.8cm}
\mathcal{D}(s)=\int_\tau ^s S(s,\sigma) G(\sigma)d\sigma.\label{eq:delta:eq} \end{align}
 We assume that
$\vq(s)$ is in
$\vartheta_A(s)$ for each
 $s\in [\tau,\tau+\rho]$, where $\rho>0$. Clearly, from the choice we made for $K$ right after \eqref{def:chi}, the proof of Proposition \ref{prop:projet:q}, Part (iii) follows from the following:
\begin{lem}[Estimates of the different terms of the Duhamel formulation \eqref{qinteg}]\label{lem:A_5}
There exists some $K_5=K_5(N,p,\mu)>0$ such that for whenever $K\ge K_5$, there exists $A_5>0$ such that for all $A\geq A_5
,$ and $\rho>0$
there exists $s_{05}(A,\rho)$, such that
for all $s_0\geq s_{05}(A,\rho),$  if we
assume that
for all $s\in[\tau, \tau+\rho]$, $ \vq(s)$ satisfies (\ref{qequation}), $\vq(s)\in \vartheta_A(s)$ and $\partial_y\vq$ satisfies the estimates given in  Parts (i)-(ii)-(iii) of Proposition \ref{prop:regu-parab-q-equation}, with $\tau\geq s_0,$
then, we have the following results for all $s\in[\tau, \tau+\rho]$:
\begin{itemize}
\item[(i)] {\bf(Linear term)}
$$\left\|\frac{\mathcal{A}_{-}(y,s)}{1+|y|^3}\right\|_{L^{\infty}}\leq Ce^{-\frac{1}{2}(s-\tau)}\left\|\frac{\vq_-(y,\tau)}{1+|y|^3}\right\|_{L^{\infty}}+C{e^{-(s-\tau)^2}\over s^{3\beta}}\|\vq_e(\tau)\|_{L^\infty}+{C\over s^\gamma},$$
$$\|\mathcal{A}_e(s)\|_{L^\infty}\leq Ce^{-\frac{(s-\tau)}{p}}\|\vq_e(\tau)\|_{L^\infty}+Ce^{s-\tau}s^{3\beta}\left\|\frac{\vq_{-}(y,\tau)}{1+|y|^3}\right\|_{L^{\infty}}+{C\over s^{\gamma-3\beta}}.$$
\item[(ii)]{\bf (Nonlinear source term)}
\begin{equation*}
|\mathcal{B}_-(y,s)|\leq\frac C{s^{\gamma}}(s-\tau)(1+|y|^3),~
\|\mathcal{B}_{e}(s)\|_{L^\infty} \leq {C\over s^{\gamma-3\beta}}(s-\tau)e^{s-\tau},\end{equation*}
\item[(iii)]{\bf (Corrective term)}
$$|\mathcal{C}_-(y,s)| \leq \frac C{s^{\gamma}}(s-\tau)(1+|y|^3),\;
\|\mathcal{C}_e(s)\|_{L^\infty}\leq \frac C{s^{\gamma-3\beta}}(s-\tau)e^{s-\tau}.$$
\item[(iv)]{\bf(Nonlinear gradient term)}
$$|\mathcal{D}_{-}(y,s)|\leq \frac C{s^{\gamma}}(s-\tau)(1+|y|^3),\;
\|\mathcal{D}_e(s)\|_{L^\infty} \leq \frac C{s^{\gamma-3\beta}}(s-\tau)e^{s-\tau}.$$
\end{itemize}
\end{lem}
 \emph{ {Proof of Lemma \ref{lem:A_5} }}:
 We consider $A\geq 1$, $\rho>0$ and
$s_0\geq \rho$. We then consider $\vq(s)$ a solution of (\ref{qequation}) satisfying $\vq(s)\in \vartheta_A(s)$
such that $\partial_y\vq$ satisfies the estimates given in Parts (i)-(ii)-(iii) of Proposition \ref{prop:regu-parab-q-equation} for all  $s\in[\tau,\tau+\rho]$, for some $\tau\geq s_0.$

\bigskip

(i) The proof of (i) follows by simple modifications of the argument for \cite[Lemma 3.5]{MZsbupdmj97}, see also \cite{BK}; in fact, making the change of variable in the potential
$$ \overline{V}(y,\overline{s})=V(y,s),\; \mbox{ with }\; \sqrt{\overline{s}}=s^\beta,$$
we reduce to the case of the standard nonlinear heat equation
treated in \cite{MZsbupdmj97}. For that reason, we omit the proof.

\bigskip

(ii)-(iv) Consider  $s_0\geq \rho$. If  we take $\tau\geq s_0$, then $\tau+\rho\leq 2\tau$ and if  $\tau\leq \sigma\leq s\leq \tau+\rho$ then
  \begin{equation}\label{order:inverse:tau-t:s}
  \frac{1}{2\tau}\leq\frac{1}{s}\leq\frac{1}{\sigma}\leq \frac{1}{\tau}.\end{equation}
Let us recall from Bricmont and Kupiainen \cite{BK} that for all $y,x \in \mathbb{R}$
\begin{equation}\label{ineq:B-K} |K(s,\sigma,y,x)|~\leq Ce^{(s-\sigma)\mathcal{L}}(y,x),\; 1\leq \sigma\leq s\leq 2\sigma,\end{equation}
\begin{equation}\label{ineq:B-K2} \left|\int|K(s,\sigma,y,x)|(1+|x|^m)dx\right|\leq Ce^{(s-\sigma)}(1+|y|^m),\; m\geq 0,\; 1\leq \sigma\leq s\leq 2\sigma,
\end{equation}
where $e^{\theta \mathcal{L}}$ is given in (\ref{eq:semigroup:theta-L-case1}) (see  \cite[p. 545]{BK}). The proof of (ii)-(iv) follows by using the fact that $R,\; B,\; G$ are in $\vartheta_C(s),$ the linear part estimates (i),
and by integration. This concludes the proof of Lemma \ref{lem:A_5}  and Proposition \ref{prop:projet:q} too.
\end{proof}

Now, with Proposition \ref{prop:projet:q}, which gives the
projection of the equation \eqref{qequation} on the different
components of the decomposition \eqref{qprojection}, we are ready to
prove Proposition \ref{prop:rt2dimen}.

\bigskip

\begin{proof}[Proof of Proposition \ref{prop:rt2dimen}] Let us consider $A\geq \max(A_6,1),$ and
$s_0\geq \max(s_{14}(A), s_{06}(A)).$ Later in the proof we will fix
$A$ larger,
and $s_0$ larger.
From our conditions on $A$ and
$s_0$,
we already see that the conclusion of Proposition \ref{prop:regu-parab-q-equation} holds, and so does the conclusion of Proposition \ref{prop:projet:q}.
 We then consider $\vq$ a solution of \eqref{qequation} with initial data $\psi_{s_0,d_0,d_1}$ given by (\ref{initial-data}) with $(d_0,d_1)\in \mathcal{D}_T$,
such that
$\vq(s)\in \vartheta_A(s),$ for all $ s\in [s_0,s_1]$
 with $\vq(s_1)\in \partial\vartheta_A(s_1).$

 \bigskip

By definition \ref{prop:V-def} of $\vartheta(s_1)$, it is enough to prove that
\begin{equation}
 \label{1++carre}
|\vq_2(s_1)|<{\sqrt{A}\over s_1^{4\beta-1}},\;\quad  \quad \left\|{\vq_-(y,s_1)\over 1+|y|^3}\right\|_{L^\infty(\R)}<{A\over s_1^\gamma}\; \quad \mbox{and }\;\quad \|\vq_e(s_1)\|_{L^\infty(\R)}<{A^2\over s_1^{\gamma-3\beta}}.
\end{equation}
%
%
%
We begin with $\vq_2.$ Assume by contradiction that
$$\vq_2(s_1)={\sqrt{A}\over s_1^{4\beta-1}}.$$
Since for all $s\in [s_0,s_1],$ $|\vq_2(s)|\leq {\sqrt{A}\over s_1^{4\beta-1}},$ it follows by differentiation that
$$\vq_2'(s_1)\geq -(4\beta-1){\sqrt{A}\over s_1^{4\beta}},$$
on the one hand. On the other hand, since we have
from Proposition \ref{prop:projet:q},
$$\vq_2'(s)=-{2\beta+1\over s}\vq_2+O\left({1\over s^{4\beta}}\right),$$
we see that
$$\vq_2'(s_1) \leq  -{2\beta+1\over s_1}\vq_2(s_1)+{C\over s_1^{4\beta}}\\
=  {C-\sqrt{A}(2\beta+1)\over s_1^{4\beta}}.
$$
This leads to a contradiction with the previous inequality, for $A$ sufficiently large, since $\beta<1$, hence $4\beta-1<2\beta+1$. Similarly, we get a contradiction if $\vq_2(s_1)=-{\sqrt{A}\over s_1^{4\beta-1}}.$
Thus, $|\vq_2(s_1)|<{\sqrt{A}\over s_1^{4\beta-1}}.$

\medskip

To prove that \eqref{1++carre} holds for $\vq_-$ and $ \vq_e,$
it is enough to prove that, for all $s\in [s_0,s_1],$
\begin{equation}
\label{eextintin}
\|\vq_e(s)\|_{L^\infty(\R)}~\leq~ {1\over 2}{A^2\over
s^{\gamma-3\beta}},\; \;\;\;\left\|{\vq_-(y,s)\over 1+|y|^3}
\right\|_{L^\infty(\R)}~\leq~ {1\over 2}{A\over s^{\gamma}},
\end{equation}
provided that $A$ is large enough, and $s_0$ is large enough. Define
$\sigma=\log(A/C_0),$ with $C_0>0$  an appropriate constant to be
chosen later, and take $s_0\geq \sigma,$
so that for all $\tau\geq s_0$ and $s\in
[\tau,\tau+\sigma],$ we have
\begin{equation}
\label{inqss0}
\tau\leq s\leq \tau+\sigma\leq \tau+s_0\leq 2\tau,\; \mbox{ hence }\; {1\over 2\tau}\leq {1\over s}\leq {1\over \tau}\leq {2\over s}.
\end{equation}
We consider two cases in the proof of \eqref{eextintin}.

\medskip

{\bf Case 1: $s\leq s_0+\sigma.$} From our estimates on initial data stated in item (ii) of Proposition \ref{prop:diddc}, it is clear that \eqref{eextintin} is satisfied with $\tau=s_0$, provided that $A\ge 2$. Using Proposition \ref{prop:projet:q} Part (iii), with  $\tau=s_0,$ as well as \eqref{inqss0} and estimate \eqref{e54}, we get
\begin{eqnarray*}
\left\|{\vq_{-}(y,s)\over 1+|y|^3}\right\|_{L^\infty(\R)}
&\leq & C e^{-\frac{(s-s_0)}{2}}\left\|{\vq_{-}(y,s_0)\over 1+|y|^3}\right\|_{L^\infty(\R)}
+\frac C{s^{{3\beta}}}e^{-(s-s_0)^2}\| \vq_{e}(s_0)\|_{L^\infty}+\frac C{s^\gamma}(1+s-s_0)\\
&\leq & \frac C{s^\gamma} e^{-\frac{(s-s_0)}{2}}+\frac C{s^\gamma}(1+s-s_0)\\
&\leq & \frac C {s^\gamma} \left(1+\sigma\right),
\end{eqnarray*}
and
\begin{eqnarray*}
\|\vq_{e}(s)\|_{L^\infty}&\leq & Ce^{-\frac{(s-s_0)}{p}}\|\vq_{e}(s_0)\|_{L^\infty}+Ce^{s-s_0} s^{3\beta}\left\|\frac{\vq_{-}(y,s_0)}{1+|y|^3}\right\|_{L^\infty}+\frac C{s^{\gamma-3\beta}}\left(1+(s-s_0)e^{s-s_0}\right)\\
&\leq & \frac C{s^{\gamma-3\beta}}e^{s-s_0}+\frac C{s^{\gamma-3\beta}}\left(1+(s-s_0)e^{s-s_0}\right)\\
&\leq & \frac C{s^{\gamma-3\beta}}\left(e^{\sigma}+1+\sigma e^\sigma\right).
\end{eqnarray*}
Since $\sigma = \log \frac A{C_0}$, taking $A$ sufficiently large, we get \eqref{eextintin}.

\medskip

{\bf Case 2: $s>s_0+\sigma.$} Let $\tau=s-\sigma>s_0.$ Applying Part (iii) of Proposition \ref{prop:projet:q} and using the fact that $\vq(\tau) \in \vartheta_A(\tau),$ we have
\begin{eqnarray*}
\left\|{\vq_{-}(y,s)\over 1+|y|^3}\right\|_{L^\infty(\R)} &\leq & C e^{-\frac{(s-\tau)}{2}}\left\|{\vq_{-}(y,\tau)\over 1+|y|^3}\right\|_{L^\infty(\R)}
+\frac C{s^{{3\beta}}}e^{-(s-\tau)^2}\| \vq_{e}(\tau)\|_{L^\infty}+\frac C{s^\gamma}(1+s-\tau),\\
&\leq & \frac C{s^\gamma}\left( e^{-\frac{\sigma}{2}}A+e^{-\sigma^2}A^2+1+\sigma\right)
\end{eqnarray*}
and
\begin{eqnarray*}
\|\vq_{e}(s)\|_{L^\infty}&\leq & Ce^{-\frac{(s-\tau)}{p}}\|\vq_{e}(\tau)\|_{L^\infty}+Ce^{s-\tau} s^{3\beta}\left\|\frac{\vq_{-}(\tau)}{1+|y|^3}\right\|_{L^\infty}+\frac C{s^{\gamma-3\beta}}\left(1+(s-\tau)e^{s-\tau}\right)\\
&\leq & \frac{C_0'}{s^{\gamma-3\beta}}\left(e^{-{\sigma\over p}}A^2+e^{\sigma}A+1+\sigma e^\sigma\right),
\end{eqnarray*}
for some $C_0'>0$.\\
Fixing $C_0=10 C_0'$ and $\sigma = \log \frac A{C_0}$, then taking $A$ sufficiently large we see that \eqref{eextintin} follows.

\medskip

In conclusion, we have just proved that \eqref{eextintin} follows in both cases, hence \eqref{1++carre} follows too.
Since $\vq(s_1)\in \partial \vartheta_A(s_1),$ we see that $\left(\vq_0(s_1),\vq_1(s_1)\right)\in
\partial\left( \left[-\f{A}{s_1^{2\beta+1}},
\f{A}{s_1^{2\beta+1}}\right]^2\right)$, by definition \ref{prop:V-def} of $\vartheta_A(s)$.
This concludes the proof of Part (i) of Proposition
\ref{prop:rt2dimen}.

\bigskip

(ii) From Part (i), there exists $m=0,\; 1$ and $\omega=\pm 1$ such that
$$\vq_m(s_1)=\omega{A\over s_1^{2\beta+1}}.$$
Using Part (i) of Proposition \ref{prop:projet:q}, we have that
$$\omega\vq_m'(s_1)\geq \left(1-{m\over 2}\right)\omega\vq_m(s_1)-{C\over s_1^{2\beta+1}}\geq \left[\left(1-{m\over 2}\right)A-C\right]{1\over s_1^{2\beta+1}}.$$
Hence, for $A$ large enough, $\omega\vq_m'(s_1)>0.$ This concludes the proof of Proposition \ref{prop:rt2dimen}.
\end{proof}
\section{Single point blow-up  and final profile}\label{secsing}
\setcounter{equation}{0} In this section we prove Theorem \ref{th1}.
To do so, we first establish a result of no blow-up under some
threshold. This is done in a separate  subsection. In the second
subsection, we give the proof of Theorem \ref{th1}.
\subsection{No blow-up under some threshold}
 In this section, we prove the following result which is similar in the spirit to the result of Giga and Kohn in \cite[Theorem 2.1, p. 850]{GKcpam89}, though we need genuine new parabolic regularity estimates to control the nonlinear gradient term, and this makes the originality of our contribution. This is our statement:
\begin{prop}[No blow-up under some threshold]
\label{similar to Giga-Kohn}
For all $C_0>0,$ there exists $\varepsilon_0>0$ such that if $u(\xi,\tau)$ satisfies for all $|\xi|<1$ and $\tau \in [0,1)$,
\begin{align*}
|\partial_t u-\Delta u|&\leq C_0\left(1+|u|^p+|\nabla u|^q\right),\\
|\partial_t \nabla u-\Delta (\nabla u)-\mu \nabla|\nabla u|^q|&\leq C_0\left(1+|\nabla u||u|^{p-1}\right),
\end{align*}
with $q=\frac{2p}{p+1}$, $p>3$ and $\mu \in \R$, together with the bound
\begin{equation}\label{petit}
|u(\xi,\tau)|+\sqrt{1-\tau}|\nabla u(\xi,\tau)|\leq  \varepsilon_0(1-\tau)^{-{1\over p-1}},
\end{equation}
then $$\forall\;  |\xi|\leq {1/16},\; \forall \; \tau\in [0,1),\;
|u(\xi,\tau)|+|\partial_\xi u(\xi,\tau)|\leq M\varepsilon_0,$$
for some $M=M(C_0, p, \mu)>0$.
In particular, $u$ and $|\nabla u|$ do not blow up at $\xi=0$ and
$\tau=1.$
\end{prop}

We proceed in 2 parts in order to prove this proposition:\\
- In Part 1, we write Duhamel forumlations satisfied by truncations of the solution and its gradient. We also recall from \cite{GKcpam89} an integral computation table.\\
- In Part 2, using the previous estimates and an iteration process, we give the proof of Proposition \ref{similar to Giga-Kohn}.

\bigskip

{\bf Part 1:
 A toolbox for the proof}

In this part, we give a Duhamel formulation for the equations satisfied by $u$ and its gradient.

\medskip

Denoting by $B_r$ the set $\{x\in \R^N\; |\; |x|<r\}$, where $r>0$, we 
establish the following Lemma:
\begin{lem}[Truncation and Duhamel Formulations of the equations]
 \label{utile}
 Let $r\in(0,1]$ and $\phi_r$ be a smooth function supported on $B_r$ such that $\phi_r\equiv 1$ on
$B_{r/2}$ and $0\leq \phi_r\leq 1.$ Let $w= \phi_r u,$ $w_1=\phi_r v,\; v=\nabla u,$ where $u$ is as in Proposition
\ref{similar to Giga-Kohn}.
Then we have the following for all $\xi \in \R$ and $0\le \tau<1$,
\begin{itemize}
\item[(i)] $\left|-\partial_\tau w+\Delta w+u\Delta \phi_{r}-2\nabla\cdot(u\nabla\phi_{r})\right|\le C_0\left(1+|w||u|^{p-1}+
|\nabla u|^q\right)$.
 \item[(ii)]$\left|-\partial_\tau w_1+\Delta w_1+v\Delta \phi_r-2\nabla\cdot(v\nabla \phi_r)
 +\mu \nabla\cdot\left(|v|^{q-2}v w_1\right)-\mu |v|^q\nabla\phi_r\right|\le C_0(1+|w_1||u|^{p-1})$.
 \item[(iii)]
\begin{eqnarray*}
 \|w(\tau)\|_{L^\infty} &\leq & \|u(0)\|_{L^\infty(B_r)}+C\int_0^\tau (\tau-s)^{-1/2}\|u(s)\|_{L^\infty(B_r)}ds \\ &&
 +C\int_0^\tau \|u(s)\|_{L^\infty(B_r)}^{p-1}\|w(s)\|_{L^\infty}ds +C\int_0^\tau \|\nabla u(s)\|_{L^\infty(B_r)}^qds.
\end{eqnarray*}
 \item[(iv)]
 \begin{eqnarray*}
\|w_1(\tau)\|_{L^\infty} &\leq & \|\nabla u(0)\|_{L^\infty(B_r)}+C\int_0^\tau (\tau-s)^{-1/2}\|\nabla u(s)\|_{L^\infty(B_r)}ds
+C\int_0^\tau \|u\|_{L^\infty(B_r)}^{p-1}\|w_1(s)\|_{L^\infty}ds\\ & & + C\int_0^\tau (\tau-s)^{-1/2}
\|w_1(s)\|_{L^\infty}\|\nabla u(s)\|_{L^\infty(B_r)}^{q-1}ds+C\int_0^\tau\|\nabla u(s)\|_{L^\infty(B_r)}^{q}ds.
\end{eqnarray*}
\end{itemize}
\end{lem}
\begin{rem} {\rm The truncated functions $w$ and $w_1$ do depend on $r$, but we omit that dependence in order to avoid complicated notations.}
\end{rem}
\begin{proof}$ $\\
(i)-(ii) The proof is trivial.\\
(iii)-(iv) We only prove item (iii) since item (iv) follows similarly.\\
 Writing the equation given in item (i) in its Duhamel formulation, we see that
\begin{align*}
 &\left|w(\tau) - e^{\tau\Delta}w(0)-\int_0^\tau e^{(\tau-s)\Delta}\left[u\Delta \phi_r-2\nabla(u\nabla \phi_r)\right]ds\right|\\
\le & C
\int_0^\tau e^{(\tau-s)\Delta}\left[|u|^{p-1}|w|\right]ds
 +C\int_0^\tau e^{(\tau-s)\Delta}|\nabla u|^qds,
\end{align*}
where $e^{\tau\Delta}$ is the heat semigroup. Recall first the following well-known smoothing effect of the heat semigroup:
$$\|e^{\tau\Delta}f\|_{L^\infty}\leq \|f\|_{L^\infty},\;
\|\nabla e^{\tau\Delta}f\|_{L^\infty}\leq \frac C{\sqrt\tau}\|f\|_{L^\infty},\; \forall\; \tau>0,\; f\in W^{1,\infty}(\R).$$
Since the truncation $\phi_r$ is supported in the ball $B_r$, only the $L^\infty(B_{r})$ of $u$ and $\nabla u$ are involved, and item (iii) follows.
This concludes the proof of Lemma \ref{utile}.
\end{proof}
Now, we recall the following integral computation table from Giga and Kohn \cite{GKcpam89}:
\begin{lem}[An integral computation table; see Lemma 2.2 page 851 in Giga and Kohn \cite{GKcpam89}]\label{integral} For $0<\alpha<1$, $\theta>0$ and $0\le \tau <1$, the integral
\[
I(\tau) = \int_0^\tau (\tau -s)^{-\alpha}(1-s)^{-\theta} ds
\]
satisfies
\begin{itemize}
\item[(i)] $I(\tau) \le \left((1-\alpha)^{-1}+(\alpha+\theta-1)^{-1}\right) (1-\tau)^{1-\alpha-\theta}$ if $\alpha+\theta>1$,
 \item[(ii)] $I(\tau) \le (1-\alpha)^{-1}+|\log (1-\tau)|$ if $\alpha+\theta=1$,
 \item[(iii)] $I(\tau) \le (1-\alpha-\theta)^{-1}$ if $\alpha+\theta<1$.
\end{itemize}
\end{lem}


{\bf Part 2: The proof of Proposition \ref{similar to Giga-Kohn}}

Using the various estimates of Part 1, we are ready to proceed by iteration to prove Proposition \ref{similar to Giga-Kohn}.

\begin{proof}[Proof of Proposition \ref{similar to Giga-Kohn}]
We follow the iteration method
of Giga and Kohn \cite[Theorem 2.1, p. 850]{GKcpam89}, based on the Duhamel formulations given in Lemma \ref{utile}, though we need here new ideas coming from sharp parabolic estimates.
 We give the proof in 4 steps, improving \eqref{petit} step by step, in order to prove the boundedness of $u$ and $\nabla u$ at the end of the 4th step.

\medskip

{\bf Step 1:} We apply Lemma \ref{utile} with $r=1$. If $w_1=\phi_1 v$ with $v=\nabla u$, then we have from item (iv) in Lemma \ref{utile} together with the bound \eqref{petit}:
%
%
\begin{eqnarray*}
\|w_1(\tau)\|_{L^\infty} &\leq & \|\nabla u(0)\|_{L^\infty(B_1)}+C\int_0^\tau (\tau-s)^{-\frac 12}
 \|v(s)\|_{L^\infty(B_1)}ds
+C\int_0^\tau \|u(s)\|_{L^\infty(B_1)}^{p-1}\|w_1(s)\|_{L^\infty}ds\\ & & + C\int_0^\tau (\tau-s)^{-\frac 12}
\|w_1(s)\|_{L^\infty}\|\nabla u(s)\|_{L^\infty(B_1)}^{q-1}ds+C\int_0^\tau\|\nabla u(s)\|_{L^\infty(B_1)}^{q}ds\\
&\leq & \varepsilon_0+
C\varepsilon_0\int_0^\tau (\tau-s)^{-\frac 12}(1-s)^{-{1\over p-1}-{1\over 2}}ds
+C\varepsilon_0^{p-1}\int_0^\tau (1-s)^{-1}\|w_1(s)\|_{L^\infty}ds\\
 &&+C\varepsilon_0^{q-1}\int_0^\tau
(\tau-s)^{-\frac 12}(1-s)^{-\frac 12}
\|w_1(s)\|_{L^\infty}ds+C\varepsilon_0^{q}\int_0^\tau (1-s)^{-{p\over p-1}}ds.
\end{eqnarray*}
Since $q<p$,
using Lemma \ref{integral},
we see that for $\epsilon_0$ small enough, we have
\[
\|w_1(\tau)\|_{L^\infty} \le C\varepsilon_0(1-\tau)^{-\frac 1{p-1}}+C\varepsilon_0^{q-1}\int_0^\tau (\tau-s)^{-\frac 12}(1-s)^{-\frac 12}\|w_1(s)\|_{L^\infty}ds.
\]
Using a Gronwall's argument together with Lemma \ref{integral},
we see that
\[
\forall \tau\in[0,1),\;\; \|w_1\|_{L^\infty}\leq 2C\varepsilon_0(1-\tau)^{-{1\over p-1}},
\]
for $\varepsilon_0$ sufficiently small.
In particular, since $w_1=\nabla u$ when $|\xi|\le \frac 12$, it follows that
\begin{equation}\label{petit1}
\forall \tau \in [0,1),\;\;\forall |\xi|\le \frac 12,\;\;|\nabla u(\xi,\tau)|\le 2C\varepsilon_0(1-\tau)^{-\frac 1{p-1}},
\end{equation}
which improves the bound on $\nabla u$ in \eqref{petit}, when $|\xi|\le \frac 12$.

\medskip

{\bf Step 2:} Now, we take $r=\frac 12$ and focus on $w=\phi_{1/2 }u$. Using item (iii) in Lemma \ref{utile}, together with \eqref{petit}, \eqref{petit1} and Lemma \ref{integral}, we write
\begin{eqnarray*}
 \|w(\tau)\|_{L^\infty}
& \leq & \|u(0)\|_{L^\infty(B_{1/2})}+C\int_0^\tau
(\tau-s)^{-1/2}\|u(s)\|_{L^\infty(B_{1/2})}ds
\\ &&
+C\int_0^\tau \|u(s)\|_{L^\infty(B_{1/2})}^{p-1}\|w(s)\|_{L^\infty}ds
+C\int_0^\tau \|\nabla u(s)\|_{L^\infty(B_{1/2})}^qds\\ & \leq & \varepsilon_0+C\varepsilon_0\int_0^\tau
(\tau-s)^{-1/2}(1-s)^{-{1\over p-1}}ds
\\ &&
+C\varepsilon_0^{p-1}\int_0^\tau (1-s)^{-1}\|w(s)\|_{L^\infty}ds
+C\varepsilon_0^{q}\int_0^\tau (1-s)^{-q/(p-1)}ds\\ & \leq & C\varepsilon_0
+C\varepsilon_0^{p-1}\int_0^\tau (1-s)^{-1}\|w(s)\|_{L^\infty}ds,
\end{eqnarray*}
(remember that $p>3$).
By Gronwall's inequality, we deduce that
$$\|w(\tau)\|_{L^\infty}\leq C(1-\tau)^{-c\varepsilon_0^{p-1}},\; \forall\tau\in [0,1).$$
Since $w(\xi,\tau)=u(\xi,\tau)$, for all $|\xi|\le \frac 14$, it follows that
\begin{equation}\label{petit2}
\forall \tau \in [0, 1),\;\;\forall |\xi|\le \frac 14,\;\;|u(\xi,\tau)|\le C(1-\tau)^{-c\varepsilon_0^{p-1}}.
\end{equation}

\medskip

{\bf Step 3:} For the step 3 we take $r=1/4.$ We now consider $\bar{w}_1=\phi_{1/4}v:=\phi_{1/4}\nabla u.$ Applying item (iv) in Lemma \ref{utile} and using the bounds \eqref{petit1} and \eqref{petit2}, together with Lemma \ref{integral}, we see that
\begin{eqnarray*}
\|\bar{w}_1(\tau)\|_{L^\infty} &\leq &
\|\nabla u(0)\|_{L^\infty(B_{1/4})}
+C\int_0^\tau (\tau-s)^{-\frac 12}\|\nabla u(s)\|_{L^\infty(B_{1/4})}ds\\
&&+C\int_0^\tau \|u(s)\|_{L^\infty(B_{1/4})}^{p-1}\|\bar w_1(s)\|_{L^\infty}ds +C\int_0^\tau (\tau-s)^{-\frac 12}
\|\bar w_1(s)\|_{L^\infty}\|\nabla u(s)\|_{L^\infty(B_{1/2})}^{q-1}ds\\
&&+C\int_0^\tau\|\nabla u(s)\|_{L^\infty(B_{1/4})}^{q}ds \\
&\leq &
\varepsilon_0+C\varepsilon_0\int_0^\tau (\tau-s)^{-1/2}(1-s)^{-1/(p-1)}ds
+C\int_0^\tau (1-s)^{-c(p-1)\varepsilon_0^{p-1}}\|\bar w_1(s)\|_{L^\infty}ds\\ & & +C\varepsilon_0^{q-1}
\int_0^\tau (\tau-s)^{-1/2}(1-s)^{-1/(p+1)}
\|\bar w_1(s)\|_{L^\infty}ds+C\varepsilon_0^{q}\int_0^\tau(1-s)^{-2p/(p^2-1)}ds \\
&\leq &
C\varepsilon_0+C\varepsilon_0^{q-1}
\int_0^\tau (\tau-s)^{-1/2}(1-s)^{-1/(p+1)}
\|\bar w_1(s)\|_{L^\infty}ds
\end{eqnarray*}
for $\epsilon_0$ small enough (remember that $p>3$).
Using again Gronwall's technique, together with Lemma \ref{integral},  we see that
$$\|\bar w_1(\tau)\|_{L^\infty}\leq C\varepsilon_0,\; \forall \; \tau\in [0,1).$$
Since $\bar w_1(\xi,\tau) = \nabla u(\xi,\tau)$ when $|\xi|\le \frac 18$, it follows that
\begin{equation}\label{petit2bis}
\forall \tau \in [0,1),\;\;
\forall |\xi|\le \frac 18,\;\;
|\nabla u(\xi,\tau)|\le C \varepsilon_0.
\end{equation}

\medskip

{\bf Step 4:} For the step 4 we take $r=1/8.$ We now consider $\tilde{w}=\phi_{1/8}u$ and use item (iii) in Lemma \ref{utile}, the bounds in \eqref{petit}, \eqref{petit1} and \eqref{petit2bis}, together with Lemma \ref{integral} to derive that
\begin{eqnarray*}
 \|\tilde{w}(\tau)\|_{L^\infty} & \leq & \|u(0)\|_{L^\infty(B_{1/8})}+C\int_0^\tau
(\tau-s)^{-1/2}\|u(s)\|_{L^\infty(B_{1/8})}ds \\ &&
+C\int_0^\tau \|u(s)\|_{L^\infty(B_{1/8})}^{p-1}\|\tilde{w}(s)\|_{L^\infty}ds
+C\int_0^\tau \|\nabla u(s)\|_{L^\infty(B_{1/8})}^qds\\ & \leq & C\varepsilon_0+
C\int_0^\tau (1-s)^{-c(p-1)\varepsilon_0^{p-1}}\|\tilde{w}(s)\|_{L^\infty}ds
\end{eqnarray*}
(use the fact that $p>3$).
Again, by Gronwall's argument, we see that
$$\|\tilde{w}(\tau)\|_{L^\infty}\leq M\varepsilon_0,\; \forall \; \tau\in [0,1).$$
Since $\tilde w(\xi,\tau) = u(\xi, \tau)$, for all $|\xi|\le \frac 1{16}$ and $\tau \in [0,1)$, it follows that
\begin{equation}\label{petit3}
\forall \tau \in [0,1),\;\;\forall |\xi|\le \frac 1{16},\;\;
|u(\xi,\tau)|\le C \varepsilon_0.
\end{equation}
Using \eqref{petit2bis} and \eqref{petit3}, we get the conclusion of the proof of Proposition \ref{similar to Giga-Kohn}.
\end{proof}

\subsection{Proof of Theorem \ref{th1}}
This section is dedicated to the proof of Theorem \ref{th1}. We will present the proofs of items (i), (ii) and (iii) separately.
\begin{proof}[Proof of Theorem \ref{th1} Part (i)]
If we introduce for $\epsilon>0$
\begin{equation}\label{defgamma}
\gamma = \min(5\beta-1, 2\beta+1)-\epsilon,
\end{equation}
and recall that $p>3$, then we see that \eqref{gamma} holds, provided $\epsilon$ is small enough.  Therefore, our strategy in Section \ref{secexis} works, and we get from Propositions \ref{dans vs} and \ref{prop:V-def-prop} the existence of a solution $\vq$ to equation \eqref{qequation}, defined for all $y\in \R$ and $s\ge s_0$, for some $s_0\ge 1$, such that
\[
\forall s\ge s_0,\;\;\|\vq(s)\|_{W^{1,\infty}(\R)}\leq {C\over s^{\gamma-3\beta}}.
\]
Then, using \eqref{eq:q+phi+def}, this yields $w(y,s)$, a solution to equation \eqref{eq:wequation}, defined for all $y\in \R$ and $s\ge s_0$, such that
$$\|w(y,s)-\varphi(y,s)\|_{W^{1,\infty}(\R)}\leq {C\over s^{\gamma-3\beta}},$$
where the profile $\varphi$ is introduced in \eqref{feq}.\\
Introducing
\[
T=e^{-s_0}
\]
and the function $u(x,t)$ defined from $w(y,s)$ by \eqref{framwork:selfsimilar}, we see that $u$ is a solution to equation \eqref{eq:uequation} defined for all $x\in \R$ and $t\in [0,T)$, which satisfies
\begin{equation}\label{profile}
\forall t\in [0,T),\;\;\left\|(T-t)^{1\over p-1}u(y\sqrt{T-t},t)-\vpz\left({y\over
|\log(T-t)|^\beta}\right)\right\|_{W^{1,\infty}(\R)}\leq
{C\over |\log(T-t)|^{\gamma-3\beta}}
\end{equation}
(use here the fact that $\gamma-3\beta<2\beta$ which comes from \eqref{gamma}).\\
In particular, $u(t)\in W^{1,\infty}(\R)$, for all $t\in [0, T)$, and
$\lim_{t\to T}(T-t)^{{1\over p-1}}u(0,t)=(p-1)^{{1\over p-1}},$
hence  $\lim_{t\to T}u(0,t)=\infty$, which means that $u$ blows up at time $t=T$, (at least) at the origin.
Moreover,
we have
$$\|u(t)\|_{L^\infty(\R)}\leq C (T-t)^{-{1\over p-1}},\; \forall\; t\in [0,T).$$
We now turn to proving that $|\nabla u(x,t)|$ blows up at the origin.
By Proposition \ref{prop:regu-parab-q-equation}, part (ii), and \eqref{eq:q+phi+def},
we have for all $s\ge -\log T$ and $y\in \R$,
\[
|\nabla w(y,s)-\nabla\varphi(y,s)|=|\nabla v(y,s)|\leq C\frac{A}{s^{2\beta+1}}(1+|y|)+
C\frac{\sqrt{A}}{{s^{4\beta-1}}}(1+|y|^2)+C
\frac{A^2}{s^{\gamma}}(1+|y|^3).
\]
Put $y:=y(s)=s^\alpha,$ with $\; 0<\alpha<\min(2\beta-1,{\gamma-2\beta\over 2}).$ Then, from \eqref{e20151}, we see that
$$\partial_y\varphi(y(s),s)=-{2b\over (p-1)}{y(s)\over s^{2\beta}}\left[\vpz\left({y(s)\over s^\beta}\right)\right]^p\sim {C\over s^{2\beta-\alpha}}\mbox{ as } s\to\; \infty.$$
The conditions on $\alpha$ imply that $|\nabla w(y(s),s)|\sim Cs^{\alpha-2\beta},$ as $\; s\to\; \infty.$ By the relation between
$w,\; u$; $y,\; s,\; t$ and $x$ given in \eqref{framwork:selfsimilar}, we get
$$\left|\nabla u\left(\sqrt{T-t}|\log(T-t)|^\alpha,t\right)\right|\sim C(T-t)^{-{1\over 2}-{1\over p-1}}\left|\log(T-t)\right|^{\alpha-2\beta}\to
\infty \mbox{ as }
t\to T.$$
In particular, $$\|\nabla u(t)\|_{L^\infty(\R)}\geq C \left(T-t\right)^{-{1\over 2}-{1\over p-1}}\left|\log(T-t)\right|^{\alpha-2\beta}.$$
Since $\sqrt{T-t}|\log(T-t)|^\alpha\to 0$ as $t\to T,$ $\nabla u$ blows up at time $T,$ (at least) at the origin.\\
Since \eqref{bupprof1} follows from \eqref{profile} and \eqref{defgamma}, this concludes the proof of item (i) of Theorem \ref{th1}.
\end{proof}

\begin{proof}[Proof of Theorem \ref{th1} Part (ii)] Consider  $x_0\not=0$.   By Part (i), we have that
$$\|w(\cdot,s)-\varphi(\cdot,s)\|_{W^{1,\infty}(\R)}\leq {C\over s^{\gamma-3\beta}},$$
for all $s\geq s_0=-\log T$. Using the relation between $w,$  $u,$ $y,\;
s$ and $x,\; t$ given by \eqref{framwork:selfsimilar}, and by the
definition of $\varphi$ given by \eqref{feq}, we get that
\begin{equation}
\label{11cercle} \sup_{x\in \R}\left(T-t\right)^{{1\over
p-1}+{1\over 2}}|\nabla u(x,t)|\leq {\|\nabla
\vpz\|_{L^\infty(\R)}\over \left|\log(T-t)\right|^{\beta}}+
{C\over \left|\log(T-t)\right|^{\gamma-3\beta}}\to 0, \; \mbox{ as }
t\to T,\end{equation} and
\begin{equation}
\label{12cercle}\sup_{|x-x_0|<{|x_0|\over
2}}\left(T-t\right)^{{1\over p-1}}|u(x,t)|\leq
\left|\vpz\left({|x_0|/2\over
\sqrt{T-t}|\log(T-t)|^\beta}\right)\right|+{C\over
\left|\log(T-t)\right|^{\gamma-3\beta}}\to 0,\; \mbox{ as } t\to T.
\end{equation}
Consider $\delta>0$ to be chosen small enough so various estimates
hold. Then,
for $K_0>0$ to be fixed later,
we define $t_0(x_0)$
by:
\begin{align}
&|x_0|=K_0\sqrt{T-t_0(x_0)}\left|\log\left(T-t_0(x_0)\right)\right|^{\beta},
&\mbox{ if }|x_0|\le \delta, \label{eqt0}\\
&t_0(x_0)=t_0(\delta), &\mbox{ if }|x_0|> \delta.\nonumber
\end{align}
Note that $t_0(x_0)$ is unique if $|x_0|$ is sufficiently
small. Note also that $t_0(x_0)\to T$ as $x_0\to 0.$ Let us
introduce the rescaled functions
\begin{equation}
\label{eqdeU} U(x_0,\xi,\tau)= \left(T-t_0(x_0)\right)^{{1\over
p-1}}u(x,t),
\end{equation}
\begin{equation}
 \label{newV}
 V(x_0,\xi,\tau):=\nabla_\xi U(x_0,\xi,\tau),
\end{equation}
where
\begin{equation}
\label{lexett} x=x_0+\xi \sqrt{T-t_0(x_0)},\; t=t_0(x_0)+\tau
(T-t_0(x_0)),\; \xi\in \R, \tau\in [-{t_0(x_0)\over T-t_0(x_0)}, 1).
\end{equation}
It is easy to see from \eqref{11cercle} and \eqref{12cercle} that
Proposition \ref{similar to Giga-Kohn} applies to
$U(x_0,\cdot,\cdot),$ hence that $x_0$ is not a blow-up point of $u$
and $\nabla u.$  Let us insist on the fact that our argument works
for any $x_0\not=0$ without any smallness assumptions, thanks to the
adapted definition of $t_0(x_0)$ in \eqref{eqt0}. This proves the
single point blow-up result. Thus we deduce that $u$ and $\nabla u$
blow up simultaneously at time $T$ and only at $x=0.$
\end{proof}

\begin{proof}[Proof of Theorem \ref{th1} Part (iii)] We divide the proof into two parts.
In the first part we show the existence of the final profile $u^*.$
In the second Part, we find an equivalent of $u^*$, and bound $|\partial_x u^*|$.

\medskip

{\it {\bf Part 1: Existence of the final profile.}} \label{pagepart1} In this part, we
show the existence of a blow-up final profile $u^*\in
C^1\left(\R\setminus \{0\}\right)$ such that $u(x,t)\to u^*,$ as
$t\to T,$ in $C^1$ of every compact
of $\R\setminus\{0\}$.
The case
$\mu=0$ is treated in \cite{Zihp98}. In comparison, with the case
$\mu=0,$ we need to use advanced parabolic regularity as in
Proposition \ref{similar to Giga-Kohn}, which is the
extended Giga-Kohn no-blow-up result for parabolic equations with a
nonlinear gradient term.

If $v=\nabla u,$ then, we derive from equation \eqref{eq:uequation}
the following system satisfied by $(\partial_t u,\partial_t v):$
\begin{equation}
 \label{eq:u}
\partial_t u  = \Delta u+\mu|v|^q+|u|^{p-1}u
\end{equation}
\begin{equation}
 \label{eq:v}
\partial_t v  = \Delta v+\mu \nabla\left(|v|^q\right)+p|u|^{p-1}v.
\end{equation}
Since we know from Part (ii) that $u$ and $v$ are bounded uniformly
on $\mathcal{K}\times [0,T)$ for any compact set $\mathcal{K}\subset
\R\setminus \{0\},$ using parabolic regularity techniques, similar
to Proposition \ref{similar to Giga-Kohn}, we can show that
$\partial_t u$ and $\partial_t v$ are also bounded on
$\mathcal{K}'\times [T/2,T)$ for any $\mathcal{K}'\subset
\R\setminus \{0\}.$ Therefore as in \cite{MerleCPAM92}, there
exists $u^*$ in $C^1\left(\R\setminus \{0\}\right)$ such that
$u(t,x)\to u^*(x)$ and $\partial_x u(x,t)\to \partial_x u^*(x)$ as $t\to T,$
uniformly on compact sets of $\R\setminus \{0\}.$

\medskip

{\it {\bf Part 2: Equivalent of the final profile.}} Let us now find
an equivalent of $u^*$ as $x\to 0.$
Consider $x_0\neq 0$.
Since $u$ is a solution of the
equation \eqref{eq:uequation} and $q=2p/(p+1)$,
it follows that
$U$ and $V$
defined in \eqref{eqdeU}-\eqref{newV} satisfy the equations:
\begin{equation}
 \label{eq:U}
\partial_\tau U  = \Delta_\xi U+\mu|\nabla_\xi U|^q+|U|^{p-1}U,
\end{equation}
\begin{equation}
 \label{eq:V2}
\partial_\tau V  = \Delta_\xi V+\mu \nabla_\xi\left(|V|^q\right)+p|U|^{p-1}V.
\end{equation}
By Part (i) of Theorem \ref{th1}, we have:
\begin{equation}
\label{6cercle}
\sup_{|\xi|\leq 6\left|\log\left(T-t_0(x_0)\right)\right|^{\beta/2}}\left|U(x_0,\xi,0)-\vpz(K_0)\right|\leq
{C\over \left|\log\left(T-t_0(x_0)\right)\right|^{\gamma'}},
\end{equation}
with
\begin{equation}
\label{7cercle}
\gamma'=\min(\beta/2,\gamma-3\beta),
\end{equation}
 and
\begin{equation}
\label{4cercle}\sup_{|\xi|\leq 6\left|\log\left(T-t_0(x_0)\right)\right|^{\beta/2}}\left|V(x_0,\xi,0)\right|\leq
{C\over \left|\log\left(T-t_0(x_0)\right)\right|^{\gamma-3\beta}}.
\end{equation}
Using \eqref{11cercle} and \eqref{12cercle}, we see that for all
$\tau\in [0,1)$ and $|\xi|\leq 6\left|\log\left(T-t_0(x_0)\right)\right|^{\beta/2}$, we have
\[
(1-\tau)^{{1\over p-1}}\left[|U(x_0,\xi,\tau)|+\sqrt{1-\tau}|V(x_0,\xi,\tau)|\right]\leq
 \varphi^0\left(\frac{K_0}2\right)+\frac C{|\log(T-t_0(x_0))|^{\gamma-3\beta}}\equiv \bar \varepsilon_0(K_0,x_0),
\]
with
 \begin{equation}
 \label{cercle}
 \overline{\varepsilon_0}\to 0\;  \mbox{ as }\; |x_0|\to 0 \mbox{ and }\; K_0\to \infty.
 \end{equation}
Fixing $K_0$ large enough and $|x_0|$ small enough, we can make $\bar \epsilon_0(K_0, x_0)\le \epsilon_0$, where $\epsilon_0$ is defined in Proposition \ref{similar to Giga-Kohn}. Applying
%
%
Proposition \ref{similar to Giga-Kohn}, we deduce
 that for all  $\tau\in [0,1],$
\begin{equation}
 \label{2cercle}\sup_{|\xi|\leq 5\left|\log(T-t_0(x_0))\right|^{\beta/2}}\left[|U(x_0,\xi,\tau)|+|V(x_0,\xi,\tau)|\right]\leq M_0\equiv M\overline{\varepsilon_0},
 \end{equation}
for some $M>0$.
With these estimates, we proceed in three steps to finish the proof of item (iii) in Theorem \ref{th1}, using a truncation argument, as for Proposition \ref{similar to Giga-Kohn}, and recalling the definition of $\gamma'$ given in \eqref{7cercle}:

- In Step 1, we show that
\begin{equation}
\label{8cercle}
\forall \; \tau\in [0,1),\;\forall \; |\xi|\leq 2\left|\log\left(T-t_0(x_0)\right)\right|^{\beta/2},\; |V(x_0,\xi,\tau)|\leq {C\over
\left|\log(T-t_0(x_0))\right|^{\gamma'}}.
\end{equation}

- In Step 2, we show that
\begin{equation}
\label{9cercle}
\forall \; \tau\in [0,1),\; \forall \; |\xi|\leq \left|\log\left(T-t_0(x_0)\right)\right|^{\beta/2},\; |U(x_0,\xi,\tau)-\overline{U}_{K_0}|\leq {C\over
\left|\log(T-t_0(x_0))\right|^{\gamma'}},
\end{equation}
where $\overline{U}_{K_0}$ is the solution of the ordinary differential equation
$$ \overline{U}_{K_0}'(\tau)=\overline{U}_{K_0}^p(\tau),$$
with initial data
$$\overline{U}_{K_0}(0)= \vpz(K_0),$$ given by
\begin{equation}
\label{3cercle}
\overline{U}_{K_0}(\tau)=\left((p-1)(1-\tau)+bK_0^2\right)^{-{1\over p-1}}.
\end{equation}

- In Step 3, we justify that
\begin{equation}
\label{10cercle}
\forall,\; \tau\in [{1\over 2},1),\; \forall\; |\xi|\leq {1\over 2}\left|\log\left(T-t_0(x_0)\right)\right|^{\beta/2},\; |\partial_\tau U(x_0,\xi,\tau)|+|\partial_\tau V(x_0,\xi,\tau)|\leq C,
\end{equation}
which yields limits for $U$ and $V$ as $\tau\to 1,$ hence for $u$ and $\partial_x u$ as $t\to T,$ by definitions \eqref{eqdeU} and \eqref{newV} of $U$ and $V.$

Let $\phi\in C^{\infty}(\R)$ be a cut-off function, with $\mbox{supp}\;\phi\subset B(0,1),$ $0\leq \phi\leq 1$ and $\phi\equiv 1$ on $B(0,1/2).$
Introducing
$$\phi_r(\xi)=\phi\left({\xi\over r \left|\log(T-t_0(x_0))\right|^{\beta/2}}\right),$$
we see that
\begin{equation}\label{der}
\|\nabla \phi_r\|_{L^\infty(\R)}\le \frac C{|\log(T-t_0(x_0))|^{\frac \beta 2}}
\mbox{ and }
\|\Delta \phi_r\|_{L^\infty(\R)}\le \frac C{|\log(T-t_0(x_0))|^{\beta}}.
\end{equation}

\bigskip

{\it Step 1: Proof of \eqref{8cercle}.} Let us consider $r=2$ and
introduce
\begin{equation}
\label{5cercle}
V_2=\phi_2 V.
\end{equation}
Then,
arguing as for
Lemma \ref{utile}, we have
for all $\xi \in \R$ and $\tau \in [0,1)$,
$$\partial_\tau V_2=\Delta_\xi V_2+V\Delta_\xi \phi_2-2\nabla_\xi(V\nabla_\xi \phi_2)+p|U|^{p-1}V_2+
\mu\nabla_\xi\left(|V|^{q-2}VV_2\right)
-\mu|V|^q\nabla_\xi \phi_2.$$
 Taking the $L^\infty$-norm on  the Duhamel equation satisfied by $V_2$, then using \eqref{der}, \eqref{2cercle} and \eqref{4cercle}, we get for all $\tau \in [0,1)$,
\begin{eqnarray*}
 \|V_2(\tau)\|_{L^\infty(\R)} & \leq  &\|V_2(0)\|_{L^\infty(\R)}+\frac{CM_0}{ \left|\log(T-t_0(x_0))\right|^{\beta}}
 +\frac{CM_0}{\left|\log(T-t_0(x_0))\right|^{\beta/2}} \\ &&
 +M_0^{p-1}\int_0^\tau \|V_2(s)\|_{L^\infty(\R)} ds
 +C M_0^{q-1} \int_0^\tau \frac{\|V_2(s)\|_{L^\infty(\R)}}{\sqrt{\tau-s}}  ds
+\frac{CM_0^q}{\left|\log(T-t_0(x_0))\right|^{\beta/2}} \\ & \leq  &
 {C\over \left|\log(T-t_0(x_0))\right|^{\gamma'}} +C\eta_0 \int_0^\tau  \frac{\|V_2(s)\|_{L^\infty(\R)}}{\sqrt{\tau-s}} ds,
\end{eqnarray*}
where $\eta_0=\max(M_0^{p-1},\; M_0^{q-1}, |\log(T-t_0(x_0))|^{-\frac \beta 2}).$
Since $\eta_0$ can be made sufficiently small by taking $|x_0|$ small enough and $K_0$ large enough (see \eqref{cercle} and \eqref{2cercle}), using Gronwall's inequality, we deduce
that
$$\|V_2(\tau)\|_{L^\infty(\R)} \leq  {2C\over \left|\log(T-t_0(x_0))\right|^{\gamma'}},\; \forall \tau\in [0,1),$$
and \eqref{8cercle} follows.

\bigskip

{\it Step 2: Proof of \eqref{9cercle}.} Let us consider $r=1$ and
introduce
 $$\mathcal{U}=U-\overline{U}_{K_0},\; \mathcal{U}_1=\phi_1 \mathcal{U}.$$
Then $\mathcal{U}_1$ satisfies the equation
$$\partial_\tau\mathcal{U}_1=\Delta \mathcal{U}_1+\mathcal{U}\Delta \phi_1-2\nabla(\mathcal{U}\nabla\phi_1)+a\mathcal{U}_1+
\mu |\nabla U|^{q}\phi_1,$$
where
$$a(x_0,\xi,\tau)={|U(x_0,\xi,\tau)|^{p-1}U(x_0,\xi,\tau)-\overline{U}_{K_0}^p(\tau)\over U(x_0,\xi,\tau)-\overline{U}_{K_0}},\; \mbox{ if }\; U(x_0,\xi,\tau)\not=\overline{U}_{K_0},$$
$$a(x_0,\xi,\tau)= p\overline{U}_{K_0}^{p-1}(\tau),\;\; \mbox{ Otherwise},$$
and satisfies
$$|a(x_0,\xi,\tau)|\leq C\eta_1$$
where $\eta_1=\max(K_0^{-2},M_0^{p-1})\to 0$ as $|x_0|\to 0$ and $K_0\to \infty$, by \eqref{3cercle} and \eqref{2cercle}. Using \eqref{6cercle}, \eqref{8cercle}, \eqref{der} and the Duhamel equation on $\mathcal{U}_1,$ we get
\begin{eqnarray*}
 \|\mathcal{U}_1(\tau)\|_{L^\infty(\R)} & \leq  &\|\mathcal{U}_1(0)\|_{L^\infty(\R)}+\frac{CM_0}{  \left|\log(T-t_0(x_0))\right|^{\beta}}
 +
\frac{CM_0}{\left|\log(T-t_0(x_0))\right|^{\beta/2}}  \\ && +
 C\eta_1\int_0^\tau \|\mathcal{U}_1(s)\|_{L^\infty(\R)} ds
 + {C \over \left|\log(T-t_0(x_0))\right|^{q\gamma'}} \\ & \leq  &
 {C\over \left|\log(T-t_0(x_0))\right|^{\gamma'}} +C \eta_2\int_0^\tau \|\mathcal{U}_1(s)\|_{L^\infty(\R)} ds,
\end{eqnarray*}
where $\eta_2=\min(\eta_1,1/\left|\log(T-t_0(x_0))\right|^{\beta/2}) .$ Since $\eta_2\to 0$ as $|x_0|\to 0$ and $K_0\to \infty,$ using Gronwall's inequality, we deduce  that
$$\|\mathcal{U}_1(\tau)\|_\infty \leq  {C\over \left|\log(T-t_0(x_0))\right|^{\gamma'}},\; \forall \tau\in [0,1),$$
and \eqref{9cercle} follows.

\bigskip

{\it Step 3: Proof of \eqref{10cercle} and conclusion of the Proof of
Theorem \ref{th1}.} From \eqref{eq:U} and \eqref{eq:V2}, we write
the following system verified by $(\partial_\tau U,\partial_\tau
V),$ for all $\tau\in [0,1)$ for all $\xi\in \R,$
$$\partial_\tau(\partial_\tau U)  = \Delta_\xi (\partial_\tau U)+q\mu (\partial_\tau V) |V|^{q-2}V+p|U|^{p-1}(\partial_\tau U)$$
$$\partial_\tau(\partial_\tau V)  = \Delta_\xi (\partial_\tau V) +\mu q\nabla_\xi\left((\partial_\tau V)|V|^{q-2}V\right)+p\nabla\left(|U|^{p-1}\partial_\tau U\right).$$
Using \eqref{8cercle} and \eqref{9cercle}, and classical parabolic regularity, as in Steps 1 and 2,
we see that \eqref{10cercle} follows.
 Then, as in \cite{MerleCPAM92}, we
have that $\lim_{\tau\to 1}U(x_0,0,\tau)$ and $\lim_{\tau\to
1}V(x_0,0,\tau)$ exist for $x_0$ sufficiently small. Moreover, using
\eqref{8cercle} and \eqref{9cercle}, we see that
\begin{equation}
\label{11cercleenrouge} \lim_{\tau \to 1}U(x_0,0,\tau)\sim
\overline{U}_{K_0}(1)= \left(bK_0^2\right)^{-{1\over p-1}},\;
\mbox{ as } x_0\to 0,
\end{equation}
and
\begin{equation}
\label{12cercleenrouge} |\lim_{\tau \to 1}V(x_0,0,\tau)|\leq {C\over
\left|\log(T-t_0(x_0))\right|^{\gamma'}},
\end{equation}
for $|x_0|$ small enough, where $\gamma'$ is defined in
\eqref{7cercle}. Recall from the beginning of the proof Part (iii)
of Theorem \ref{th1} (see Part 1 page \pageref{pagepart1}), that  $\lim_{t\to T}u(x_0,t):=u^*(x_0)$ and
$\lim_{t\to T}\nabla u(x_0,t):=
 \partial_xu^*(x_0).$ Therefore,
from the definitions \eqref{eqdeU} and \eqref{newV} of $U$ and $V$,
we write
 \begin{equation*}
 u^*(x_0)=\lim_{t\to T}u(x_0,t) =  \lim_{\tau\to 1}{U(x_0,0,\tau)\over
 \left(T-t_0(x_0)\right)^{{1\over p-1}}}
 \sim  \left(bK_0^2\right)^{-{1\over p-1}}\left(T-t_0(x_0)\right)^{-{1\over p-1}},\;
 \mbox{ as } x_0\to 0,
 \end{equation*}
 and
\begin{equation*}
|\partial_xu^*(x_0)| = \left|\lim_{\tau \to 1}\partial_x u(x_0,t)\right|
 =  \left|\lim_{\tau \to 1}{V(x_0,0,\tau)\over
\left(T-t_0(x_0)\right)^{{1\over p-1}+{1\over 2}}} \right| \leq
 {C\over \left(T-t_0(x_0)\right)^{{1\over p-1}+{1\over
2}}\left|\log(T-t_0(x_0))\right|^{\gamma'}}.
\end{equation*}
Using  \eqref{eqt0}, we have
\[
 \log|x_0|\sim {1\over 2}\log(T-t_0(x_0))
 \mbox{ and }
 T-t_0(x_0) \sim {|x_0|^2\over 2^{2\beta}K_0^2|\log|x_0||^{2\beta}},\; \mbox{ as } x_0\to 0.
 \]
Hence,
 $$u^*(x_0) \sim \left(b|x_0|^2\over \big[2\left|\log|x_0|\right|\big]^{2\beta}
 \right)^{-{1\over p-1}},\; \mbox{ as } x_0\to 0,$$
and
$$|\partial_xu^*(x_0)|\leq {C |x_0|^{-{p+1\over p-1}} \over \left|\log|x_0|\right|^{\gamma'-
2\beta^2}}.$$
Since $\gamma=\min\left(2\beta+1,5\beta-1\right)-\varepsilon$ with $\varepsilon>0$ by \eqref{defgamma} and $\gamma'=\min(\gamma-3\beta, \frac \beta 2)$ by \eqref{7cercle},
we easily see that

$$\gamma'-2\beta^2=
\left\{
  \begin{array}{ll}
    {1-3p\over (p-1)^2}-\varepsilon,\; \mbox{ if }\; 3<p\leq 7 & \\
    {-p^2+2p-5\over 2(p-1)^2}-\varepsilon,\; \mbox{ if }\; p> 7. &\\
  \end{array}
\right.
$$
This concludes the proof of Theorem \ref{th1}.
 \end{proof}

 \section{Stability} In this section we will prove Theorem \ref{th2}. In fact, the stability is a
 natural by-product of the existence proof, thanks to a geometrical interpretation
 of the parameters of the finite-dimensional problem (i.e.
 $(d_0,d_1)$ in \eqref{initial-data}) in terms of the blow-up time and the
 blow-up point.

\bigskip



Let us first explain the strategy of the proof, and leave the
technical details for the following subsection. Finally, in the third section, we briefly conclude the proof of Theorem \ref{th2}, then state a stronger version, valid for blow-up solutions having the profile \eqref{145} only for a subsequence (see Theorem \ref{th2}' page \pageref{th2'} below).

\subsection{Strategy of the proof}\label{substrat}

 Let us consider $\hat{u}$ the
constructed solution of equation \eqref{eq:uequation} in Theorem
\ref{th1}, and call $\hat{u}_0$ its initial data in
$W^{1,\infty}(\R),$ and $\hat{T}$ its blow-up time.
From the
construction method in Section 4 (see Proposition \ref{dans vs}),
consider $\hat{A}\geq 1$ such that
\begin{equation}\label{vva}
\forall\; s\geq -\log
\hat{T},\; \hat{\vq}(s)\in \vartheta_{\hat{A}}(s),
\end{equation}
 where
\begin{equation}\label{defvc}
\hat{\vq}(y,s)=\hat{w}(y,s)-\varphi(y,s),\; \hat{w}(y,s)=e^{-{s\over p-1}}
\hat{u}\left(ye^{-{s\over 2}},\hat{T}-e^{-s}\right)
\end{equation}
 and $\varphi$
is defined in \eqref{feq} (here and throughout this section, we consider the constant $K$ defining the truncation in \eqref{def:chi} as fixed). Now, we consider $u_0\in
W^{1,\infty}(\R)$ such that $\|\varepsilon_0\|_{W^{1,\infty}(\R)}$
is small, where
\begin{equation}
\label{tunis4triangle} \varepsilon_0=u_0-\hat{u_0}.
\end{equation}
We denote by $u_{u_0}$ the solution of equation \eqref{eq:uequation}
with initial data $u_0$ and $T(u_0)\le +\infty$ its maximal time of existence,
from the Cauchy theory in $W^{1,\infty}(\R).$

\medskip

Our aim is to show that, for some $A_0>0$ and $\sigma_0\geq -\log(\hat{T}),$ large enough, if $\varepsilon_0$ is small enough, then
$T(u_0)$ is finite, and $u_{u_0}$ blows up at time $T(u_0)$ only at
one blow-up point $a(u_0),$ with
\begin{equation}
\label{tunis2triangle} T(u_0)\to \hat{T},\; a(u_0)\to 0\; \mbox{ as
}\;
\varepsilon_0=u_0-\hat{u_0}\to 0\mbox{ in }W^{1,\infty}(\R),
\end{equation}
and
$$\vq_{T(u_0),\,a(u_0),\,u_0}(s)\in \vartheta_{A_0}(s)$$
for $s\geq s_0$ large enough,
where, for any $(T,a)\in \R^2,$ we introduce
\begin{equation}
\label{tunistriangle3}
\vq_{T,\,a,\,u_0}(y,s)=w_{T,\,a,\,u_0}(y,s)-\varphi(y,s),
\end{equation}
and $w_{T,\,a,\,u_0}$ is the similarity variable version centered
at $(T,a)$ of $u_{u_0}(x,t),$ the solution of equation
\eqref{eq:uequation} with initial data $u_0.$ More precisely, we have
\begin{equation}
\label{tunistriangle1}
w_{T,\,a,\,u_0}(y,s)=(T-t)^{\frac{1}{p-1}}u(x,t),\; y=(x-a)/{\sqrt{T-t}},~
  s=-\log{(T-t)}.
  \end{equation}
Indeed, with the estimates of Section 5,
 we deduce that $u_{u_0}$ satisfies the same estimates as
 ${\hat{u}},$ given in Theorem \ref{th1}, which is the desired
 conclusion of Theorem \ref{th2}. Note that for the moment, we don't even
 know that $T(u_0)$ is finite, hence asking the question of the existence of $a(u_0)$
 is non relevant and talking about $\vq_{T(u_0),\,a(u_0),\,u_0}$ is meaningless at this stage.

\medskip

 Anticipating our aim in \eqref{tunis2triangle}, we will study $\vq_{T,\,a,\,u_0},$ where
 $(T,a)$ is arbitrary in a small neighborhood of $(\hat{T}, 0),$ hoping to have some hint
 that some particular value $(\bt(u_0),\ba(u_0))$ close to $(\hat{T}, 0),$ will correspond
 to the aimed $(T(u_0),a(u_0)).$
Note that all the $\vq_{T,\,a,\,u_0}$ satisfy the same equation,
 namely \eqref{qequation}, for all $(y,s)\in \R\times [-\log T, -\log(T-T(u_0))_+)$
 (by convention, we note $-\log(0_+)=\infty$).
Introducing
\[
\epsilon(x,t) = u(x,t) - \hat u(x,t),\;\; \mbox{ for all } x\in \R\mbox{ and }0\le t <\min(T(u_0),\hat T),
\]
we see from \eqref{tunistriangle3}, \eqref{tunistriangle1} and \eqref{tunis4triangle}
 that for any $\sigma_0\in[ -\log \hat T, -\log(\hat T-T(u_0))_+)$, we have
\begin{equation}\label{defpb}
v_{T,\,a,\,u_0}(y,s_0)=\bar \psi(\sigma_0, u_0, T,a,y)
\equiv
(1+\tau)^{{1\over p-1}}
\left[\hat v(z,\sigma_0)+\varphi(z,\sigma_0)+e^{-\frac{\sigma_0}{p-1}} \epsilon\left(e^{-\frac{\sigma_0}2}z, \hat T - e^{-\sigma_0}\right)\right]-\varphi(y,s_0)\\
\end{equation}
with
\begin{equation}\label{defs0}
\tau = (T-\hat T)e^{\sigma_0},\;\alpha = a e^{\frac{\sigma_0}2},\;
s_0=s_0(\sigma_0, \tau)=\sigma_0-\log(1+\tau)\mbox{ and }z=y\sqrt{1+\tau}+\alpha
\end{equation}
(note that we used the fact that
 $$\hat{\vq}=\vq_{\hat{T},0,\hat{u}_0}= w_{\hat{T},0,\hat{u}_0}-\varphi= \hat{w}-\varphi,$$
which follows from \eqref{defvc} and \eqref{tunistriangle1}).



\medskip

In this context, $\bpsi(\sigma_0, u_0,T,a,y)$ appears as initial data
for equation \eqref{qequation} at initial time $s=s_0$.
Though the initial time $s_0=s_0(\sigma_0, \tau)$ is changing with $T$,
this reminds us of an analogous situation: in
the constructing procedure in Section 4, we had initial data
$\psi_{s_0,d_0,d_1}$ \eqref{initial-data} at $s=-\log \hat T,$ for the
same equation \eqref{qequation}, depending on two parameters
$(d_0,d_1).$

\medskip

What if by chance, {\it the application $(T,a)\mapsto
\bpsi(\sigma_0, u_0, T,\,a)$ satisfies the same initialization estimates as
$(d_0,d_1)\mapsto \psi_{s_0,d_0,d_1}$ (See Proposition
\ref{prop:diddc})?}
\medskip

In that case, the construction procedure would work, starting from
$\bpsi(\sigma_0, u_0, T,\,a,y)$ at time
$s=s_0,$
including the
reduction to a finite dimensional problem, and the topological
argument involving the two parameters $T$ and $a,$ resulting in the existence of
$(\bt(u_0),\ba(u_0)),$ such that equation \eqref{qequation} with
initial data at time
$s=s_0,$
$\bpsi(\sigma_0, u_0, \bt(u_0),\,\ba(u_0))$, has a solution $\bv_{\sigma_0,u_0}$ such
that
\begin{equation}
\label{rectangletunis} \forall \;
s\geq s_0,\;
\bv_{\sigma_0,u_0}(s)\in \vartheta_{A_0}(s).
\end{equation}
But then, remember that by definition,
$\bpsi_{\bt(u_0),\,\ba(u_0),\,u_0}$ is the initial data
also
 at
time $s=s_0$ defined in \eqref{defs0}, for $\vq_{\bt(u_0),\,\ba(u_0),\,u_0}(y,s),$
another solution of the same equation \eqref{qequation}. From
uniqueness in the Cauchy problem,
both solutions are equal, and have the same domain of
definition, and the same trapping property in
$\vartheta_{A_0}(s).$ In particular, recalling that
$\vq_{\bt(u_0),\,\ba(u_0),\,u_0}(y,s)$ is defined for all $(y,s) \in
\R\times [-\log \bt(u_0), -\log\left((\bt(u_0)-T(u_0))_+\right))$,
this implies that
$$\bt(u_0)=T(u_0),$$ and
$$\forall \; s\geq s_0,\;
\vq_{\bt(u_0),\,\ba(u_0),\,u_0}(y,s)=\bv_{\sigma_0,u_0}\left(y,s\right)$$
and from \eqref{rectangletunis}, we
have
\begin{equation}\label{trap}
\forall \; s\geq s_0,\; \vq_{T(u_0),\,a(u_0),\,u_0}(s)\in
\vartheta_{A_0}\left(s\right).
\end{equation}

\medskip

Using our technique in Section \ref{secsing}, we see that our original function
blows up in finite time $T(u_0)$ only at one blow-up point,
$\ba(u_0),$ and that $u(x-\ba(u_0))$ satisfies the profile estimates
given in Theorem \ref{th1}.

Of course, all this holds, provided that we check that the application $(T,a)\mapsto
\bpsi(\sigma_0, u_0,T,\,a)$ satisfies a statement analogous to Proposition
\ref{prop:diddc}, and that we show that
$$\bt(u_0)\to \hat{T},\; \ba(u_0)\to 0\; \mbox{ as }\;
\varepsilon_0=u_0-\hat{u_0}\to 0\mbox{ in }W^{1,\infty}(\R).
$$
Let us do that in the following subsections.


\subsection{Behavior of ``initial data'' $\bpsi$ for $(T,a)$ near $(\hat T, 0)$}

As explained in the previous subsection, here, we are left with the proof of an analogous statement to Proposition \ref{prop:diddc} for  initial data $\bar \psi(\sigma_0,u_0, T,a)$ \eqref{defpb} as a function of $(T,a)$. This is the new statement:

\bigskip

\noindent{\bf Proposition \ref{prop:diddc}' (Properties of initial data  $\bar \psi(\sigma_0,u_0,T,a)$ \eqref{defpb})}
{\it There exists $\bar C>0$ such that for any $A_0\ge \bar C\hat A$, there exists $\hat \sigma_0(A_0)>0$ large enough, such that for any $\sigma_0\ge \hat \sigma_0(A_0)$,
 there exists $\hat \epsilon_0(\sigma_0)>0$ small enough such that for all $u_0\in B_{W^{1,\infty}}(\hat u_0, \hat \epsilon_0(\sigma_0)):=\{ u\in W^{1,\infty}\; |\; \|u-\hat u_0\|_{W^{1,\infty}}\leq \hat \epsilon_0(\sigma_0)\}$,
the following holds:
\begin{itemize}
\item[(i)] There exists a set
\begin{equation}\label{bD}
\mathcal{\bar D}_{A_0,\sigma_0,u_0}\subset \{(T,a)\;|\;|T-\hat T|\le \frac{2e^{-\sigma_0}A_0(p-1)}{\kappa \sigma_0^{2\beta+1}},\;|a|\le  \frac{e^{-\frac{\sigma_0}2} A_0(p-1)^2}{b\kappa \sigma_0}\},
 \end{equation}
whose boundary is a Jordan curve such that the mapping
\[
(T,a) \mapsto s_0^{2\beta+1}(\bar \psi_0, \bar \psi_1)(\sigma_0, u_0,T,a)
\mbox{ where }s_0=s_0(\sigma_0, \tau) = \sigma_0 -\log(1+\tau)
\]
is one to one from $\mathcal{\bar D}_{A_0,\sigma_0,u_0}$ onto $[- A_0, A_0]^2$.
Moreover, it is of degree $-1$ on the boundary.
\item[(ii)]  For all $(T,a) \in \mathcal{\bar D}_{A_0, \sigma_0, u_0},$
$\bar \psi(\sigma_0, u_0,T,a) \in \vartheta_{A_0}(s_0)$ with strict
inequalities except for the two first, namely $(\bar \psi_0,\bar \psi_1)(\sigma_0, u_0,T,a)$, in the sense that
\begin{equation*}
|\bar \psi_m|\le \frac{A_0}{s_0^{2\beta+1}},\;m=0,1,\;\;
|\bar \psi_2|<\frac {C\sqrt{\hat A}}{s_0^{4\beta-1}},\;\;
|\bar \psi_-(y)|<~
\frac {3\hat A}{s_0^{\gamma}} (1+|y^3|),\;\forall\; y\in\; \R,
  \;\|\bar \psi_e\|_{L^\infty}< \frac {C\hat A^2}{s_0^{\gamma-3\beta}}.
\end{equation*}
\item[(iii)] Moreover, for all $(T,a)\in \mathcal{\bar D}_{u_0},$ we
have
\begin{equation*}
 \|\nabla\bar\psi\|_{L^\infty(\R)} \leq \frac{C{\hat A}^2}{s_0^{\gamma-3\beta}}
\mbox{ and }
|\nabla \bar \psi_-(y)|\leq
\frac{C{\hat A}^2}{s_0^{\gamma}}(1+|y|^3),
 \; \forall \; y\in \R.
 \end{equation*}
\end{itemize}
}
\begin{rem}\label{rems0}
{\rm Since $\bar\psi(\sigma_0, u_0, T,a)$ is the considered initial data for equation \eqref{qequation} at time $s=s_0$, we naturally decompose it according to \eqref{def:q:proj}-\eqref{qprojection} with $s=s_0$ defined in \eqref{defs0}.}
\end{rem}

\bigskip

In fact, this statement follows directly from the following expansion of $\bar \psi(\sigma_0, u_0,T,a)$ \eqref{defpb} for $(T,a)$ close to $(\hat T, 0)$:
\begin{lem}[Expansion of modes for $(T,a)$ close to $(\hat{T},\hat{a})$]\label{lemexp}
 For  $\sigma_0>0$ large enough, there exists
 $C_0(\sigma_0)>0$ such that for any
\begin{equation}\label{condta}
\|\varepsilon_0\|_{W_{1,\infty(\R)}}\le \frac 1{C_0},\;|\tau|\le \frac 12,\;|\alpha|\le 1,
\end{equation}
we have the following expansions:
\begin{itemize}
\item[(i)]
$|(\frac {s_0}{\sigma_0})^{2\beta+1}\bar \psi_0(\sigma_0, u_0,T,a)-\hat v_0(\sigma_0)-\frac{\kappa\tau}{p-1}|\le C(\frac 1{{\sigma_0}^{4\beta}}
+\frac{|\tau|{\hat A}^2}{\sigma_0^{\gamma-3\beta}}
+\frac{|\alpha|\hat A}{\sigma_0^{2\beta+1}}
+\frac{\alpha^2}{\sigma_0^{2\beta}}+\tau^2+|\tau\alpha|+|\alpha|^3)
+ C_0\|\varepsilon_0\|_{W^{1,\infty}(\R)}$,
\item[(ii)]
$|(\frac {s_0}{\sigma_0})^{2\beta+1}\bar \psi_1(\sigma_0, u_0,T,a)-\hat v_1(\sigma_0)
+{2b\kappa \over (p-1)^2}{\alpha\over {\sigma_0}^{2\beta}}|\le
C(e^{-\sigma_0^{2\beta}}+\frac{|\tau|{\hat A}^2}{\sigma_0^{\gamma-3\beta}}+
\frac{|\alpha|\sqrt{\hat A}}{\sigma_0^{4\beta-1}}+\tau^2+|\tau\alpha|+|\alpha|^3
)+
 C_0\|\varepsilon_0\|_{W^{1,\infty}(\R)}$,
    \item[(iii)]
$|\bar \psi_2(\sigma_0, u_0,T,a)|\le
\frac{C\sqrt{\hat A}}{{\sigma_0}^{4\beta-1}}+C(|\tau|+\frac{|\alpha|\hat A}{\sigma_0^\gamma}+|\alpha|^3)+C_0\|\varepsilon_0\|_{W^{1,\infty}(\R)}$,
\item[(iv)]
$|s_0^{-2\beta-1}\partial_\tau[s_0^{2\beta+1}\bar \psi_0](\sigma_0, u_0,T,a)-
{\kappa\over p-1}|\le C({1\over \sigma_0}+|\tau|+|\alpha|)+C_0\|\varepsilon_0\|_{W^{1,\infty}(\R)}$,
\item[(v)]
$|\partial_\tau[(\frac{s_0}{\sigma_0})^{2\beta+1}\bar \psi_1](\sigma_0, u_0,T,a)|\le
C(\frac {\hat A}{\sigma_0}+|\tau|+|\alpha|)+C_0\|\varepsilon_0\|_{W^{1,\infty}(\R)}$,
\item[(vi)]
$|\partial_\alpha\bar \psi_0(\sigma_0, u_0,T,a)|\le C(\frac{\hat A}{\sigma_0^{2\beta+1}}+|\tau|+\frac{|\alpha|}{\sigma_0^{2\beta}}+\alpha^2)
+C_0\|\varepsilon_0\|_{W^{1,\infty}(\R)}$,
\item[(vii)]
$|\partial_\alpha\bar \psi_1(\sigma_0, u_0,T,a)+
{2b\kappa \over (p-1)^2}{1\over {\sigma_0}^{2\beta}}|\le
C(\frac{\sqrt{\hat A}}{\sigma_0^{4\beta-1}}+|\tau|+\frac{\hat A|\alpha|}{\sigma_0^\gamma}+\alpha^2)
+C_0\|\varepsilon_0\|_{W^{1,\infty}(\R)}$,
    \item[(viii)]
$\left\|\frac{\bar \psi_-(\sigma_0, u_0,T,a,y)}{1+|y|^3}\right\|_{L^\infty(\R)}\le
C(\frac {\hat A}{\sigma_0^\gamma}+|\tau|+\frac{|\alpha|}{\sigma_0^{2\beta}}+\frac{{\hat A}^2|\alpha|^3}{\sigma_0^{\gamma-3\beta}})
+C_0\|\varepsilon_0\|_{W^{1,\infty}(\R)}$,
    \item[(ix)]
$\left\|\bar \psi_e(\sigma_0, u_0,T,a)\right\|_{L^\infty(\R)}\le
C(\frac{{\hat A}^2}{\sigma_0^{\gamma-3\beta}}+|\tau|+\frac{|\alpha|}{\sigma_0^\beta})
+C_0\|\varepsilon_0\|_{W^{1,\infty}(\R)},$
\item[(x)] $\|\nabla \bar \psi(\sigma_0, u_0,T,a)\|_{L^\infty(\R)}\le
C(\frac{{\hat A}^2}{\sigma_0^{\gamma-3\beta}}+\frac{|\tau|}{\sigma_0^\beta}+\frac{|\alpha|}{\sigma_0^{2\beta}})
+C_0\|\varepsilon_0\|_{W^{1,\infty}(\R)}$,

\item[(xi)] $\left\|\frac{\nabla \bar \psi_-(\sigma_0, u_0,T,a,y)}{1+|y|^3}\right\|_{L^\infty(\R)}\le
C(\frac{{\hat A}^2}{\sigma_0^\gamma}+\frac{{\hat A}^2(|\tau|+|\alpha|^3)}{\sigma_0^{\gamma-3\beta}}+\frac{|\alpha|}{\sigma_0^{2\beta}})
 +C_0\|\varepsilon_0\|_{W^{1,\infty}(\R)}$.
\end{itemize}
\end{lem}
\begin{rem}
As already stated in Remark \ref{rems0}, we decompose $\bar \psi$ according to \eqref{def:q:proj}-\eqref{qprojection} with $s=s_0$ defined in \eqref{defs0}.
\end{rem}
Indeed, let us first use this lemma to derive Proposition \ref{prop:diddc}', then, we will give its proof.

\medskip

\begin{proof}[Proof of Proposition \ref{prop:diddc}' assuming that Lemma \ref{lemexp} holds]
$ $\\
(i)
Introducing the following change of functions and variables:
\begin{equation}\label{trans}
\tilde \psi_m(\sigma_0, \varepsilon_0,\tilde \tau, \tilde \alpha) = s_0^{2\beta+1}\bar \psi_m(\sigma_0, u_0,T,a)\mbox{ for }m=0,1,\;\;
\varepsilon_0=u_0-\hat u_0,\;\;
\tilde \tau = \sigma_0^{2\beta+1}\tau\mbox{ and }
\tilde \alpha = \sigma_0 \alpha,
\end{equation}
we readily see from the previous lemma that whenever
\[
\|\varepsilon_0\|_{W^{1,\infty}(\R)}\le \frac 1{C_0(\sigma_0)},\;\;
|\tilde \tau|\le \frac{\sigma_0^{2\beta+1}}2,\;\;|\tilde \alpha|\le \sigma_0,
\]
and $\sigma_0$ is large enough, it follows that
\begin{align}
|\tilde \psi_0(\sigma_0, \varepsilon_0,\tilde \tau, \tilde \alpha)-\sigma_0^{2\beta+1}\hat v_0(\sigma_0)-\frac{\kappa\tilde\tau}{p-1}|\le& C(\frac 1{{\sigma_0}^{2\beta-1}}
+\frac{|\tilde \tau|{\hat A}^2}{\sigma_0^{\gamma-3\beta}}
+\frac{|\tilde \alpha|}{\sigma_0}
+\frac{{\tilde\alpha}^2}{\sigma_0}+\frac{{\tilde\tau}^2}{\sigma_0^{2\beta+1}}+\frac{|\tilde \tau\tilde\alpha|}{\sigma_0})\nonumber\\
&+\frac{C|\tilde\alpha|^3}{\sigma_0^{2(1-\beta)}}
+ C_0\sigma_0^{2\beta+1}\|\varepsilon_0\|_{W^{1,\infty}(\R)},\label{psi0}\\
|\tilde \psi_1(\sigma_0, \varepsilon_0,\tilde \tau, \tilde \alpha) -\sigma_0^{2\beta+1}\hat v_1(\sigma_0)
+{2b\kappa \over (p-1)^2}\tilde\alpha|\le &
C(\sigma_0^{2\beta+1}e^{-\sigma_0^{2\beta}}+\frac{|\tilde\tau|{\hat A}^2}{\sigma_0^{\gamma-3\beta}}+
\frac{|\tilde\alpha|\sqrt{\hat A}}{\sigma_0^{2\beta-1}}+\frac{{\tilde\tau}^2}{\sigma_0^{2\beta+1}}+\frac{|\tilde\tau\tilde\alpha|}{\sigma_0})\nonumber\\
&+C\frac{|\tilde\alpha|^3}{\sigma_0^{2(1-\beta)}}
+
 C_0\sigma_0^{2\beta+1}\|\varepsilon_0\|_{W^{1,\infty}(\R)},\label{psi1}
\end{align}
and
\begin{align*}
{\rm Jac}_{\tilde \tau, \tilde \alpha}(\sigma_0, \varepsilon_0,\tilde \psi_0, \tilde \psi_1)(\tilde \tau, \tilde \alpha)=&
\left|
\begin{array}{cc}
\frac \kappa{p-1}&0,\\
0&-\frac{2b\kappa}{(p-1)^2}
\end{array}
\right|
+C(\frac{\sqrt{\hat A}}{\sigma_0^{2\beta-1}}+\frac{|\tilde \tau|}{\sigma_0}+\frac{|\tilde \alpha|}{\sigma_0}+\frac{{\tilde \alpha}^2}{\sigma_0^{2(1-\beta)}})\\
&+C_0\sigma_0^{2\beta}\|\varepsilon_0\|_{W^{1,\infty}(\R)}.
\end{align*}
Consider now $A_0\ge 2 \hat A$, $\sigma_0$ large and $\varepsilon_0$ such that
\begin{equation}\label{conde0}
\|\varepsilon_0\|_{W^{1,\infty}(\R)}\le \hat \varepsilon_0(\sigma_0)\equiv \frac 1{\sigma_0^{2\beta +1}C_0(\sigma_0)}.
\end{equation}
From the above-mentioned expansions, we see that for $\sigma_0$ large enough, the function
\[
(\tilde\tau, \tilde \alpha) \mapsto (\tilde \psi_0, \tilde \psi_1)(\sigma_0,, \varepsilon_0, \tilde \tau, \tilde \alpha)
\]
is a $C^1$ diffeomorphism from the rectangle
\[
{\mathcal R}_{A_0}\equiv
 \left[-\frac{2(p-1)A_0}\kappa, \frac{2(p-1)A_0}\kappa\right]\times \left[-\frac{(p-1)^2A_0}{b\kappa}, \frac{(p-1)^2A_0}{b\kappa}\right]
\]
onto a set ${\mathcal E}_{A_0,\sigma_0,u_0}$ which approaches (in some appropriate sense) from the rectangle
$[\sigma_0^{2\beta+1}\hat v_0(\sigma_0)-2A_0,  \sigma_0^{2\beta+1}\hat v_0(\sigma_0)+2A_0]\times [\sigma_0^{2\beta+1}\hat v_0(\sigma_1)-2A_0,  \sigma_0^{2\beta+1}\hat v_1(\sigma_0)+2A_0]$ as $\sigma_0\to \infty$. Since $\hat v(\sigma_0)\in \vartheta_{\hat A}(\sigma_0)$
by \eqref{vva},
 hence $|\sigma_0^{2\beta+1}\hat v_m(\sigma_0)|\le \hat A$ for $m=0,1$, we clearly see that
\[
[-A_0,A_0]^2 \subset {\mathcal E}_{A_0,\sigma_0,u_0}
\]
for $\sigma_0$ large enough, hence, there exists a set $\tilde{\mathcal D}_{A_0,\sigma_0,u_0}\subset {\mathcal R}_{A_0}$ such that
\[
(\tilde \psi_0, \tilde \psi_1)(\sigma_0, \varepsilon_0,\tilde {\mathcal D}_{A_0,\sigma_0,u_0})=[-A_0, A_0]^2.
\]
Moreover, from \eqref{psi0}-\eqref{psi1}, the function $(\tilde \psi_0, \tilde \psi_1)$ has degree $-1$ on the boundary of $\tilde{\mathcal D}_{A_0,\sigma_0,u_0}$. Using back the transformation \eqref{trans} gives the conclusion of item (i).\\
(ii) Take $(T,a)\in \bar {\mathcal D}_{A_0,\sigma_0, u_0}$ and let us check that
$\bpsi(\sigma_0, u_0, T,a)\in \vartheta_{A_0}(s_0)$.\\
First, from item (i), we know by construction that
$|\bpsi_m(\sigma_0, u_0, T,a)|\le A_0s_0^{-2\beta-1}$
for $m=0,1$. For the other estimates to be checked, note first from \eqref{bD} and the definition \eqref{defs0} of $(\tau, \alpha)$ that we have
\[
|\tau| \le \frac{2A_0(p-1)}{\kappa \sigma_0^{2\beta+1}}
\mbox{ and }|\alpha|\le \frac{A_0(p-1)^2}{b\kappa \sigma_0}.
\]
Therefore, using \eqref{conde0} then items (iii), (viii) and (ix) of Lemma \ref{lemexp}, we see that for $\sigma_0$ large, we have
\begin{equation*}
|\bpsi_2(\sigma_0, u_0, T,a)|\le \frac{C \sqrt{\hat A}}{\sigma_0^{4\beta-1}},\;\;
\left\|\frac{\bar \psi_-(\sigma_0, u_0,T,a,y)}{1+|y|^3}\right\|_{L^\infty(\R)}\le \frac {C \hat A}{\sigma_0^\gamma},\;\;
\left\|\bar \psi_e(\sigma_0, u_0,T,a)\right\|_{L^\infty(\R)}\le
\frac{C{\hat A}^2}{\sigma_0^{\gamma-3\beta}}.
\end{equation*}
If $A_0\ge \bar C \hat A$, for some large $\bar C>0$, then we see that the conclusion follows.\\
(iii) The conclusion follows from  items (x) and (xi) in Lemma \ref{lemexp}, proceeding similarly to what we did for item (ii).\\
This concludes the proof of  Proposition \ref{prop:diddc}', assuming that Lemma \ref{lemexp} holds.
\end{proof}

Now, we are left with the proof of Lemma \ref{lemexp}.

\begin{proof}[Proof of Lemma \ref{lemexp}] The proof is similar in the spirit to the case of the unperturbed semilinear heat equation (with $\mu=0$ in \eqref{eq:uequation}) treated in \cite{MZsbupdmj97} (see Lemma B.2 page 186 in that paper). However, due to the difference of the scaling in the profile $\varphi$ \eqref{feq} (we had $\beta=\frac 12$ in \cite{MZsbupdmj97} whereas we have $\beta=\frac{p+1}{2(p-1)}$ here), we need to carefully give details for all the computations in the proof.

\medskip

From \eqref{defpb}, we write
\[
\bar \psi(\sigma_0, u_0, T,a,y)
=
\bar \psi^1(y,s_0)
+\bar \psi^2(y,s_0)
+\bar \psi^3(y,s_0)
+\bar \psi^4(y,s_0)
\]
where
\begin{align}
\bar \psi^1(y,s_0) &=(1+\tau)^{{1\over p-1}}
e^{-\frac{\sigma_0}{p-1}} \epsilon\left(e^{-\frac{\sigma_0}2}z, \hat T - e^{-\sigma_0}\right),\;\;
&\bar \psi^2(y,s_0)& =(1+\tau)^{{1\over p-1}}\hat v(z,\sigma_0)\nonumber\\
\bar \psi^3(y,s_0)&=(1+\tau)^{{1\over p-1}}\varphi(z,\sigma_0),\;\;
&\bar \psi^4(y,s_0)& =
-\varphi(y,s_0),\label{defbpq}
\end{align}
where $\tau$, $\alpha$ and $s_0$ are given in \eqref{defs0}.
In the following, we proceed in 4 steps, to prove analogous statements to Lemma \ref{lemexp} with $\bar \psi$ replaced by $\bar \psi^i$ for $i=1,\dots,4$. Of course, Lemma \ref{lemexp} then follows by addition. Note that as for $\bar \psi$, we decompose $\bar \psi_i$ according to \eqref{def:q:proj}-\eqref{qprojection} with $s=s_0$ \eqref{defs0}. This also justifies the lighter notation $\bar\psi_i(y,s_0)$, where we insist on the dependence on the space variable $y$ and the time variable $s_0$ corresponding to the initial time where $\bar\psi$ is considered (see Remark \ref{rems0}).

\bigskip

{\bf Step 1: Expansions of $\bar\psi^1$}

Let us first note that since equation \eqref{eq:uequation} is wellposed in $W^{1,\infty}(\R)$ through a simple fixed-point argument (see page \pageref{eq:uequation}), it follows that
\[
\|\varepsilon(\hat T-e^{-\sigma_0})\|_{W_{1,\infty}(\R)}\le C_0\|\varepsilon_0\|_{W_{1,\infty}(\R)},
\]
whenever $\|\varepsilon_0\|_{W_{1,\infty}(\R)}\le \frac 1{C_0}$, for some
$C_0=C_0(\sigma_0)$.\\
Therefore, using our techniques throughout this paper, we clearly see that for $\sigma_0$ large enough and $(\tau, \alpha)$ satisfying \eqref{condta}, we have
\begin{align*}
&|\bar \psi^1_0(s_0)|+
|\bar \psi^1_1(s_0)|+
|\bar \psi^1_2(s_0)|+
|\partial_\tau[(\frac{s_0}{\sigma_0})^{2\beta+1}\bar \psi^1_0](s_0)|+
|\partial_\tau[(\frac{s_0}{\sigma_0})^{2\beta+1}\bar \psi^1_1](s_0)|
+
|\partial_\alpha\bar \psi^1_0(s_0)|+
|\partial_\alpha\bar \psi^1_1(s_0)|\\
+&
\left\|\frac{\bar \psi^1_-(y,s_0)}{1+|y|^3}\right\|_{L^\infty(\R)}+
\left\|\bar \psi^1_e(s_0)\right\|_{L^\infty(\R)}
+\|\nabla \bar \psi^1(s_0)\|_{L^\infty(\R)}
+\left\|\frac{\nabla \bar \psi^1_-(y,s_0)}{1+|y|^3}\right\|_{L^\infty(\R)} \\
\le& C\|\varepsilon(\hat T-e^{-\sigma_0})\|_{W_{1,\infty}(\R)}\le C_0\|\varepsilon_0\|_{W_{1,\infty}(\R)}.
\end{align*}

\bigskip

{\bf Step 2: Expansions of $\bar\psi^2$}

Anticipating Step 3, where we handle $\bar\psi^3$, we note that both functions $\bar\psi^2$ and $\bar\psi^3$ are
of the form $(1+\tau)^{\frac 1{p-1}}g(z,\sigma_0)$, where $g(\sigma_0)\in W^{1,\infty}(\R)$. For that reason, we can handle both functions at the same time, by first expanding in the style of Lemma \ref{lemexp}, a function $\bar g$ defined by
\begin{equation}\label{defhb}
\bar g(y,s_0)=
(1+\tau)^{\frac 1{p-1}}g(z,\sigma_0)\mbox{ where }z=y\sqrt{1+\tau} +\alpha
\mbox{ and }s_0=\sigma_0-\log(1+\tau),
\end{equation}
where $g(\sigma_0)\in W^{1,\infty}(\R)$.
The following statement allows us to conclude for $\bar \psi^2$:
\begin{lem}[Expansions of $\bar g(y,\sigma_0)$]\label{clproj} If $g(\sigma_0)\in W^{1,\infty}(\R)$ and $\bar g$ is defined by \eqref{defhb}, then, the following expansions hold for $\sigma_0$ large enough and $(\tau, \alpha)$ satisfying \eqref{condta}:
\begin{align*}
&\mbox{\rm (i) }&(\frac{s_0}{\sigma_0})^{2\beta+1}\bar g_0(s_0)=& g_0(\sigma_0)\left(1+\frac \tau{p-1}\right)+O(\alpha g_1(\sigma_0))+O((|\tau|+\alpha^2) g_2(\sigma_0))\\
&&&+O((e^{-\sigma_0^{2\beta}}+\tau^2+\frac{|\tau|}{\sigma_0}+|\tau \alpha|+|\alpha|^3)\|g(\sigma_0)\|_{L^\infty}),\\
&\mbox{\rm (ii) }&(\frac{s_0}{\sigma_0})^{2\beta+1}\bar g_1(s_0)=&g_1(\sigma_0)\left(1+\frac {(p+1)\tau}{2(p-1)}\right)+
2\alpha g_2(\sigma_0)+O(g_3(\sigma_0)(|\tau|+\alpha^2))\\
&&&
    +O((e^{-\sigma_0^{2\beta}}+\tau^2+\frac{|\tau|}{\sigma_0}+|\tau \alpha|+|\alpha|^3)\|g(\sigma_0)\|_{L^\infty}),\\
&\mbox{\rm (iii) }&\bar g_2(s_0)=g_2(\sigma_0)&+O(g_3(\sigma_0)\alpha)+O(g_4(\sigma_0)\alpha^2)
+O((e^{-\sigma_0^{2\beta}}+|\tau |+|\alpha|^3)\|g(\sigma_0)\|_{L^\infty}),\\
&\mbox{\rm (iv) }&s_0^{-2\beta-1}\partial_\tau[s_0^{2\beta+1} \bar g_0](s_0)=&\frac {g_0(\sigma_0)}{p-1}+O(g_2(\sigma_0))+ O((\sigma_0^{-1}+|\tau |+|\alpha|)\|g(\sigma_0)\|_{L^\infty}),\\
&\mbox{\rm (v) }&\partial_\tau[(\frac{s_0}{\sigma_0})^{2\beta+1} \bar g_1](s_0)=&O(g_1(\sigma_0))+O(g_3(\sigma_0))+O((\sigma_0^{-1}+|\tau |+|\alpha|)\|g(\sigma_0)\|_{L^\infty}),\\
&\mbox{\rm (vi) }&\partial_\alpha \bar g_0(s_0)=&O(g_1(\sigma_0))+O(\alpha g_2(\sigma_0))+O((e^{-\sigma_0^{2\beta}}+|\tau |+\alpha^2)\|g(\sigma_0)\|_{L^\infty}),\\
&\mbox{\rm (vii) }&\partial_\alpha \bar g_1(s_0)=&2g_2(\sigma_0)+O(\alpha g_3(\sigma_0))+O((e^{-\sigma_0^{2\beta}}+|\tau |+\alpha^2)\|g(\sigma_0)\|_{L^\infty}),\\
&\mbox{\rm (viii) }&\left\|\frac{\bar g_-(y,s_0)}{1+|y|^3}\right\|_{L^\infty(\R)}\le&C \left\|\frac{g_-(z,\sigma_0)}{1+|z|^3}\right\|_{L^\infty(\R)}
+C\left(
\sigma_0^{-4\beta}+|\tau|+|\alpha|^3\right)\|g(\sigma_0)\|_{L^\infty(\R)}\\
&&&+C|\alpha|\left(|g_1(\sigma_0)|+|g_2(\sigma_0)|\right),\\
&\mbox{\rm (ix) }&\left\|\bar g_e(y,s_0)\right\|_{L^\infty(\R)}\le&\left\|g_e(z,\sigma_0)\right\|_{L^\infty(\R)}+C(|\tau|
+|\alpha|)\|g(\sigma_0)\|_{L^\infty(\R)},\\
&\mbox{\rm (x) }&\|\nabla \bar g(s_0)\|_{L^\infty(\R)}\le &C \|\nabla g(\sigma_0)\|_{L^\infty(\R)},\\
&\mbox{\rm (xi) }&\left\|\frac{\nabla \bar g_-(y,s_0)}{1+|y|^3}\right\|_{L^\infty(\R)}\le &C \left\|\frac{\nabla g_-(z,\sigma_0)}{1+|z|^3}\right\|_{L^\infty(\R)}+C\|\nabla g(\sigma_0)\|_{L^\infty(\R)}\left(|\tau|+\frac{|\alpha|}{\sigma_0^{4\beta}}\right)+O(\alpha g_2(\sigma_0))\\
&&&+O(\alpha g_3(\sigma_0))+O(\alpha^2 g_4(\sigma_0))+C\|g(\sigma_0)\|_{L^\infty(\R)}\left(
\sigma_0^{-5\beta}+|\tau|+|\alpha|^3\right).
\end{align*}
\end{lem}
\begin{rem} {\rm We would like to insist on the fact that the notation $O(g)$ stands here for a function bounded by $Cg$, where $C$ is a universal constant, depending only on $p$ and $\mu$, the function $\chi$ and the constant $K>0$ defining the truncation in \eqref{def:chi}.}
\end{rem}
\begin{proof} The proof is straightforward from the definition of the decomposition in \eqref{qprojection}. As it is only technical, we leave it to Appendix \ref{appproj}.
\end{proof}
Indeed, using \eqref{vva},
item (i) of Proposition \ref{prop:rt2dimen}
and parabolic regularity stated in Proposition \ref{prop:regu-parab-q-equation} (using the part of the proof with $s\ge s_0+1$ only), we see that for $\sigma_0$ large enough, we have
\begin{align*}
&\hat v(\sigma_0)\in \vartheta_{\hat A}(\sigma_0),\;
|\hat v_j(\sigma_0)|\le \frac {C\hat A}{\sigma_0^\gamma}\mbox{ for }j=3\mbox{ or }4,\\
&\|\hat v(\sigma_0)\|_{L^\infty(\R)}+
\|\nabla \hat v(\sigma_0)\|_{L^\infty(\R)}\le \frac{C {\hat A}^2}{\sigma_0^{\gamma-3\beta}}
\mbox{ and }
 \left\|\frac{\nabla \hat v_-(z,\sigma_0)}{1+|z|^3}\right\|_{L^\infty(\R)}\le \frac{C {\hat A}^2}{\sigma_0^\gamma}.
\end{align*}
Using the definition of $\vartheta_{\hat A}$ given in Definition \eqref{prop:V-def}, then applying Lemma \ref{clproj} with $g=\hat v$, we see from \eqref{gamma} that for $\sigma_0$ large enough, we have
\begin{align*}
|(\frac{s_0}{\sigma_0})^{2\beta+1}\bpd_0(s_0)-\hat v_0(\sigma_0)|&\le \frac {C {\hat A}^2e^{-\sigma_0^{2\beta}}}{\sigma_0^{\gamma-3\beta}}
+\frac{ C{\hat A}^2}{\sigma_0^{\gamma-3\beta}}\left(|\tau|+|\alpha|^3\right)
+\frac{C\hat A |\alpha|}{\sigma_0^{2\beta+1}}
+\frac {C \sqrt{\hat A}}{\sigma_0^{4\beta-1}}\alpha^2,\\
|(\frac{s_0}{\sigma_0})^{2\beta+1}\bpd_1(s_0)-\hat v_1(\sigma_0)|&\le \frac {C {\hat A}^2e^{-\sigma_0^{2\beta}}}{\sigma_0^{\gamma-3\beta}}
+\frac{ C{\hat A}^2}{\sigma_0^{\gamma-3\beta}}\left(|\tau|+|\alpha|^3\right)
+\frac {C \sqrt{\hat A}}{\sigma_0^{4\beta-1}}|\alpha|,\\
|\bpd_2(s_0)|&\le \frac {C\sqrt{\hat A}}{\sigma_0^{4\beta-1}}
+\frac{ C{\hat A}^2}{\sigma_0^{\gamma-3\beta}}\left(|\tau|+|\alpha|^3\right)
+\frac {C \hat A}{\sigma_0^\gamma}|\alpha|,\\
|s_0^{-2\beta-1}\partial_\tau[s_0^{2\beta+1} \bpd_0](s_0)|&\le  \frac {C \hat A^2}{\sigma_0^{\gamma-3\beta+1}}
+\frac{ C{\hat A}^2}{\sigma_0^{\gamma-3\beta}}\left(|\tau|+|\alpha|\right),\\
|s_0^{-2\beta-1}\partial_\tau[s_0^{2\beta+1} \bpd_1](s_0)|&\le  \frac {C \hat A^2}{\sigma_0^{\gamma-3\beta+1}}
+\frac{ C{\hat A}^2}{\sigma_0^{\gamma-3\beta}}\left(|\tau|+|\alpha|\right),\\
|\partial_\alpha \bpd_0(s_0)|&\le  \frac {C\hat A}{\sigma_0^{2\beta+1}}
+\frac {C \sqrt{\hat A}}{\sigma_0^{4\beta-1}}|\alpha|
+\frac{ C{\hat A}^2}{\sigma_0^{\gamma-3\beta}}\left(|\tau|+\alpha^2\right),\\
|\partial_\alpha \bpd_1(s_0)|&\le  \frac {C\sqrt{\hat A}}{\sigma_0^{4\beta-1}}
+\frac {C\hat A}{\sigma_0^\gamma}|\alpha|
+\frac{ C{\hat A}^2}{\sigma_0^{\gamma-3\beta}}\left(|\tau|+\alpha^2\right),\\
\left\|\frac{\bpd_-(y,s_0)}{1+|y|^3}\right\|_{L^\infty(\R)} &\le \frac {C \hat A}{\sigma_0^\gamma}
+\frac {C\sqrt{\hat A}}{\sigma_0^{4\beta-1}}|\alpha|
+\frac {C {\hat A}^2}{\sigma_0^{\gamma-3\beta}}\left(|\tau|+|\alpha|^3\right),\;\;
\left\|\bpd_e(s_0)\right\|_{L^\infty(\R)}\le \frac {C {\hat A}^2}{\sigma_0^{\gamma-3\beta}},\\
\left\|\nabla \bpd(s_0)\right\|_{L^\infty(\R)}&\le \frac {C {\hat A}^2}{\sigma_0^{\gamma-3\beta}},\;\;
\left\|\frac{\nabla \bpd_-(y,s_0)}{1+|y|^3}\right\|_{L^\infty(\R)} \le \frac {C {\hat A}^2}{\sigma_0^\gamma}
+\frac {C\sqrt{\hat A}}{\sigma_0^{4\beta-1}}|\alpha|
+\frac {C {\hat A}^2}{\sigma_0^{\gamma-3\beta}}\left(|\tau|+|\alpha|^3\right).
\end{align*}

\bigskip

{\bf Step 3: Expansions of $\bar\psi^3$}

The first part of the estimates follows from Lemma \ref{clproj}, as for $\bpd$. The second part needs more refinements, which are eased by the fact that $\bpt$ is explicit.

\medskip

In order to apply Lemma \ref{clproj} with $g=\varphi$, we first introduce the following estimates on $\varphi$ which follow from straightforward computations:\\
For $s$ large enough, we have
\begin{align*}
&\varphi_0(s)=\kappa+ O(s^{-4\beta}),\;\;
\varphi_2(s) = -\frac a2s^{-2\beta}+O(s^{-4\beta}),\;\;
\varphi_1(s)=\varphi_3(s)=0,\;\;
\varphi_4(s) = O(s^{-4\beta}),\\
&\|\varphi(s)\|_{L^\infty(\R)}\le C,\;\;
\end{align*}
where $a$ is introduced in \eqref{aeq}. Applying Lemma \ref{clproj} with $g=\varphi$ and using the value of $a$ in \eqref{aeq}, we see that
\begin{align*}
&\left|(\frac{s_0}{\sigma_0})^{2\beta+1}\bpt_0(s_0)-\kappa-\frac{\kappa\tau}{p-1}\right|\le C(\frac 1{\sigma_0^{4\beta}}+\frac{\alpha^2}{\sigma_0^{2\beta}}+\tau^2+\frac{|\tau|}{\sigma_0}+|\tau \alpha|+|\alpha|^3),\\
&\left|(\frac{s_0}{\sigma_0})^{2\beta+1}\bpt_1(s_0)+\frac{2b\kappa\alpha}{(p-1)^2\sigma_0^{2\beta}}\right|\le C(\frac{|\alpha|}{\sigma_0^{4\beta}}+e^{-\sigma_0^{2\beta}}+\tau^2+\frac{|\tau|}{\sigma_0}+|\tau \alpha|+|\alpha|^3),\\
&\bpt_2(s_0)=-\frac a{2\sigma_0^{2\beta}}
+O\left(\frac 1{\sigma_0^{4\beta}}\right)
+O(|\tau|+|\alpha|^3),\\
&{s_0}^{-2\beta-1}\partial_\tau [{s_0}^{2\beta+1}\bpt_0](s_0) = \frac \kappa{p-1}+O\left(\frac 1{\sigma_0^{2\beta}}\right)+O(|\tau|+|\alpha|),\\
&|{\sigma_0}^{-2\beta-1}\partial_\tau [{s_0}^{2\beta+1}\bpt_1](s_0)|\le C(\sigma_0^{-1}+|\tau|+|\alpha|),\\
&|\partial_\alpha \bpt_0(s_0)|\le C\left(\frac {|\alpha|}{\sigma_0^{2\beta}}+e^{-\sigma_0^{2\beta}}+|\tau|+\alpha^2\right),\;\;
\partial_\alpha \bpt_1(s_0)=-\frac{2b\kappa}{(p-1)^2\sigma_0^{2\beta}}+O\left(\frac 1{\sigma_0^{4\beta}}\right)+O(|\tau|+\alpha^2).
\end{align*}
Now, for the remaining components, we need more refined estimates, based on the explicit formula of $\varphi$ defined in \eqref{feq}, in particular the following, for $\sigma_0$ large enough and for all $y\in \R$,
\begin{align}
&|\bpt(y,s_0) - \varphi(y,\sigma_0)|\le C\left(|\tau|+\frac{|\alpha|}{\sigma_0^{2\beta}}\right)(1+|y|^3),\;\;
|\bpt(y,s_0) - \varphi(y,\sigma_0)|\le C\left(|\tau|+\frac{|\alpha|}{\sigma_0^\beta}\right),\nonumber\\
&|\nabla\bpt(y,s_0) - \nabla\varphi(y,\sigma_0)|\le C\left(\frac{|\tau|}{\sigma_0^\beta}+\frac{|\alpha|}{\sigma_0^{2\beta}}\right),\label{cal3}
\end{align}
(for the second and third estimates, use a first order Taylor expansion, and evaluate the error according to the position of $|y|$ with respect to $1$).
In fact, in Step 4, we will obtain analogous estimates for $\bpq(y,s_0)$. Therefore, as far as the remaining components are concerned, we wait for Step 4, where we will directly obtain the contribution of $\bpt(y,s_0)+\bpq(y,s_0)$ to those components.

\bigskip

{\bf Step 4: Expansions of $\bpq$}

Here, we need to refine the estimates we gave for $\varphi$ in Step 3. From further refinements, we write for $s$ large enough and for all $y\in \R$,
\[
\left|\varphi(y,s) -\left[\kappa-\frac a{2s^{2\beta}}h_2(y)\right]\right|
\le C\frac{|y|^4}{s^{4\beta}},\;\;
|\partial_s \varphi(y,s)|\le C\min\left(\frac 1s,\;\frac{1+y^2}{s^{2\beta+1}}\right),\;\;
|\partial_s \nabla \varphi(y,s)|\le \frac C{s^{\beta+1}}.
\]
Since $\bpq(y) =- \varphi(y,s_0)$ from \eqref{defbpq} with $s_0=\sigma_0-\log(1+\tau)$, it follows that
\begin{align}
&\left|\bpq(y,s_0) +\left[\kappa-\frac a{2\sigma_0^{2\beta}}h_2(y)\right]\right|
\le C\frac{|y|^4}{\sigma_0^{4\beta}}+C\frac{|\tau|(1+y^2)}{\sigma_0^{2\beta+1}},\label{cal4}\\
&\left|\bpq(y,s_0) +\varphi(y,\sigma_0)\right|\le  C|\tau|\min\left(\frac 1{\sigma_0},\;\frac{(1+y^2)}{\sigma_0^{2\beta+1}}\right),\;\;
&\left|\nabla \bpq(y,s_0) +\nabla \varphi(y,\sigma_0)\right|\le C|\tau|\sigma_0^{-1-\beta}. \nonumber
\end{align}
Therefore, by definition of the decomposition \eqref{qprojection}, 
we see that
\begin{align*}
&(\frac{s_0}{\sigma_0})^{2\beta+1}\bpq_0(s_0)=-\kappa+O(\sigma_0^{-4\beta})+O(\tau \sigma_0^{-1}),\;\;
\bpq_1(s_0)=0,\;\;
\bpq_2(s_0)=\frac a{2\sigma_0^{2\beta}}+O(\sigma_0^{-4\beta})+O(\tau \sigma_0^{-2\beta-1}),\\
&|\sigma_0^{-2\beta-1}\partial_\tau [s_0^{2\beta+1}\bpq_0](s_0)|\le \frac C{\sigma_0},\;\;
\partial_\tau [s_0^{2\beta+1}\bpq_1](s_0)=0,\;\;
\partial_\alpha \bpq_0(s_0)=\partial_\alpha \bpq_1(s_0)=0.
\end{align*}
As for the remaining components, we will directly estimate the contribution of $\bpt(y,s_0)+\bpq(y,s_0)$. Using \eqref{cal3} and \eqref{cal4}, we see that
\begin{align*}
|\bpt(y,s_0)+\bpq(y,s_0)|&\le C\min\left\{
\left(|\tau|+\frac{|\alpha|}{\sigma_0^{2\beta}}\right)(1+|y|^3),
\left(|\tau|+\frac{|\alpha|}{\sigma_0^\beta}\right)\right\},\\
|\nabla\bpt(y,s_0)+\nabla\bpq(y,s_0)|&\le C\left(\frac{|\tau|}{\sigma_0^\beta}+\frac{|\alpha|}{\sigma_0^{2\beta}}\right).
\end{align*}
Using these estimates, we see that for all $y\in \R$,
\begin{align*}
&|\bpt_-(y,s_0)+\bpq_-(y,s_0)|\le  C\left(|\tau|+\frac{|\alpha|}{\sigma_0^{2\beta}}\right)(1+|y|^3),\;\;
|\bpt_e(y,s_0)+\bpq_e(y,s_0)|\le  C\left(|\tau|+\frac{|\alpha|}{\sigma_0^\beta}\right),\\
&|\nabla\bpt_-(y,s_0)+\nabla\bpq_-(y,s_0)|\le C\left(\frac{|\tau|}{\sigma_0^\beta}+\frac{|\alpha|}{\sigma_0^{2\beta}}\right)(1+|y|^3)
\end{align*}
(use the fact that
\begin{align*}
\left\|\frac{v_-(y,s_0)}{1+|y|^3}\right\|_{L^\infty(\R)}&\le C\left\|\frac{v(y,s_0))}{1+|y|^3}\right\|_{L^\infty(\R)},\\
|\nabla v_-(y,s_0))|&\le C(\|\nabla v(s_0)\|_{L^\infty(\R)}+ \frac{\|v(s_0)\|_{L^\infty(\R)}}{\sigma_0^{4\beta}})(1+|y|^3),
\end{align*}
which follow from the definition of the decomposition \eqref{qprojection}).

\medskip

{\it Conclusion of the proof of Lemma \ref{lemexp}}: In fact, the conclusion follows by adding the various estimates obtained in Steps 1 through 4.
This concludes the proof of Lemma \ref{lemexp}.\\
 Since we have already showed that Proposition  \ref{prop:diddc}' follows from Lemma \ref{lemexp}, this also concludes the proof of Proposition \ref{prop:diddc}'.
\end{proof}

\subsection{Conclusion of the proof of Theorem \ref{th2} and generalization}

In this subsection, we briefly explain how to conclude the proof of Theorem \ref{th2} from the initialization given in  Proposition \ref{prop:diddc}' and our techniques developed for the existence proof. Then, we give a stronger version, valid for blowing-up solutions having the profile \eqref{145} only for a subsequence.

\bigskip

 \begin{proof}[Conclusion of the proof of Theorem \ref{th2}]
As we will see here, thanks to Proposition \ref{prop:diddc}', the stability proof reduces to an existence proof. Note that having an initial time $s_0$ \eqref{defs0} for equation \eqref{qequation} depending on the parameter $T$ changes nothing to the situation. Note also that thanks to Proposition \ref{prop:diddc}', we may write a statement Proposition \ref{prop:rt2dimen}' analogous to Proposition \ref{prop:rt2dimen}, valid for $A\ge \bar C' \hat A$,  for some large enough constant $\bar C'>0$. Let us now briefly explain the proof of Theorem \ref{th2}.

\medskip

Fix some $A_0=\max(\bar C,\bar C')$ where $\bar C>0$ is introduced in Proposition \ref{prop:diddc}', and $\bar C'>0$ introduced above.
Consider then an arbitrary $\eta>0$, and fix $\sigma_0$ large enough, so that
\[
\frac{2e^{-\sigma_0}A_0(p-1)}{\kappa \sigma_0^{2\beta+1}}+
\frac{e^{-\frac{\sigma_0}2} A_0(p-1)^2}{b\kappa \sigma_0}
\le \eta.
\]
Then, take initial data $u_0\in B_{W^{1,\infty}(\R)}(\hat u_0, \hat \varepsilon_0(\sigma_0))$, where $\hat \varepsilon_0(\sigma_0)$ is defined in Proposition \ref{prop:diddc}'.
As explained above in the strategy of the proof, we find some parameters $(\bar T(u_0), \bar a(u_0))$ satisfying
 \begin{equation}\label{neigh}
|\bar T(u_0)-\hat T|+|\bar a(u_0)|\le \eta,
\end{equation}
such that
\[
\forall \; s\geq \bar s_0\equiv \sigma_0 - \log(1+(\bar T(u_0)-\hat T) e^{\sigma_0}),\; \vq_{\bar T(u_0),\,\bar a(u_0),\,u_0}(s)\in
\vartheta_{A_0}\left(s \right).
\]
In particular, $u(x,t)$ blows up at time $T(u_0)=\bar T(u_0)$, only at one blow-up point $a(u_0)=\bar a(u_0)$, with the profile \eqref{145}. Since $\eta$ is arbitrary in \eqref{neigh}, this means that
\[
(T(u_0), a(u_0))\to (\hat T, \hat a)\mbox{ as }u_0\to \hat u_0.
\]
This concludes the proof of Theorem \ref{th2}.
\end{proof}


\bigskip

{\it Validity of the stability result}: As a concluding remark, we would like to mention that our stability proof works not only around the constructed solution in Theorem \ref{th1}, but also for blow-up solutions having the profile \eqref{145} only for a subsequence.
More precisely, this is our ``twin'' statement for the stability result stated in Theorem \ref{th2}:

\medskip

\label{th2'}
\noindent {\bf Theorem \ref{th2}'} (Stability of blow-up solutions of equation \eqref{eq:uequation} having the profile \eqref{145} only for a subsequence) . {\it Consider $\hat u(x,t)$ a solution to equation \eqref{eq:uequation} with initial data $\hat u_0$, which blows up at time $\hat T$ only at one blow-up point $\hat a$, such that
\begin{equation}\label{48}
v_{\hat T, \hat a, \hat u_0}(s_n)\in \vartheta_{\hat A}(s_n),\;\;
\|\nabla v_{\hat T, \hat a, \hat u_0}(s_n)\|_{L^\infty(\R)}\le \frac {C{\hat A}^2}{s_n^{\gamma-3\beta}}
\mbox{ and }
\left\|\frac{\nabla v_{\hat T, \hat a, \hat u_0,-}(y,s_n)}{1+|y|^3}\right\|_{L^\infty(\R)}\le \frac {C{\hat A}^2}{s_n^\gamma},
\end{equation}
for some $\hat A>0$ and for some sequence $s_n\to \infty$, where $\hat v_{\hat T, \hat v, \hat u_0}$ is defined in \eqref{tunistriangle3}.
Then, there exists a neighborhood $V_0$ of $\hat{u}_0$ in $W^{1,\infty}(\R^N)$ such that for any $\hat{u}_0\in V_0,$
Equation  \eqref{eq:uequation} has a unique solution $u$ with initial data $u_0,$ $u$ blows up in finite time $T(u_0)$ and at a single point $a(u_0).
$
Moreover, \eqref{bupprof1} is satisfied by $u(x-a(u_0),t)$ (with $T$ replaced by $T(u_0)$ and for all $t\in [T(u_0) -e^{-s_{n_0}}, T(u_0))$ for some $n_0\in \N$, not just for a sequence). We also have
$$T(u_0)\to \hat{T},\; a(u_0)\to \hat a,\; \mbox{ as }\; u_0\to \hat{u_0}\; \mbox{ in } W^{1,\infty}(\R^N).$$
}
\begin{rem} {\rm As a consequence of this theorem, we see that $\hat u$ has the profile \eqref{145} not only for a sequence, but for the whole time range $[T(u_0) -e^{-s_{n_0}}, T(u_0))$ (in particular, \eqref{48} holds also for all $s$ large, and not just for a sequence) . In other words, having the profile \eqref{145} only for a sequence (together with some estimates on the gradient) is equivalent to having that profile for the whole range of times.}
\end{rem}

\appendix

\section{A technical result related to the projection \eqref{qprojection}}\label{appproj}

In this section, we prove Lemma \ref{clproj}.

\begin{proof}[Proof of Lemma \ref{clproj}]
Consider $\sigma_0$ to be taken large enough. We assume that $(\tau, \alpha)$ satisfies \eqref{condta}.\\
(i)-(ii) Consider $m=0,1,2$. From \eqref{qprojection}, we see that
\[
\bar g_m(s_0)= \int_{\R}\bar g(y,s_0)k_m(y) \chi(y,s_0)\rho(y) dy,
\]
where $k_m(y)$ and $\chi(y,s_0)$ are given in \eqref{defkm} and \eqref{def:chi}.
Using \eqref{defhb} and making the change of variables $y\to z$, we write
\begin{equation}\label{defbgm}
\bar g_m(s_0)= (1+\tau)^{\frac 1{p-1} - \frac 12}\int_{\R} g(z,\sigma_0)k_m(y)
 \chi(y,s_0)\rho(y) dz.
\end{equation}
First, note that one easily checks that
\begin{align}
|s_0^{2\beta+1}-\sigma_0^{2\beta+1}|&\le C|\tau|\sigma_0^{2\beta},\nonumber\\
|\chi(y,s_0)-\chi(z,\sigma_0)|&\le \frac C{\sigma_0^\beta}\left(|z\tau|+|\alpha|\right)1_{\left\{|z|\ge \frac K2 \sigma_0^\beta\right\}},\label{dlchi}\\
\left|\rho(y) - \rho(z)\left[1-\frac{\alpha^2}4+\frac \alpha 2 z +\frac {2\tau+\alpha^2} 8 z^2\right]\right|&\le C\left(\tau^2+|\tau\alpha|+|\alpha|^3\right)(\rho(z))^{\frac 45},\label{dlro}
\end{align}
whenever $(\tau, \alpha)$ satisfies \eqref{condta}. Then, we give the following conversion table  for the polynomials involved in the expression of $\bar g_m$:
\begin{align*}
&(1+\tau)^{\frac 1{p-1}-\frac 12}k_m(y) \left[1-\frac{\alpha^2}4+\frac \alpha 2 z +\frac {2\tau+\alpha^2}8 z^2\right]\\
=&\left(1+\frac \tau{p-1}\right)k_0(z)+\alpha k_1(z)+(2\tau+\alpha^2) k_2(z)+O((\tau^2+|\tau\alpha|)(1+|z|^2))\mbox{ if }m=0,\\
=&\left(1+\frac {(p+1)\tau}{2(p-1)}\right)k_1(z)+2\alpha k_2(z)+3 (2\tau+\alpha^2) k_3(z) +O((\tau^2+|\tau\alpha|+|\alpha|^3)(1+|z|^3))\mbox{ if }m=1.
\end{align*}
Recalling that
\begin{equation}\label{defgj}
g_j(\sigma_0)=\int_{\R} g(z)k_j(z) \chi(z,\sigma_0)\rho(z) dz,
\end{equation}
the result follows for items (i) and (ii).\\
(iii) Now, we take $m=2$. The proof follows the same pattern as for the previous items, though we need less accuracy in $\tau$. Indeed, we need this less precise version of \eqref{dlro}:
\begin{equation*}
\left|\rho(y) - \rho(z)\left[1-\frac{\alpha^2}4+\frac \alpha 2 z +\frac {\alpha^2} 8 z^2\right]\right|\le C\left(|\tau|+|\alpha|^3\right)(\rho(z))^{\frac 34}.
\end{equation*}
As before, we need the following expansion:
\begin{equation*}
(1+\tau)^{\frac 1{p-1}-\frac 12}k_2(y) \left[1-\frac{\alpha^2}4+\frac \alpha 2 z +\frac {\alpha^2}8 z^2\right]
=k_2(z)+3\alpha k_3(z)+6\alpha^2k_4(z)+O(|\tau|+|\alpha|^3)(1+z^4)).
\end{equation*}
Using \eqref{dlchi} and \eqref{defgj} yields item (ii).\\
(iv)-(v) Take $m=0$ or $1$.
From \eqref{defhb} and \eqref{defbgm}, we write
\begin{align*}
&s_0^{-2\beta-1}\partial_\tau[s_0^{2\beta+1} \bar  g_m](s_0) = \left(\frac 1{p-1} - \frac 12\right) \frac{\bar g_m(s_0)}{1+\tau}-\frac{(2\beta+1)}{s_0(1+\tau)}\bar g_m(s_0)\\
& +(1+\tau)^{\frac 1{p-1} - \frac 12}\int_{\R} g(z,\sigma_0)\partial_\tau y\left(\frac 12 k_{m-1}'(y)
 \chi(y,s_0)+k_m(y) \partial_y \chi(y,s_0)-\frac y2k_m(y) \chi(y,s_0)\right)\rho(y) dz\\
&+(1+\tau)^{\frac 1{p-1}-\frac 12}\int_{\R} g(z,\sigma_0) k_m(y) \partial_\tau s_0\partial_s \chi(y,s_0) \rho(y) dz,
\end{align*}
with the convention that $k_{-1}(y) =0$.
Noting from \eqref{defhb} and the definition \eqref{def:chi} of $\chi$ that
\begin{align}
y=\frac{z-\alpha}{\sqrt{1+\tau}},\;
\partial_\tau y = -\frac{z-\alpha}{2(1+\tau)^{\frac 32}},\;
\left|\partial_y \chi(y,s_0)\right|\le \frac C{\sigma_0^\beta}1_{\left\{|z|\ge \frac K2 \sigma_0^\beta\right\}},\;\label{dty}\\
\partial_\tau s_0 = -\frac 1{1+\tau},\;\;
\left|\partial_s \chi(y,s_0)\right|\le \frac C{\sigma_0}1_{\left\{|z|\ge \frac K2 \sigma_0^\beta\right\}},\nonumber
\end{align}
then, using the various estimates presented for the proof of items (i) to (iii), we directly get the conclusions for items (iv) and (v).\\
(vi)-(vii) Take $m=0$ or $1$. From \eqref{defbgm}, we write
\begin{align*}
&\partial_\alpha \bar  g_m(s_0)\\
=&(1+\tau)^{\frac 1{p-1} - \frac 12}\int_{\R} g(z,\sigma_0)\partial_\alpha y\left(\frac 12 k_{m-1}'(y)
 \chi(y,s_0)+k_m(y) \partial_y \chi(y,s_0)-\frac y2k_m(y) \chi(y,s_0)\right)\rho(y) dz.
\end{align*}
Note that
\[
\partial_\alpha y = -\frac 1{\sqrt{1+\tau}}.
\]
Using this even less precise version of \eqref{dlro}:
\begin{equation*}
\left|\rho(y) - \rho(z)\left[1+\frac \alpha 2 z\right]\right|\le C\left(|\tau|+\alpha^2\right)(\rho(z))^{\frac 34},
\end{equation*}
together with \eqref{dlchi} and \eqref{dty},
we see that
\begin{align*}
\partial_\alpha \bar  g_0(\sigma_0) &= \int_\R g(z,\sigma_0)\frac{z-\alpha}2\left(1+\frac \alpha 2 z\right) \rho (z) dz
 +O((e^{-\sigma_0^{2\beta}}+|\tau |+\alpha^2)\|g(\sigma_0)\|_{L^\infty}),\\
\partial_\alpha \bar  g_1(\sigma_0) &= \int_\R g(z,\sigma_0)\left(\frac{(z-\alpha)^2}4-\frac 12\right)\left(1+\frac \alpha 2 z\right) \rho (z) dz
 +O((e^{-\sigma_0^{2\beta}}+|\tau |+\alpha^2)\|g(\sigma_0)\|_{L^\infty}).
\end{align*}
Since
\begin{align*}
\frac{z-\alpha}2\left(1+\frac \alpha 2 z\right)&=k_1(z)+2\alpha k_2(z)+O(\alpha^2(1+z^2)),\\
\left(\frac{(z-\alpha)^2}4-\frac 12\right)\left(1+\frac \alpha 2 z\right) &= 2 k_2(z) +6 \alpha k_3(z) +O(\alpha^2(1+|z|^3)),
\end{align*}
arguing as for the previous items gives the result.\\
(viii) From \eqref{qprojection}, we write
\begin{align*}
\bar g_-(y,s_0) &= \chi(y,s_0)\bar g(y,s_0)- \sum_{i=0}^2 \bar g_i(s_0) h_i(y),\\
g_-(z,\sigma_0) &= \chi(z,\sigma_0)g(z,\sigma_0)- \sum_{i=0}^2 g_i(\sigma_0) h_i(z).
\end{align*}
Making the difference and using the definition \eqref{defhb} of $\bar g$, we see that
\begin{align}
\bar g_-(y,s_0)-g_-(z,\sigma_0)& =g(z,\sigma_0)\left((1+\tau)^{\frac 1{p-1}-\frac 12}\chi(y,s_0)-\chi(z,\sigma_0)\right)\label{g-g-}\\
&- \sum_{i=0}^2 \left(\bar g_i(s_0)-g_i(\sigma_0))\right) h_i(y)+\sum_{i=0}^2 g_i(\sigma_0) \left(h_i(z)-h_i(y)\right).\nonumber
\end{align}
Since we have from \eqref{dlchi},
\begin{equation}\label{difchi}
\left|(1+\tau)^{\frac 1{p-1}-\frac 12}\chi(y,s_0)-\chi(z,\sigma_0)\right|
\le C|\tau|+ \frac C{\sigma_0^\beta}\left(|\tau y|+|\alpha|\right)1_{\{|y|\ge \frac K4 \sigma_0^{\beta}\}}
\le C\left(|\tau|+\frac{|\alpha|}{\sigma_0^{4\beta}}\right)(1+|y|^3),
\end{equation}
and from the definition \eqref{h_m} of $h_i$,
\[
h_0(y)-h_0(z)=0,\;|h_1(y)-h_1(z)|\le C(|\alpha|+|\tau y|),\;
|h_2(y)-h_2(z)|\le C(\alpha^2+|\alpha y|+|\tau| y^2),
\]
and from \eqref{defhb}
\begin{equation}\label{bzy}
1+|z|^3 \le 2(1+|y|^3),
\end{equation}
we see that the conclusion follows from items (i)-(iii) (use the fact that $|g_j(\sigma_0)|\le C\left\|\frac{g_-(z,\sigma_0)}{1+|z|^3}\right\|_{L^\infty(\R)}$ for $j=3,4$, which follows from \eqref{def:q_b}).\\
(ix) From \eqref{def:q:projbis}, we write
\[
\bar g_e(y,s_0)= (1-\chi(y,s_0))\bar g(y,s_0)
\mbox{ and }
g_e(z,\sigma_0)= (1-\chi(z,\sigma_0))g(z,\sigma_0),
\]
therefore, by definition \eqref{defhb} of $\bar g$, we have
\[
\bar g_e(y,s_0)= g_e(z,\sigma_0)+\left(\left[(1+\tau)^{\frac 1{p-1}-\frac 12}-1\right][1-\chi(y,s_0)]-[\chi(y,\sigma_0)-\chi(z,s_0)]\right)g(z,\sigma_0).
\]
Since we have from the definition \eqref{def:chi} of $\chi$ and the expression \eqref{defhb} of $y$,
\[
|\chi(y,s_0)-\chi(z,\sigma_0)|\le \frac C{\sigma_0^\beta}(|y \tau|+|\alpha|)1_{\{\frac K2\sigma_0^\beta\le |y|\le \frac {3K}2\sigma_0^\beta\}} \le
C(|\tau|+|\alpha|),
\]
the result follows.\\
(x) The result follows from the differentiation of \eqref{defhb}:
\[
\nabla \bar g(y,s_0)= (1+\tau)^{\frac 1{p-1}+\frac 12}\nabla g(z,\sigma_0)\mbox{ with }y=z\sqrt{1+\tau}+\alpha.
\]
(xi) Differentiating \eqref{g-g-}, we see that
\begin{align*}
&\nabla \bar g_-(y,s_0)-\sqrt{1+\tau}\nabla g_-(z,\sigma_0)
=\sqrt{1+\tau}\nabla g(z,\sigma_0)\left((1+\tau)^{\frac 1{p-1}-\frac 12}\chi(y,s_0)-\chi(z,\sigma_0)\right)\\
&+g(z,\sigma_0)\left((1+\tau)^{\frac 1{p-1}-\frac 12}\nabla\chi(y,s_0)-\sqrt{1+\tau}\nabla\chi(z,\sigma_0)\right)\\
&- \sum_{i=0}^2 \left(\bar g_i(s_0)-g_i(\sigma_0))\right)i h_{i-1}(y)+\sum_{i=0}^2 g_i(\sigma_0)i \left(\sqrt{1+\tau}h_{i-1}(z)-h_{i-1}(y)\right).\nonumber
\end{align*}
Since
\begin{align*}
\left|(1+\tau)^{\frac 1{p-1}-\frac 12}\nabla\chi(y,s_0)-\sqrt{1+\tau}\nabla\chi(z,\sigma_0)\right|&\le C(\frac{|\tau|}{\sigma_0^\beta}+\frac C{\sigma_0^{2\beta}}(|\tau y|+|\alpha|)1_{\{\frac K2\le \frac{|y|}{\sigma_0^\beta}\le \frac{3K}2\}}\\
&\le \frac C{\sigma_0^{4\beta}}(|\tau|+\frac{|\alpha|}{\sigma_0^\beta})(1+|y|^3)
\end{align*}
by definition \eqref{def:chi} of $\chi$, using the fact that
\[
\left|\sqrt{1+\tau}h_0(z)-h_0(y)\right|\le C|\tau|\mbox{ and }
 \left|\sqrt{1+\tau}h_1(z)-h_1(y)\right|\le C\left(|\tau y|+|\alpha|\right),
\]
we derive the conclusion from
\eqref{bzy},
\eqref{difchi} together with items (i)-(iii).
\end{proof}

{\bf{Acknowledgements.}} Slim Tayachi would like to thank the Fondation Sciences Math\'ematiques de Paris and Laboratoire \'Equations aux D\'eriv\'ees Partielles LR04ES03 of University  Tunis El Manar for their financial support.
Part of this work was done when S.T. was visiting the Laboratoire Analyse G\'eom\'etrie et Applications (LAGA) of University Paris 13, in particular, on a Visiting Professor position in 2015. He is grateful to LAGA for the hospitality and the stimulating atmosphere.

\appendix

\end{document}